\documentclass[11pt, a4paper]{amsart}

\usepackage{a4wide}
\usepackage{amssymb}
\usepackage{amsmath}
\usepackage{amsthm}
\usepackage{amstext}
\usepackage{amscd}
\usepackage{latexsym}
\usepackage{graphics}
\usepackage{color}
\usepackage{enumitem}
\usepackage{stmaryrd}
\usepackage{bm}
\usepackage{longtable}
\usepackage[all]{xy}
\usepackage[colorlinks, pagebackref]{hyperref}
\hypersetup{
  colorlinks=true,
  citecolor=blue,
  linkcolor=blue,
  urlcolor=blue}

\makeatletter
\def\@tocline#1#2#3#4#5#6#7{\relax
  \ifnum #1>\c@tocdepth 
  \else
    \par \addpenalty\@secpenalty\addvspace{#2}%
    \begingroup \hyphenpenalty\@M
    \@ifempty{#4}{%
      \@tempdima\csname r@tocindent\number#1\endcsname\relax
    }{%
      \@tempdima#4\relax
    }%
    \parindent\z@ \leftskip#3\relax \advance\leftskip\@tempdima\relax
    \rightskip\@pnumwidth plus4em \parfillskip-\@pnumwidth
    #5\leavevmode\hskip-\@tempdima
      \ifcase #1
      \or\or \hskip 2em \or \hskip 2em \else \hskip 3em \fi%
      #6\nobreak\relax
    \dotfill\hbox to\@pnumwidth{\@tocpagenum{#7}}\par
    \nobreak
    \endgroup
  \fi}
\makeatother

\theoremstyle{plain}

\newtheorem{theoremintro}{Theorem}[]
\newtheorem{corollaryintro}[theoremintro]{Corollary}

\newtheorem{theorem}{Theorem}[section]

\newtheorem{lemma}[theorem]{Lemma}
\newtheorem{corollary}[theorem]{Corollary}
\newtheorem{proposition}[theorem]{Proposition}

\theoremstyle{definition}

\newtheorem{convention}[theorem]{Convention}
\newtheorem{notation}[theorem]{Notation}
\newtheorem{remark}[theorem]{Remark}

\newtheorem{definition}[theorem]{Definition}
\newtheorem{example}[theorem]{Example}

\newtheorem{remarkintro}[theoremintro]{Remark}

\newtheorem{conjintro}[theoremintro]{Conjecture}

\numberwithin{equation}{section}

\newcommand{\length}{{\rm length}}

\newcommand{\Ker}{{\rm Ker \ }}
\newcommand{\Coker}{{\rm Coker \ }}
\newcommand{\Pic}{{\rm Pic}}

\newcommand{\Hom}{{\rm Hom}}
\newcommand{\Sym}{{\rm Sym}}

\newcommand{\Ext}{{\rm Ext}}

\newcommand{\Spec}{{\rm Spec \,}}

\newcommand{\Id}{{\rm Id}}

\newcommand{\GW}{{\rm GW}}
\newcommand{\pr}{{\rm pr}}

\renewcommand{\tilde}{\widetilde}

\newcommand{\sA}{{\mathcal A}}

\newcommand{\sD}{{\mathcal D}}
\newcommand{\sE}{{\mathcal E}}
\newcommand{\sF}{{\mathcal F}}
\newcommand{\sG}{{\mathcal G}}
\newcommand{\sH}{{\mathcal H}}

\newcommand{\sM}{{\mathcal M}}

\newcommand{\sO}{{\mathcal O}}

\newcommand{\sV}{{\mathcal V}}

\newcommand{\sX}{{\mathcal X}}

\newcommand{\A}{{\mathbb A}}
\newcommand{\Ctcell}{\tilde{C}^{\rm cell}}    
\newcommand{\CScell}{{C}^{\rm Sus,cell}} 
\newcommand{\HScell}{{\mathbf H}^{\rm Sus,cell}} 

\newcommand{\C}{{\mathbb C}}

\newcommand{\G}{{\mathbb G}}

\newcommand{\N}{{\mathbb N}}

\renewcommand{\P}{{\mathbb P}}

\newcommand{\R}{{\mathbb R}}

\newcommand{\V}{{\mathbb V}}

\newcommand{\Z}{{\mathbb Z}}

\newcommand{\bM}{\mathbf{M}}
\newcommand{\bN}{\mathbf{N}}
\newcommand{\bH}{\mathbf{H}}
\newcommand{\Ga}{{\mathbb G}_a}
\newcommand{\GL}{\mathbf{GL}}        
\newcommand{\PGL}{\mathbf{PGL}}        
\newcommand{\SL}{\mathbf{SL}}        
\newcommand{\Sp}{\mathbf{Sp}}

\newcommand{\KM}{{\mathbf K}^{\rm M}} 		
\newcommand{\KMW}{{\mathbf K}^{\rm MW}}		
\newcommand{\Sing}{{\rm Sing_*^{\A^1}}} 	
\newcommand{\piA}{{\bm \pi}^{\A^1}}   		
\newcommand{\HA}{{\mathbf H}^{\A^1}}  		
\newcommand{\HAred}{\~{\mathbf H}^{\A^1}}  	
\newcommand{\Hcell}{{\mathbf H}^{\rm cell}} 
\newcommand{\HS}{{\mathbf H}^{\rm S}} 		
\newcommand{\ZA}{{\mathbf Z}_{\A^1}}
\newcommand{\Ccell}{{C}^{\rm cell}}			
\renewcommand{\Ctcell}{\tilde{C}^{\rm cell}}    
\renewcommand{\CScell}{{C}^{\rm Sus,cell}} 
\renewcommand{\HScell}{{\mathbf H}^{\rm Sus,cell}} 

\newcommand{\CExt}{{\rm CExt}}
\newcommand{\holim}{{\rm holim}}

\def\<{\langle}
\def\>{\rangle} 
\def\-{\overline} 
\def\~{\widetilde}

\input{xy}
\xyoption{all}

\begin{document}

\title{Cellular $\A^1$-homology and the motivic version of Matsumoto's theorem}

\author{Fabien Morel}
\address{Mathematisches Institut, Ludwig-Maximilians-Universit\"at M\"unchen, Theresienstr. 39, 
D-80333 M\"unchen, Germany.}
\email{morel@math.lmu.de}

\author{Anand Sawant}
\address{School of Mathematics, Tata Institute of Fundamental Research, Homi Bhabha Road, Colaba, Mumbai 400005, India.}
\email{asawant@math.tifr.res.in}
\date{}

\thanks{Anand Sawant acknowledges support of the Department of Atomic Energy, Government of India, under project no. 12-R\&D-TFR-5.01-0500.  He was supported by DFG Individual Research Grant SA3350/1-1 while he was affiliated to Ludwig-Maximilians-Universit\"at, M\"unchen, where a part of this work was carried out.\\
Anand Sawant and Fabien Morel are currently supported by the India DST-DFG Program on Motivic Algebraic Topology.}

\subjclass[2010]{14F42, 20G15 (Primary)}

\begin{abstract}
We define a new version of $\A^1$-homology, called cellular $\A^1$-homology, for smooth schemes over a field that admit an increasing filtration by open subschemes with cohomologically trivial closed strata. We provide several explicit computations of cellular $\A^1$-homology and use them to determine the $\A^1$-fundamental group of a split reductive group over an arbitrary field, thereby obtaining the motivic version of Matsumoto's theorem on universal central extensions of split, semisimple, simply connected algebraic groups. As applications, we uniformly explain and generalize results due to Brylinski-Deligne and Esnault-Kahn-Levine-Viehweg, determine the isomorphism classes of central extensions of such an algebraic group by an arbitrary strictly $\A^1$-invariant sheaf and also reprove classical results of E. Cartan on homotopy groups of complex Lie groups.
\end{abstract}

\maketitle

\tableofcontents

\setlength{\parskip}{4pt plus 2pt minus 2pt} 
\raggedbottom

\section{Introduction}
\label{section introduction}

The aim of this work is to determine the $\A^1$-fundamental sheaf of groups $\piA_1(G)$ of a split reductive group $G$ over a field $k$. Our computation of $\piA_1(G)$ relies on three facts: the combinatorics of the root datum of a split reductive group, the $\A^1$-homotopy purity theorem of \cite{Morel-Voevodsky} and the structure of $\A^1$-homotopy sheaves of connected spaces over a perfect field $k$ \cite{Morel-book}; the latter two being foundational results from $\A^1$-algebraic topology.  The main novelty of our approach is defining a new version of $\A^1$-homology, called \emph{cellular $\A^1$-homology}, for smooth schemes admitting a \emph{nice} stratification (more precisely, a \emph{cellular structure}).  The cellular $\A^1$-homology sheaves are obtained as the homology sheaves of the \emph{cellular $\A^1$-chain complex}, which is associated to the cellular structure.  The cellular $\A^1$-homology theory, inspired by cellular homology in classical topology, is independent of the cellular structure, and is very often entirely computable.  
Some desired but elusive properties of $\A^1$-homology can be seen to hold for cellular $\A^1$-homology; for instance, it follows immediately from definitions that cellular $\A^1$-homology of a scheme vanishes in degrees beyond the dimension of the scheme.

In this work, we will freely use the terminology introduced and used in \cite{Morel-Voevodsky} and \cite{Morel-book} (we refer the reader to the section on notation and conventions at the end of the introduction).  Throughout this paper, $k$ will denote a fixed base field (which will soon be assumed to be perfect) and $Sm_k$ will denote the category of smooth $k$-schemes, endowed with the Nisnevich topology.  Unless otherwise stated, all sheaves are meant to be sheaves for the Nisnevich topology on $Sm_k$.  Let $G$ be a split reductive group $G$ over $k$.  Let $G_{\rm der}$ denote the derived subgroup of $G$ with universal simply connected cover $G_{\rm sc} \to G_{\rm der}$ (in the sense of the theory of algebraic groups) with kernel $\mu$.  The algebraic group $\mu$ is known to be a finite group scheme of multiplicative type (it is thus a product of group schemes of the form $\mu_n$).  Let $\KM_n$ (respectively, $\KMW_n$) denote the $n$th unramified Milnor $K$-theory (respectively, Milnor-Witt $K$-theory) sheaf.  

\begin{theoremintro}
\label{theorem intro main}
For a split reductive group $G$ over a field $k$, we have an exact sequence 
\[
1 \to \piA_1(G_{\rm sc}) \to \piA_1(G) \to \mu \to 1
\]
of Nisnevich sheaves of abelian groups.  If $G$ is a split, semisimple, almost simple, simply connected algebraic group over a field $k$, then there exists an isomorphism of sheaves
\[
\piA_1(G) \simeq \begin{cases} \KM_2, & \text{ if $G$ is not of symplectic type;} \\ \KMW_2,&  \text{ if $G$ is of symplectic type.} \end{cases}
\]
\end{theoremintro}

The first assertion of Theorem \ref{theorem intro main} follows from standard arguments.  Indeed, one has a short exact sequence
\[
1 \to G_{\rm der} \to G \to {\rm corad}(G) \to 1,
\]
where ${\rm corad}(G)$ denotes the coradical torus of $G$.  Since ${\rm corad}(G)$ is $\A^1$-rigid, it follows that $\piA_1(G) \simeq \piA_1(G_{\rm der})$, reducing the computation of the $\A^1$-fundamental group of a split reductive group to its semisimple part.  The $\A^1$-fiber sequence
\[
G_{\rm sc} \to G_{\rm der} \to B_{gm} \mu,
\]
where $B_{gm} \mu$ is the geometric classifying space considered by Totaro \cite{Totaro} and also studied in \cite{Morel-Voevodsky}, gives rise to a short exact sequence of sheaves of groups of the form
\[
1 \to \piA_1(G_{\rm sc}) \to \piA_1(G_{\rm der}) \to \mu \to 1.
\]
Furthermore, observe that $G_{\rm sc}$ is a product of its almost simple factors.  As a consequence, the computation of the $\A^1$-fundamental group of a split reductive group reduces to the computation of the $\A^1$-fundamental group of a split, semisimple, almost simple, simply connected algebraic group.  

\begin{remarkintro}
\label{rem historical Matsumoto}
A classical theorem of Matsumoto \cite{Matsumoto} determines a presentation of the center of the universal central extension of the group of rational points of a split, semisimple, simply connected algebraic group over an infinite field $k$.  Let $G$ be a such a group over $k$.  Matsumoto's theorem can be rephrased as saying that the center of the universal central extension 
\begin{equation} \label{eqn UCE}
1 \to A \to E \to G(k) \to 1
\end{equation}
of $G(k)$ is given by 
\[
A = \begin{cases} \KM_2(k), &\text{ if $G$ is not of symplectic type;} \\ \KMW_2(k), &\text{ if $G$ is of symplectic type.} \end{cases}
\]
Note that Matsumoto only considers infinite base fields; indeed, $\KM_2$ and $\KMW_2$ of a finite field are trivial.  Moreover, note that for some small finite fields $F$, the group $G(F)$ is at times not perfect and hence, does not admit a universal central extension.  

It is natural to look for central extensions \eqref{eqn UCE} that are functorial in $k$ and in $G$.  Brylinski and Deligne \cite{Brylinski-Deligne} have shown the existence of a central extension of any reductive algebraic group $G$ by $\KM_2$ as Zariski sheaves and determined the category of all such extensions in terms of Weyl-invariant integer-valued quadratic forms on the coroot lattice associated with the semisimple part of $G$.  A standard fact from topology asserts that the fundamental group of a path-connected topological group is abelian, and its universal covering is in fact a topological central extension of it by its fundamental group. If the topological group is moreover perfect, this implies the existence of a canonical morphism from the universal central extension to the universal covering. 

In view of this heuristic, it is natural to search for analogues of these facts in $\A^1$-homotopy theory of smooth algebraic varieties.  More specifically, one should expect that the central extensions of a reductive group determined by Matsumoto and Brylinski-Deligne are closely related to its universal cover in the sense of $\A^1$-homotopy theory and its $\A^1$-fundamental group in the sense of \cite{Morel-Voevodsky}.  This is precisely what Theorem \ref{theorem intro main} achieves. We stress that we do not use the theorem of Matsumoto in any form; our result is entirely independent of it.
\end{remarkintro}
 
\begin{remarkintro}
\label{remark intro perfect field}
We use some of the results of \cite{Morel-book} in our arguments, which only hold when the base field is perfect. However, by Chevalley's existence theorem \cite[XXV, 1.2]{SGAIII}, any split reductive $k$-group is obtained (up to isomorphism) by extension of a split reductive group defined over $\Z$ and consequently, a split reductive group defined over the prime field $k_0$ of $k$, which is perfect.  Using standard facts on essentially smooth base change (see, for example, \cite[Appendix A]{Hoyois}), we are reduced to proving Theorem 1 in case the base field is perfect.  Therefore, we will always work over a perfect field throughout the article.
\end{remarkintro}
 
\begin{remarkintro}
If $G$ is a split, semisimple, simply connected group such that every component of the relative root system of $G$ has isotropic rank at least $2$, then the group $\piA_1(G)(k)$ can be directly computed (see \cite{Voelkel-Wendt}).  The proof in \cite{Voelkel-Wendt} uses the $\A^1$-locality of the singular construction $\Sing G$ \cite{AHW2},  Matsumoto's results \cite{Matsumoto} and some work by Petrov and Stavrova \cite{Petrov-Stavrova} (which requires the hypothesis on the isotropic rank of the root system).  However, these computations are unsatisfactory because they do not determine $\piA_1(G)$ as a sheaf and do not work if a component of the relative root system of $G$ has rank $1$.  Moreover, the proof of $\A^1$-locality of $\Sing G$ depends on several nontrivial results in algebraic geometry.  Our proof of Theorem \ref{theorem intro main} does not make any use of the results of Matsumoto and Brylinski-Deligne, as mentioned above.  It is worthwhile to mention that we never use any model of $G$ in $\A^1$-homotopy theory (such as $\Sing G$) in our proofs.  Our proof of Theorem \ref{theorem intro main} is much more elementary and direct; see Remark \ref{remark intro K2} below.
\end{remarkintro}

\begin{remarkintro}
\label{rem pi0}
Observe that the $\A^1$-fundamental group of an algebraic group $G$ only depends on the $\A^1$-connected component of the neutral element of $G$.  We recall the description of the sheaf $\piA_0(G)$ obtained in \cite[Appendix A]{Morel-FM}.  Let $G$ be a split, semisimple algebraic group over $k$ with universal simply connected cover $G_{\rm sc} \to G$.  Let $T$ denote a maximal $k$-split torus of $G$ and let $T_{\rm sc}$ be its inverse image in $G_{\rm sc}$.  It has been shown in \cite[Theorem A.2]{Morel-FM} that there is a canonical isomorphism of the form
\[
T/^{\rm Nis}T_{\rm sc} \cong \piA_0(G), 
\]
where the left hand-side is the quotient of $T$ by the image of the induced morphism $T_{\rm sc} \to T$ as a Nisnevich sheaf of abelian groups.  In particular, $G$ is $\A^1$-connected if and only if it is simply connected in the sense of algebraic groups theory. However, Theorem \ref{theorem intro main} shows that $G$ is never $\A^1$-simply connected! 
\end{remarkintro}


We now briefly outline the main ideas in our proof of the main theorem.  It is an elementary observation (analogous to the one in classical topology) that for any simplicial sheaf of groups $\sG$ on $Sm_k$, the sheaf of groups $\piA_1(\sG)$ is abelian. If $\sG$ is $\A^1$-connected, then the Hurewicz morphism $\piA_1(\sG) \to  \HA_1(\sG)$ is an isomorphism \cite[Theorem 6.35]{Morel-book}.  Consequently, determination of the $\A^1$-homology sheaf $\HA_1(G)$ is tantamount to the proof of Theorem \ref{theorem intro main}.  However, it turns out that the $\A^1$-homology sheaves are very hard to compute and even the basic computations of them are quite involved. 

We define \emph{cellular $\A^1$-homology} (denoted by $\Hcell_*$) in Section \ref{subsection cellular chain complex} for smooth schemes admitting a \emph{cellular structure}; that is, an increasing filtration by open subschemes with a condition on the closed complements that their cohomology with values in any strictly $\A^1$-invariant sheaf is trivial (see Sections \ref{subsection cellular structures and orientations} and \ref{subsection cellular chain complex} for precise definitions and more details).  The normal bundles of these closed strata are then automatically trivial. The first nontrivial cellular $\A^1$-homology sheaf of a \emph{cellular scheme} clearly agrees with its corresponding $\A^1$-homology sheaf.  Cellular $\A^1$-homology is then the homology of an explicit chain complex of \emph{strictly $\A^1$-invariant sheaves} on $Sm_k$ directly defined using the cellular structure.  However, these computations are delicate, since one has to keep track of the chosen orientations (or trivializations) of the normal bundles of the closed strata appearing in the cellular structure.  In the setting of our main theorem, the cellular structures are naturally provided by the Bruhat decomposition of a split semisimple group, which yields an explicit chain complex $\Ccell_*(G)$ whose homology sheaves are  $\Hcell_*(G)$ (see Section \ref{section Bruhat decomposition}).  Any split, semisimple, almost simple, simply connected algebraic group $G$ over $k$ is $\A^1$-connected; hence, we have isomorphisms
\[
\piA_1(G) \cong \HA_1(G) \cong \Hcell_1(G).
\]
The differentials in $\Ccell_*(G)$ are determined by the orientations of the appropriate Bruhat cells and certain entries of the associated Cartan matrix.  We describe in Section \ref{subsection complex in rank r} how a slightly weaker version of a \emph{pinning} of an algebraic group determines these orientations.  After orienting $\Ccell_*(G)$ in low degrees using a \emph{weak-pinning} of $G$ and a choice of the reduced expression of the longest word in the Weyl group of $G$, the differentials in degrees $\leq 2$ can be explicitly computed.  This is carried out in Section \ref{section differential}, which is the technical heart of our proof of Theorem \ref{theorem intro main}.  These delicate computations culminate in the proof of Theorem \ref{theorem intro main} in Section \ref{section applications} (see Theorem \ref{thm: main}) in the following precise form.

\begin{theoremintro}
\label{theorem intro precise}
Let $G$ be a split, semisimple, almost simple, simply connected $k$-group and fix a maximal torus $T\subset G$ and a Borel subgroup $B\subset G$ containing $T$.  Let $\alpha$ be a long root in the corresponding system of simple roots of $G$ and let $S_\alpha\subset G$ be the split, semisimple, almost simple, simply connected $k$-subgroup of rank $1$ generated by $U_\alpha$ and $U_{-\alpha}$.  The inclusion $S_\alpha\subset G$ induces an epimorphism $\piA_1(S_\alpha) \twoheadrightarrow \piA_1(G)$. Moreover:
\begin{enumerate}[label=$(\alph*)$]
\item This morphism induces a canonical isomorphism
\[
\KM_2 \xrightarrow{\simeq} \piA_1(G),
\]
if $G$ is not of symplectic type.

\item This morphism induces a non-canonical isomorphism
\[
\KMW_2 \xrightarrow{\simeq} \piA_1(G),
\]
if $G$ is of symplectic type.
\end{enumerate}
\end{theoremintro}

\begin{remarkintro}
\label{remark intro SL2}
Let $G$ be as in Theorem \ref{theorem intro precise} and of symplectic type.  The Dynkin diagram of $G$ contains a unique long root, so the $\alpha$ in the statement of Theorem \ref{theorem intro precise} is canonical. However, the isomorphism $\piA_1(G) \cong \KMW_2$ is not canonical.  For instance, in the case $G = \SL_2$, it can be verified that for any $a \in k^{\times}$, the automorphism (as schemes and not as groups) of $\SL_2$ corresponding to 
\[
\begin{pmatrix}
x & y \\ z & w
\end{pmatrix} \mapsto 
\begin{pmatrix}
x & ay \\ a^{-1}z & w
\end{pmatrix}
\]
induces the morphism $\langle a \rangle \in \mathbf{GW}(k) = \Hom_{Ab_k}(\KMW_2, \KMW_2)$ on $\piA_1(\SL_2)$.  This morphism is the identity map precisely when $a$ is a square in $k$. 
\end{remarkintro}

\begin{remarkintro} 
\label{remark intro K2}
In order to prove Theorem \ref{theorem intro precise}, we need very few properties of the sheaves $\KMW_2$ and $\KM_2$ (described for instance in \cite{Morel-book}) in the computations. The reader who is not familiar with these descriptions could consider $\KMW_2$ to be defined as 
\[
\KMW_2 := \piA_1(\SL_2).
\]
Note that $\piA_1(\SL_2) = \piA_1(\Sigma(\G_m^{\wedge 2}))$.  Thus, for any pair of units $(u,v)$ in any smooth $k$-algebra $A$, we have a corresponding symbol $(u)(v)\in \KMW_2$; see conventions and notation at the end of the introduction.  We then need the action of the units modulo squares on $\KMW_2 = \piA_1(\SL_2)$ (described in Remark \ref{remark intro SL2}), that is, the morphism $\mathbf{GW}(k) \to \Hom_{Ab_k}(\KMW_2, \KMW_2)$ (we do not need the fact that it is an isomorphism).  In other words, if $\alpha$, $u$ and $v$ are units defined in $A$, we may define the multiplication of $\alpha$ modulo squares, denoted by $\<\alpha\>$, with the symbol $(u)(v)$, which will be denoted by $\<\alpha\> \cdot (u)(v)$.  We finally need to use the morphism $\eta: \KMW_1\otimes_{\A^1}\KMW_2 \to \KMW_2$, which appears in the differentials of the cellular $\A^1$-chain complexes.  However, this morphism precisely appears as one of the differentials in the cellular $\A^1$-chain complex of $\SL_3$, as we will see later. The only additional property that we use in the computations, which also follows from a careful but simple analysis of the cellular chain complex of $\SL_3$, is that for any units $\alpha$, $u$ and $v$ as above, the action of $\<\alpha\>$ on the symbol $(u)(v)$ is given by
\[
\<\alpha\> \cdot (u)(v) = (u)(v) + \eta  (\alpha) (u)(v).
\]
Once this is observed, one may define $\KM_2$ to be the quotient  
\[
\KM_2 = \KMW_2/\eta = \piA_1(\SL_2)/\eta.
\]
These formal properties are sufficient for our proof of Theorem \ref{theorem intro precise}. The reader should notice that we do not explicitly use the Steinberg relation, though it is encoded in the sheaf $\KMW_2$.
\end{remarkintro}

\begin{remarkintro} 
Let $G$ be of rank $2$ and type $A_2$, that is, $G \cong \SL_3$.  Let $\alpha$ be one of the two simple roots in the root system of $G$. The inclusion $S_\alpha\subset G$ induces an epimorphism $\piA_1(S_\alpha) \twoheadrightarrow \piA_1(G)$, which in turn induces a canonical isomorphism 
\[
\KM_2 = \piA_1(S_\alpha)/\eta \xrightarrow{\simeq} \piA_1(G).
\]
The properties noted in Remark \ref{remark intro K2} imply that the above isomorphism $\KM_2\xrightarrow{\simeq} \piA_1(G)$ is canonical and does not depend on any choices.  In the case $G$ is of type $G_2$, Theorem \ref{theorem intro precise} follows from a direct computation (see the proof of Theorem \ref{thm: main}).  If $G$ is of rank $>2$ and is not of symplectic type, one may explicitly describe part $(a)$ of Theorem \ref{theorem intro precise} as follows.  Note that we can always find a pair of long roots $\alpha$ and $\beta$ that are connected by an edge in the Dynkin diagram of $G$.  Then the semisimple subgroup-scheme $S_{\alpha,\beta}\subset G$ is simply connected and of type $A_2$ and the inclusion $S_{\alpha,\beta}\subset G$ induces a canonical isomorphism $\KM_2 = \piA_1(S_{\alpha,\beta})\xrightarrow{\simeq} \piA_1(G)$.  Surjectivity of this morphism follows from by observing that the composite $\piA_1(S_\alpha) \to \piA_1(S_{\alpha,\beta}) \to \piA_1(G)$ induced by the natural inclusions is an epimorphism.  Injectivity of this epimorphism follows from the observation that adding an edge connecting $\alpha$ or $\beta$ to another root (long or short) does not change the morphism (see the proof of Theorem \ref{thm: main}; see also Remark \ref{rem main precise}).
\end{remarkintro}

We now briefly discuss the case of non-split, semisimple  algebraic groups.  As mentioned in Remark \ref{rem pi0}, the sheaf $\pi_1^{\A^1}(G)$ depends upon the $\A^1$-connected component of the neutral element of $G$.  Note that at least over a field of characteristic $0$, a reductive algebraic group $G$ is $\A^1$-connected if and only if $G$ is semisimple, simply connected and every almost $k$-simple factor of $G$ is $R$-trivial \cite[Theorem 5.2]{Balwe-Sawant-reductive}.  Using this criterion, one can see that over special classes of fields (such as local fields and global fields) isotropic, semisimple, simply connected groups are $\A^1$-connected (see \cite{Gille} for a survey of known positive results and a counterexample in general due to Platonov).  It seems reasonable to conjecture the following regarding the $\A^1$-fundamental group of an $\A^1$-connected isotropic group.

\begin{conjintro}
\label{conj isotropic}
Let $G$ be an $\A^1$-connected, isotropic, semisimple, almost simple, simply connected algebraic group over $k$, split by a finite separable field extension $F/k$.  
\begin{enumerate}
\item If $G$ is not of symplectic type, then the universal $\KM_2$-torsor on $G$ constructed by Brylinski-Deligne \cite{Brylinski-Deligne} gives the universal covering of $G$ in the sense of $\A^1$-homotopy theory.

\item If $G$ is of symplectic type, then there exists a $\KMW_2$-torsor on $G$ giving the $\A^1$-universal covering of $G$.

\item More precisely, suppose that the relative root system of $G$ admits a coroot $\alpha$, which becomes a coroot of small length in the root system of $G_F$.  Then the inclusion $S_\alpha\subset G$ induces an isomorphism of $\KMW_2$ (respectively, $\KM_2$) with $\piA_1(G)$, if $G$ is of symplectic type (respectively, not of symplectic type).
\end{enumerate}
\end{conjintro}

\begin{remarkintro}
The study of central extensions of a (locally compact) topological group and its relationship with some version of fundamental group has a long history, the works of Steinberg \cite{Steinberg}, Moore \cite{Moore} and Matsumoto \cite{Matsumoto} being some of the major milestones.  One says that the \emph{topological fundamental group} $\pi_1(G)$ of a locally compact topological group $G$ exists if the endofunctor $\CExt(G, -)$ on the category of abelian groups associating with every abelian group $A$ the abelian group of topological central extensions of $G$ by $A$ under the operation of Baer sum is representable.  The group representing $\CExt(G, -)$ is then defined to be $\pi_1(G)$; see \cite{Moore} for a systematic study and more results.  Let $G$ be the group of $k$-rational points of a split, semisimple, almost simple, simply connected algebraic group over a locally compact nondiscrete field $k$ (for example, a local field).  Then $\pi_1(G)$ exists and the explicit description of the universal central extension of $G$ by generators and relations was given by Steinberg in \cite{Steinberg}.  In \cite[p. 194]{Moore}, Moore constructed the $2$-cocycle representing Steinberg's universal central extension.  Furthermore, with the notation of Theorem \ref{theorem intro precise}, it was shown in \cite[8.1-8.4]{Moore} that for a long root $\alpha$ in the Dynkin diagram of $G$, the inclusion $S_\alpha \subset G$ induces a surjective homomorphism 
\[
\pi_1(S_\alpha) \twoheadrightarrow \pi_1(G).
\]
Moore also showed that $\pi_1(G)$ is a quotient of $\KMW_2(k)$ or $\KM_2(k)$, depending upon whether $G$ is of symplectic type or otherwise.  These results were shortly improved by Matsumoto (see Remark \ref{rem historical Matsumoto}) by constructing certain explicit central extensions over an infinite base field.  Notice the striking analogy with Theorem \ref{theorem intro precise}, which shows that the $\A^1$-fundamental group $\piA_1(G)$ is indeed the correct analogue of $\pi_1(G)$ in the algebro-geometric/sheaf-theoretic realm.
\end{remarkintro}

We end the introduction by mentioning some direct applications of our main results.  An immediate consequence of Theorem \ref{theorem intro main} is the computation of the first cohomology group of $G$ with coefficients in a strictly $\A^1$-invariant sheaf (see Section \ref{subsection preliminaries} and conventions for precise definitions).  

\begin{corollaryintro}
\label{corintro:1}
For any split, semisimple, almost simple, simply connected algebraic group $G$ over $k$ and  a strictly $\A^1$-invariant sheaf of abelian groups $\mathbf{M}$ on $Sm_k$, we have:
\[
H^1_{\rm Zar}(G, \mathbf{M}) \simeq
\begin{cases}
~ \mathbf{M}_{-2}(k),  & \text{if $G$ is of symplectic type};\\
{}_\eta\mathbf{M}_{-2}(k), & \text{if $G$ is not of symplectic type},
\end{cases}
\]
where ${}_\eta\mathbf{M}_{-2}(k)$ denotes the $\eta$-torsion in the group $\mathbf{M}_{-2}(k)$.
\end{corollaryintro}

Indeed, since any $G$ as in Theorem \ref{theorem intro main} is $\A^1$-connected, for any strictly $\A^1$-invariant sheaf $\mathbf{M}$ on $Sm_k$, the universal coefficient formula gives an identification
\[
H^1_{\rm Nis}(G, \mathbf{M}) = \Hom_{Ab_{\A^1}(k)}(\piA_1(G), \mathbf{M})
\] 
and Corollary \ref{corintro:1} follows from the fact that $\Hom_{Ab_{\A^1}(k)}(\KMW_2, \mathbf{M}) = \mathbf{M}_{-2}(k)$ and that the change of topology morphism induces an isomorphism $H^1_{\rm Zar}(G, \mathbf{M}) \simeq H^1_{\rm Nis}(G, \mathbf{M})$ \cite[Corollary 5.43]{Morel-book}.  

\begin{remarkintro}
It has been shown in the work of Esnault-Kahn-Levine-Viehweg \cite[Proposition 3.20]{EKLV} and Brylinski-Deligne \cite[Proposition 4.6]{Brylinski-Deligne} that for a cycle module $\mathbf{M}_{\ast}$ in the sense of Rost \cite{Rost}, one has 
\[
H^1_{\rm Zar}(G, \mathbf{M}_{\ast}) = \mathbf{M}_{\ast - 2}(k).
\]
On the other hand, S. Gille has shown in \cite{Gille-Witt} if $k$ has characteristic different from $2$ and if $\mathbf{W}$ determines the unramified Witt sheaf of quadratic forms on $Sm_k$, then there is a dichotomy:
\[
H^1_{\rm Zar}(G, \mathbf{W}) =
\begin{cases}
\mathbf{W}(k),  & \text{if $G$ is of symplectic type};\\
0, & \text{if $G$ is not of symplectic type}.
\end{cases}
\]
Corollary \ref{corintro:1} uniformly explains and generalizes these computations of Esnault-Kahn-Levine-Viehweg, Brylinski-Deligne and S. Gille.
\end{remarkintro}

\begin{remarkintro}
Let $G$ be split, semisimple, almost simple, simply connected over $k$ and let $G_{\rm ad}$ denote the adjoint group of $G$, that is, the quotient of $G$ by its center.  The group $G_{\rm ad}$ acts on $G$ by conjugation, thereby giving an action of $G_{\rm ad}$ on $\piA_1(G)$.  If $G$ is not of symplectic type, this action is trivial.  In case $G$ is of symplectic type, this action is nontrivial and it can be shown that one has a canonical isomorphism
\[
\piA_1(G)/G_{\rm ad} \cong \KM_2
\]
of Nisnevich sheaves of abelian groups. In fact, analyzing the method of our proof applied through appropriate realization functors, we obtain a simple proof of the following result by S. Gille \cite{Gille-Suslin-homology}.

\begin{corollaryintro}
\label{corintro:2}
For any split, semisimple, almost simple, simply connected algebraic group $G$ over $k$, there is a canonical isomorphism
\[
\KM_2 \cong \HS_1(G) 
\]
of Nisnevich sheaves of abelian groups with transfers in the sense of Voevodsky, where $\HS_1$ denotes the first Suslin homology sheaf of quasi-projective schemes over $k$.
\end{corollaryintro}
\end{remarkintro}

\begin{remarkintro}
Using our method relying on the cellular structure given by the Bruhat decomposition, it is also possible to deduce (or reprove) the following two classical results from topology (see Section \ref{sec:method+Cartan}):
\begin{enumerate}
\item For any split, semisimple, simply connected algebraic group $G$ over $\R$, the topological space $G(\R)$ is path-connected.  Moreover, if $G$ is almost simple and not of symplectic type, there is a canonical isomorphism
\[
\pi_1(G) = \Z/2\Z.
\]
If $G$ is of symplectic type, there is a non-canonical isomorphism
\[
\pi_1(G) = \Z.
\]
The dichotomy in our main theorem is somehow explained by this fact.  Over the reals, $\KMW_2$ ``becomes'' a free abelian group of rank one (think about $\pi_1(\SL_2(\R))$), and $\eta$ becomes the multiplication by $2$. Now, the quotient of a free abelian group of rank one by $2$ is canonically isomorphic to $\Z/2\Z$.

\item (A theorem of E. Cartan \cite{Cartan}) For any split, semisimple, almost simple, algebraic group $G$ over $\C$, the topological space $G(\C)$ is $1$-connected if and only if $G$ is simply connected in the sense of algebraic groups.  Moreover, if $G$ is simply connected in the sense of algebraic groups, then $\pi_2(G(\C))=0$ and $\pi_3(G(\C))$ is a free abelian group of rank one.\footnote{The statement on $\pi_3(G(\C))$ was only observed without proof by Cartan, see \cite[page 496]{Dieudonne}; see \cite{Bott} for a proof using Morse theory.} 
\end{enumerate}
\end{remarkintro}

Another proof of Theorem \ref{theorem intro main} (and Theorem \ref{theorem intro precise}) can be obtained by using the five-term exact sequence of low-degree terms associated with the Serre spectral sequence in cellular $\A^1$-homology for the $\A^1$-fibration $G \to G/T \to BT$, where $T$ is a maximal $k$-split torus of $G$ and by explicitly computing the cellular $\A^1$-homology of $G/T$ (which is $\A^1$-weak equivalent to $G/B$, where $B$ is a Borel subgroup of $G$).  The details regarding this approach and the determination of cellular $\A^1$-homology of the generalized flag variety $G/B$ will be taken up in the sequel.  See Remarks \ref{rmk alternate proof 1}--\ref{rmk alternate proof 4} for more information.


\subsubsection*{\bf Conventions and notation} \hfill

We will freely use the standard notation in $\A^1$-homotopy theory developed in \cite{Morel-Voevodsky} and \cite{Morel-book}.  For the sake of convenience, we fix a perfect base field $k$.  The results about $\A^1$-homotopy and $\A^1$-homology sheaves we use all extend to the case of arbitrary base fields by standard results on essentially smooth base change (see \cite[Appendix A]{Hoyois}).  

We will denote by $Sm_k$ the big Nisnevich site of smooth, finite-type schemes over $k$.  We will denote the category of simplicial Nisnevich sheaves of sets over $Sm_k$ (called \emph{spaces}) by $\Delta^{\rm op}Shv_{\rm Nis}(Sm_k)$.  A morphism $\mathcal X \to \mathcal Y$ of simplicial sheaves of sets on $Sm/k$ is a \emph{Nisnevich local weak equivalence} if it induces an isomorphism on every stalk for the Nisnevich topology.  The category $\Delta^{\rm op}Shv_{\rm Nis}(Sm_k)$ admits a proper closed model structure, called the \emph{Nisnevich local injective model structure}, in which cofibrations are monomorphisms and weak equivalences are local weak equivalences.  The associated homotopy category is called the \emph{simplicial homotopy category} and is denoted by $\mathcal H_s(k)$.  The left Bousfield localization of the Nisnevich local injective model structure with respect to the collection of all projection morphisms $\mathcal X \times \mathbb A^1 \to \mathcal X$, as $\mathcal X$ runs over all simplicial sheaves, is called the \emph{$\mathbb A^1$-model structure}. The associated homotopy category is called the \emph{$\mathbb A^1$-homotopy category} and is denoted by $\mathcal H(k)$.  There is an obvious pointed analogue of this construction starting with the category whose objects are $(\sX, x)$ where $x: \Spec k \to \sX$ is a base-point, which gives rise to the \emph{pointed $\A^1$-homotopy category} $\mathcal H_{\bullet}(k)$.

For a field $F$, we will denote its Milnor-Witt $K$-theory by $\KMW_*(F) = \underset{n \in \Z}{\oplus} \KMW_n(F)$.  It is the associative graded ring generated by a symbol $\eta$ of degree $-1$ and symbols $(u)$ of degree $1$ for each $u \in F^{\times}$ with the relations $(u)(1-u)=0$; $(uv) = (u) + (v) + \eta(u)(v)$; $\eta(u) = (u)\eta$ and $\eta h=0$, where $h = 2+ \eta(-1)$, for all $u, v \in F^{\times}$.  The symbol $\<u\>$ will denote the element $1 + \eta(u) \in \KMW_0(F) = \GW(F)$.  We will denote by $\KM_n$ the $n$th unramified Milnor $K$-theory sheaf and by $\KMW_n$ the $n$th unramified Milnor-Witt $K$-theory sheaf; see \cite[Chapter 3]{Morel-book} for a detailed exposition.  Beware that we use a slightly different notation for Milnor-Witt $K$-theory than \cite{Morel-book}, the reasons for which will become clear later (see Convention \ref{convention MW} for the choice of this notation for mostly computational reasons).  The sheaves $\KM_n$ and $\KMW_n$ are strictly $\A^1$-invariant sheaves in sense of \cite[Chapter 2]{Morel-book} over any field. Recall that for $n\geq 1$, $\KMW_n$ is the free strictly $\A^1$-invariant sheaf on $\G_m^{\wedge n}$ in the sense of \cite[Theorem 3.37]{Morel-book}, and $\KM_n$ the quotient of $\KMW_n$ by $\eta$. This can in fact be considered here as the definition of these sheaves.  We will freely use the results from \cite{Morel-book} about the category $Ab_{\A^1}(k)$ of strictly $\A^1$-invariant sheaves; particularly, the fact that it is an abelian category.  We will also repeatedly use the fact that to verify that a morphism is an isomorphism, or that two morphisms are equal in $Ab_{\A^1}(k)$, it suffices to verify the same on sections over finitely generated separable field extensions of the base field $k$. 

\section{\texorpdfstring{$\A^1$}{A1}-homology theories}
\label{section cellular}

We begin with a brief recollection of the construction and basic properties of the $\A^1$-derived category and $\A^1$-homology sheaves developed in \cite{Morel-connectivity} and \cite{Morel-book}.  These are essential for the formulation of our main results and the techniques used in the proofs.  
We will then describe the construction of the cellular complexes and cellular $\A^1$-homology sheaves.
The results of this section will be applied to the case of cellular complexes associated with the Bruhat decomposition of a split, semisimple, simply connected group and its flag variety in Sections \ref{section Bruhat decomposition}, \ref{section differential} and \ref{section applications}.

\subsection{Preliminaries on \texorpdfstring{$\A^1$}{A1}-homotopy and \texorpdfstring{$\A^1$}{A1}-homology sheaves} \hfill 
\label{subsection preliminaries}

\begin{definition}
\label{definition A1-homotopy sheaves} \cite{Morel-Voevodsky}
For any pointed space $(\sX, x)$, the $n$th $\A^1$-homotopy sheaf of $(\sX,x)$, denoted by $\piA_n(\sX,x)$, is defined to be the Nisnevich sheaf associated with the presheaf $U \mapsto \Hom_{\sH_{\bullet}(k)}(S^n \wedge U_+, (\sX,x))$.
\end{definition}

We will suppress the base-point from the notation whenever it is clear from the context.  For any pointed space $\sX$, the sheaf $\piA_0(\sX)$ is a sheaf of (pointed) sets, and $\piA_n(\sX)$ is a sheaf of groups for $n \geq 1$, and $\piA_n(\sX)$ is a sheaf of abelian groups for $n \geq 2$.  

\begin{definition}
\label{definition A1-invariance}
We say that a Nisnevich sheaf $\sF$ of abelian groups on $Sm_k$ is \emph{$\A^1$-invariant} if the projection map $U \times \A^1 \to U$ induces a bijection 
\[
\sF(U) \to \sF(U \times \A^1),
\]
for every $U \in Sm_k$.  We say that $\sF$ is \emph{strictly $\A^1$-invariant} if for every integer $i \geq 0$, the projection map $U \times \A^1 \to U$ induces a bijection 
\[
H^i_{\rm Nis}(U, \sF) \to H^i_{\rm Nis}(U \times \A^1, \sF),
\]
for every $U \in Sm_k$. 
\end{definition}

One of the main results of \cite{Morel-book} is that for any pointed space $\sX$ (over a perfect base field), the sheaf $\piA_n(\sX)$ is strictly $\A^1$-invariant for $n\geq 2$.  The same is true for $\piA_1(\sX)$ if it happens to be a sheaf of abelian groups.  Observe that this is the case if $\sX$ is a group in the pointed $\A^1$-homotopy category. We refer the reader to \cite{Morel-book} for more details on the $\A^1$-homotopy sheaves.

\begin{notation}
We will denote by $Ab(k)$ the abelian category of Nisnevich sheaves of abelian groups on $Sm_k$.  We will denote the category of strictly $\A^1$-invariant sheaves on $Sm_k$ by $Ab_{\A^1}(k)$.
\end{notation}

The category $Ab_{\A^1}(k)$ of strictly $\A^1$-invariant sheaves on $Sm_k$ happens to be an abelian category.  This assertion is a consequence of the fact that the category of strictly $\A^1$-invariant sheaves on $Sm_k$ can be identified as the heart of the homological $t$-structure on the $\A^1$-derived category \cite[Lemma 6.2.11]{Morel-connectivity}.  The category $Ab_{\A^1}(k)$ has a symmetric monoidal structure \cite[Lemma 6.2.13]{Morel-connectivity}, which we will denote by $\otimes$.  This symmetric monoidal structure is compatible with the symmetric monoidal struture on the stable motivic homotopy category over $k$ given by the smash product; see \cite[Remark 6.2.16]{Morel-connectivity}.

We will use homological conventions while working with complexes of Nisnevich sheaves of abelian groups on $Sm_k$.  Let $Ch_{\geq 0}(Ab(k))$ denote the category of chain complexes $C_*$ of objects in $Ab(k)$ (with differentials of degree $-1$) such that $C_n = 0$, for all $n<0$.  Recall that the normalized chain complex functor
\[
C_*: \Delta^{\rm op} Shv_{\rm Nis}(Sm_k) \to Ch_{\geq 0} (Ab(k)) 
\]
admits a right adjoint called the Eilenberg-MacLane functor
\[
K : Ch_{\geq 0} (Ab(k)) \to \Delta^{\rm op} Shv_{\rm Nis}(Sm_k).
\]
For every $\bM \in Ab(k)$, we set $K(\bM, n) := K(\bM[n])$.  It is a fact that the normalized chain complex functor induces a functor
\[
C_*: \sH_s(k) \to D(Ab(k)) 
\]
which on $\A^1$-localization yields a functor
\[
C_*^{\A^1}: \sH(k) \to D_{\A^1}(k),
\]
where $D_{\A^1}(k)$ is the full subcategory of $D(Ab(k))$ consisting of $\A^1$-local complexes.  The inclusion of $D_{\A^1}(k)$ into $D(Ab(k))$ admits a left adjoint, called the $\A^1$-localization functor
\[
L_{\A^1} : D(Ab(k)) \to D_{\A^1}(k).
\]
Thus for a space $\sX\in \Delta^{\rm op} Shv_{\rm Nis}(Sm_k)$, one has $C_*^{\A^1}(\sX) = L_{\A^1} (C_*(\sX))$. If $\sX$ is a pointed space, the reduced chain complex $\tilde{C}_*(\sX)$ is defined to be the kernel of the canonical morphism  $C_*(\sX)\to \Z$, which is in fact a direct summand of $C_*(\sX)$.  Indeed, the inclusion of the point induces a canonical isomorphism $C_*(\sX)= \Z \oplus\tilde{C}_*(\sX)$. In the same way, the reduced $\A^1$-chain complex of $\sX$ is $\tilde{C}_*^{\A^1}(\sX) = L_{\A^1} (\tilde{C}_*(\sX))$ and we continue to have a canonical isomorphism $C^{\A^1}_*(\sX)= \Z \oplus\tilde{C}^{\A^1}_*(\sX)$.

\begin{definition}
\label{definition A1-homology}
For any $\sX \in \Delta^{\rm op} Shv_{\rm Nis}(Sm_k)$ and $n \in \Z$, the \emph{$n$th $\A^1$-homology sheaf} of $\sX$ is defined to be the $n$th homology sheaf of the $\A^1$-chain complex $C_*^{\A^1}(\sX)$.  If $\sX$ is pointed, then we define the \emph{$n$th reduced $\A^1$-homology sheaf} of $\sX$ by $\HAred_n(\sX) := \HA_n(\tilde{C}^{\A^1}_*(\sX))$.
\end{definition}

Since $\HA_0(\Spec k) \simeq \Z$ and $\HA_n(\Spec k) = 0$, for $n \neq 0$, we have an isomorphism 
\[
\HA_*(\sX) \simeq \Z \oplus \HAred_*(\sX) 
\]
of graded abelian sheaves.  Since the $\A^1$-localization functor commutes with the simplicial suspension functor in $D(Ab(k))$, for any pointed space $\sX$ and any $n \in \Z$, we have a canonical isomorphism
\[
\HAred_n(\sX) \simeq \HAred_{n+1}(S^1\wedge\sX).
\]
As a consequence of the $\A^1$-connectivity theorem \cite[Theorem 6.22]{Morel-book}, one knows that for every space $\sX$ and every integer $n$, the $\A^1$-homology sheaves $\HA_n(\sX)$ are strictly $\A^1$-invariant sheaves that vanish if $n < 0$.

\begin{notation}
For any Nisnevich sheaf of sets $\sF$ on $Sm_k$, we write
\[
\ZA[\sF] : = \HA_0(\sF). 
\]
For a pointed Nisnevich sheaf of sets $\sF$ on $Sm_k$, we write
\[
\ZA(\sF) : = \HAred_0(\sF).  
\]
\end{notation}

The sheaf $\ZA[\sF]$ is the \emph{free strictly $\A^1$-invariant sheaf on $\sF$} in the sense that we have a canonical bijection
\[
\Hom_{Ab(k)}(\ZA[\sF], \bM) \xrightarrow{\sim}  \Hom_{Shv_{\rm Nis}(Sm_k)}(\sF, \bM)
\]
for every $\bM \in Ab(k)$. The obvious variant holds for the free \emph{free strictly $\A^1$-invariant sheaf on the pointed sheaf of sets $\sF$}.

\begin{remark} 
With the above notations, one has the following result (see \cite{Morel-book}): the canonical morphism
$\Z((\G_m)^{\wedge n})\to \KMW_n$ induces an isomorphism
\[
\ZA((\G_m)^{\wedge n})\cong \KMW_n.
\]
As already noted above, the reader who is not acquainted with the sheaves $\KMW_n$ may consider this as a definition of the sheaf $\KMW_n$.
\end{remark}

If $\bM$ is strictly $\A^1$-invariant, then $K(\bM, n)$ is $\A^1$-local for every integer $n \geq 0$ and the Dold-Kan correspondence yields a bijection
\[
H^n_{\rm Nis}(\sX, \bM) = [\sX, K(\bM, n)]_s \xrightarrow{\sim} \Hom_{D(Ab(k))}(C_*^{\A^1}(\sX), \bM[n]).
\]
Moreover, if $\sX$ is $\A^1$-$n$-connected, then by the strong $\A^1$-Hurewicz theorem \cite[Theorem 6.57]{Morel-book} along with the universal coefficient theorem we have $\HA_i(\sX) = 0$ for $i \leq n$ and
\[
H^{n+1}_{\rm Nis}(\sX, \bM) \simeq \Hom_{Ab_{\A^1}(k)}(\HA_{n+1}(\sX), \bM) \simeq \Hom_{Ab_{\A^1}(k)}(\piA_{n+1}(\sX), \bM). 
\]

\subsection{Cellular structures and orientations} \hfill 
\label{subsection cellular structures and orientations}


\begin{definition}
\label{definition strict cellular}
Let $k$ be a field and let $X \in Sm_k$ be an irreducible smooth $k$-scheme of Krull dimension $n$. A \emph{strict cellular structure} on $X$ consists of an increasing filtration 
\[
\emptyset = \Omega_{-1} \subset \Omega_0(X) \subset \Omega_1(X) \subset \cdots \subset \Omega_n(X) = X 
\]
by open subschemes such that for each $i\in\{0,\dots,n\}$ the reduced induced closed subscheme $X_i:=\Omega_i(X) - \Omega_{i-1}(X)$ of $\Omega_i(X)$ is $k$-smooth and each of its irreducible components is isomorphic to $\A^{n-i}$.  We say that $X$ is a \emph{strictly cellular scheme} if $X$ is endowed with a strict cellular structure.  We will simply write $\Omega_i$ for $\Omega_i(X)$ whenever there is no confusion.
\end{definition}

\begin{example}
The most basic example of a strict cellular scheme is the projective space $\P^n$.  Consider the obvious increasing filtration of linear projective subspaces $\emptyset\subset \P^0\subset\dots\subset \P^i\subset\dots\subset\P^n$ corresponding to the canonical flag of subspaces of $\A^{n+1}$ and set $\Omega_i = \P^n-\P^{n-i-1}$.  Since $\Omega_i - \Omega_{i-1} = \P^{n-i} - \P^{n-i-1} \simeq \A^{n-i}$, the sequence of open subschemes $\emptyset = \Omega_{-1} \subset \Omega_0 \subset \Omega_1 \subset \cdots \subset \Omega_n =\P^n$ gives a strict cellular structure on $\P^n$.  
\end{example}

Another classical example of (strict) cellular structure arises from the Bruhat decomposition on the flag variety $G/B$ of a split semisimple algebraic group $G$.  We now generalize this notion of strict cellular schemes by allowing more general schemes as ``cells'' in order to get a ``cellular structure'' on $G$ itself.   

\begin{definition}
\label{definition cohomologially trivial}
We say that $X \in Sm_k$ is \emph{cohomologically trivial} if $H^n_{\rm Nis}(X, \bM) = 0$, for every $n \geq 1$ and for every strictly $\A^1$-invariant sheaf $\bM \in Ab_{\A^1}(k)$.
\end{definition}

\begin{remark} 
Clearly, $X$ is cohomologically trivial if and only if the sheaf $\ZA[X] = \HA_0(X)$ is a projective object of the abelian category $Ab_{\A^1}(k)$.  Examples of cohomologically trivial schemes include $\Spec L$, where $L$ is a finite separable field extension of $k$, $\A^1$, $\G_m$; and more generally, open subschemes of $\A^1$.  It is easy to show that the product of two cohomologically trivial smooth $k$-schemes is again  cohomologically trivial.
\end{remark} 

We can now formulate a slightly more general definition of a cellular structure:

\begin{definition}
\label{definition cellular}
Let $k$ be a field and let $X \in Sm_k$ be a smooth $k$-scheme. A \emph{cellular structure} on $X$ consists of an increasing filtration 
\[
\emptyset = \Omega_{-1} \subsetneq \Omega_0 \subsetneq \Omega_1 \subsetneq \cdots \subsetneq \Omega_s = X 
\]
by open subschemes such that for each $i\in\{0,\dots,s\}$, the reduced induced closed subscheme $X_i:= \Omega_i - \Omega_{i-1}$ of $\Omega_i$ is $k$-smooth, affine, everywhere of codimension $i$ and cohomologically trivial.  We call $X$ a \emph{cellular scheme} if $X$ is endowed with a cellular structure.  
\end{definition}

\begin{remark}
Let the notation be as in Definition \ref{definition cellular}.
\begin{enumerate}
\item In the decomposition $X_i = \Omega_i - \Omega_{i-1} = \coprod_{j \in J_i} X_{ij}$ of $X_i$ in irreducible components, each $k$-scheme $X_{ij}$ is clearly affine, smooth, cohomologically trivial and of codimension $i$ in $\Omega_i$. Observe, however, that here $s$ need not be the Krull dimension of $X$.

\item Clearly, every strict cellular structure is a cellular structure. These definitions are rather classical; here are some comments.  One might be tempted, given a cellular structure on $X$ as above, to define the $i$-th skeleton of that structure as the closed subset $sk_i(X) = \amalg_{r\in\{i,\dots,s\}, j\in J_{r}} Y_{rj}$ (disjoint) union of the strata with higher codimension.  The problem is that in general, these skeletons are not $k$-smooth (for example, consider Schubert cells in a flag variety), and there is no obvious way to express $sk_i(X)$ as obtained from $sk_{i-1}(X)$ by ``attaching'' $i$-cells.

Contrary to the first impression, the above definition relying on an increasing filtration $\emptyset = \Omega_{-1} \subset \Omega_0 \subset \Omega_1 \subset \cdots \subset \Omega_d = X$ of open subsets is closer to the notion of $CW$-structure on a topological space $X$.  If a topological space $X$ is endowed with a $CW$-structure, with $i$-skeleton $sk_i(X)\subset X$, and where $sk_i(X)$ is obtained from $sk_{i-1}(X)$ by attaching $i$-cells, one may choose a (small) open neighborhood $U_i$ of $sk_i(X)$ such that $sk_i(X)\subset U_i$ is an homotopy equivalence and so that $U_i-U_{i-1}$ is a disjoint sum of contractible spaces indexed by the number of $i$-cells.  The quotient (or rather the ``pair'') $U_i/U_{i-1}$ is then homotopy equivalent to a wedge of pointed spaces homotopy equivalent to $i$-spheres.  In the approach taken in Definition \ref{definition cellular}, the quotient $\Omega_i/\Omega_{i-1}$ is $\A^1$-weakly equivalent to the Thom space of the normal bundle $\nu_i$ of the inclusion of $Y_i \subset \Omega_i$, by the $\A^1$-homotopy purity theorem \cite[Theorem 2.23, page 115]{Morel-Voevodsky}. We will see below that the normal bundles $\nu_i$ are always trivial as the $Y_i$ are cohomologically trivial and affine. Thus, the quotient $\Omega_i/\Omega_{i-1}$ is non canonically $\A^1$-weakly equivalent to the space $(\A^{d-i}/(\A^{d-i}-\{0\}))\wedge (Y_{ij\,\,+})$.  

These $\A^1$-equivalences can be made canonical using the notion of an oriented cellular structure introduced below (see Definition \ref{definition ocellular}), making this approach quite close to the topological one!

\item The main example of a cellular structure that is not strict cellular is given by the Bruhat decomposition of a split reductive group $G$, in which the strata $Y_{ij}$ are the Bruhat cells $BwB$ with $\length(w) = \ell-i$, where $\ell$ is the length of the longest word in the Weyl group of $G$. See Section \ref{section Bruhat decomposition} below for the details.
\end{enumerate}
\end{remark}

\begin{lemma}
\label{lemma coh trivial}
Let $X$ be a cohomologically trivial smooth affine variety over a perfect field $k$.  Then every vector bundle of rank $n\geq 1$ on $X$ is trivial.
\end{lemma}
\begin{proof}
Since $X$ is smooth affine, by a result proved in \cite{Morel-book} for $r\geq 3$ and in \cite{AHW1} in general, the set of isomorphism classes of rank $n$ vector bundles on $X$ can be identified with the set $\Hom_{\sH(k)}(X, BGL_n)$.  Let $\xi$ be a rank $n$ vector bundle on $X$ and let $\gamma$ denote the corresponding element in $\Hom_{\sH(k)}(X, BGL_n)$. By a Postnikov tower argument, obstructions to lifting through the successive Postnikov fibers of $EGL_n \to BGL_n$ lie in Nisnevich cohomology groups of $X$ with coefficients in the higher $\A^1$-homotopy sheaves of $BGL_n$. These cohomology groups vanish by cohomological triviality of $X$ since the higher $\A^1$-homotopy sheaves are strictly $\A^1$-invariant by \cite{Morel-book} and the lemma follows.
\end{proof}

We now briefly recall some facts about the notion of an orientation of a vector bundle on a smooth $k$-scheme.  

\begin{definition}
\label{definition strict orientation}
Let $r\geq 1$, and let $\xi$ be a rank $r$ vector bundle over a smooth $k$-scheme $X$. We say that $\xi$ is \emph{strictly orientable} if the line bundle $\Lambda^r(\xi)$ is trivial. A \emph{strict orientation} of $\xi$ is then a choice of an isomorphism $\theta:\A^1_X \cong \Lambda^r(\xi)$, where $\A^1_X$ is the trivial vector bundle of rank $1$ over $X$.  We will denote by $\sO r^{st}(\xi)$ the set of strict orientations of $\xi$. 
 \end{definition}

For instance, a trivialization of the vector bundle $\xi$ on $X$ induces a strict orientation of $\xi$.  Given $\theta\in\sO r^{st}(\xi)$ and a unit $\alpha\in\sO(X)^\times$, the product $\alpha \cdot \theta$ is clearly a strict orientation of $\xi$ as well. Thus, the group of units $\sO(X)^\times$ acts naturally on the set $\sO r^{st}(\xi)$.

\begin{definition} 
Let $\xi$ be a strictly orientable rank $r$ vector bundle over the smooth $k$-scheme $X$. We denote by $\sO r(\xi)$ the quotient of the set of strict orientations $\sO r^{st}(\xi)$ of $\xi$ by the action of the subgroup $\sO(X)^{\times 2}$ of squares in $\sO(X)^\times$.
\end{definition}

\begin{remark} Let the notation be as above.
\begin{enumerate}
\item If a strict orientation $\theta$ of a vector bundle $\xi$ exists, then the set $\sO r^{st}(\xi)$ is isomorphic to $\sO(X)^\times$ through the action. Hence, the set $\sO r(\xi)$ is isomorphic to the $\Z/2$-vector space $\sO(X)^\times/(\sO(X)^{\times 2})$ of units in $X$ modulo the squares.

\item In general, if $\xi$ is not strictly orientable, we may still define the sheaf of orientations of $\xi$ as the quotient sheaf of the sheaf $\Lambda^r(\xi)^\times$ by the action of the subsheaf $(\G_m)^{(2)}$ of squares in $\G_m$. This is a sheaf of $\Z/2$-vector spaces over $X$ and an orientation of $\xi$ is a section of this sheaf over $X$. There are also other notions of orientability of vector bundles, see for instance \cite[Def. 4.3]{Morel-book} where an orientation is defined to be an isomorphism of $\Lambda^r(\xi)$ with the square of a line bundle. The former definition is clearly more general.

\item Observe that a vector bundle $\xi$ of rank $r$ over a smooth cohomologically trivial $k$-scheme $X$ is automatically strictly orientable.  Indeed, the rank $1$ vector bundle $\Lambda^r(\xi)$ corresponds up to isomorphism to an element of the group $H^1(X;\G_m)$, which is trivial. If $X$ is moreover affine, then Lemma \ref{lemma coh trivial} implies that such a vector bundle is trivial.
\end{enumerate}
\end{remark}

In this article, we use only one notion of orientation, namely, the one of \emph{strict orientation} given by Definition \ref{definition strict orientation}.  We will often drop the adjective \emph{strict} for the sake of brevity.

\begin{definition}
\label{definition ocellular}
Let $X \in Sm_k$ be a smooth $k$-scheme. An \emph{oriented cellular structure} on $X$ consists of a cellular structure on $X$, 
\[
\emptyset = \Omega_{-1} \subsetneq \Omega_0 \subsetneq \Omega_1 \subsetneq \cdots \subsetneq \Omega_s = X 
\]
together with, for each $i$, an orientation $o_i$ of the normal bundle $\nu_{i}$ of the closed immersion $X_i: = \Omega_i - \Omega_{i-1}\subset \Omega_i$. One also says that the choices of the $o_i$'s constitute an orientation of the given cellular structure.
\end{definition}

\begin{definition}
\label{def cellular morphism} 
Let $X$ and $Y$ be smooth $k$-schemes, both endowed with a cellular structure. A morphism of $k$-schemes $f: X\to Y$ is said to be cellular if it preserves the open filtrations, that is, $f(\Omega_i(X))\subset \Omega_i(Y)$ for every $i$. 

For each $i$, we denote by $X_i$ the reduced induced closed subscheme of $\Omega_i(X)$ on the closed complement $\Omega_i(X) -\Omega_{i-1}(X)$, and respectively by $Y_i$, the corresponding closed subschemes of $\Omega_i(Y)$.  Since a cellular morphism $f: X \to Y$  preserves the open filtrations, $f$ induces a morphism of smooth $k$-schemes $f:X_i\to Y_i$ for each $i$, and also a morphism on the induced normal bundles $\nu_f: \nu_{X_i}\to (f|_{X_i})^*(\nu_{Y_i})$. If the cellular structures of $X$ and $Y$ are oriented, then a cellular morphism $f$ is said to be \emph{oriented} if for each $n$ the induced morphism on normal bundles $\nu_f: \nu_{X_n}\to (f|_{X_n})^*(\nu_{Y_n})$ preserves the orientations (which implies in particular that $\nu_f: \nu_{X_n}\to (f|_{X_n})^*(\nu_{Y_n})$ is an isomorphism).
\end{definition}
 
\begin{example}
\label{ex proj space} 
We consider the $n$-dimensional projective space $\P^n$ over $k$. Consider the obvious increasing filtration of linear projective subspaces $$\emptyset\subset \P^0\subset\dots\subset \P^i\subset\dots\subset\P^n$$ corresponding to the canonical flag of subspaces of $\A^{n+1}$ and set $\Omega_i = \P^n-\P^{n-i-1}$. If $X_0, \ldots, X_n$ denote the usual homogeneous coordinates on $\P^n$; or, in other words, the canonical generating sections of $\sO(1)$, then clearly $\Omega_i = \overset{i}{\underset{j = 0}{\cup}} D_+(X_j)$, where $D_+(X_j)$ denotes the complement of the closed subscheme of $\P^n$ defined by $X_j$. Thus $\Omega_i - \Omega_{i-1} = \P^{n-i} - \P^{n-i-1}$ is the zero locus on $\Omega_i$ of the rational functions $\frac{X_0}{X_i}, \ldots, \frac{X_{i-1}}{X_i}$.  Since this sequence is a regular sequence, the normal bundle $\nu_i$ is trivialized by this sequence, thereby inducing an orientation. This data forms the canonical oriented cellular structure on $\P^n$.
\end{example}

The following lemma allows us to generate more examples of oriented cellular schemes. 

\begin{lemma} 
\label{lemma torsor} 
Let $T$ be a split $k$-torus, $X\in Sm_k$ be an oriented cellular scheme, and let $\pi:Y\to X$ be a $T$-torsor over $X$. Let 
\[
\emptyset = \Omega_{-1} \subsetneq \Omega_0 \subsetneq \Omega_1 \subsetneq \cdots \subsetneq \Omega_s = X 
\]
be the open filtration of $X$ defining the cellular structure. Define an increasing filtration by open subsets on $Y$ by pulling back the $\Omega_i$: $\tilde{\Omega}_i:=\pi^{-1}(\Omega_i)$. Then the sequence
\[
\emptyset = \tilde{\Omega}_{-1} \subsetneq \tilde{\Omega}_0 \subsetneq \tilde{\Omega}_1 \subsetneq \cdots \subsetneq \tilde{\Omega}_s = Y
\]
is a cellular structure on $Y$. Moreover, with the obvious notation, the normal bundle of $Y_i\subset \tilde{\Omega}_i$ is the pull-back of the normal bundle of $X_i\subset \Omega_i$ so that the oriented cellular structure on $X$ induces a canonical oriented cellular structure on $Y$.
\end{lemma}

\begin{proof} The smooth $k$-scheme $\tilde{\Omega}_i - \tilde{\Omega}_{i-1}$ is clearly affine and cohomologically trivial, as it is a (trivial) $\G_m$-torsor on $\Omega_i$. The rest easily follows from this.
\end{proof}

\begin{example} 
In this way we get a canonical oriented cell structure on $\A^{n+1}-\{0\}$ using the usual $\G_m$-torsor $\A^{n+1}-\{0\}\to \P^n$.
\end{example}

\begin{remark}
By Lemma \ref{lemma coh trivial}, the normal bundles $\nu_{i}$ in Definition \ref{definition ocellular} of an oriented cellular structure on a smooth $k$-scheme are in fact always trivial.  The main point is not the orientability, but the choice of an orientation.  In all our cases, for instance the oriented cell structure of $\P^n$ described above, the triviality is always proved directly without invoking Lemma \ref{lemma coh trivial}.
\end{remark}

We next record the following particular case of the Thom isomorphism theorem for oriented bundles.

\begin{lemma}
\label{lem:Thomiso} 
Let $\xi$ be a trivial vector bundle of rank $r$ over a smooth $k$-scheme $X$, endowed with an orientation $o\in \sO r(\xi)$. Then there exists an isomorphism $\phi: \xi \stackrel{\simeq}{\to} \A^r_X$ of vector bundles over $X$, which preserves the orientations. Moreover, if $X$ is cohomologically trivial, then the induced pointed $\A^1$-homotopy class (which is an $\A^1$-homotopy equivalence)
\[
Th(\xi) \to \A^r/(\A^r-\{0\}) \wedge (X_+)
\]
only depends on $o\in \sO r(\xi)$.
\end{lemma}

\begin{proof} Let $\psi: \xi \stackrel{\simeq}{\to} \A^r_X$ be an isomorphism of vector bundles over $X$. Taking the top exterior power $\Lambda^r$ and using the strict orientations on both sides yields an isomorphism $\A^1_X \stackrel{\simeq}{\to} \A^1_X$ of line bundles over $X$, which corresponds to an invertible regular function $\alpha$ on $X$.  Multiplying $\psi$ on the first coordinate by $\alpha^{-1}$ yields now an isomorphism $\phi: \xi \stackrel{\simeq}{\to} \A^r_X$ which preserves the two strict orientations.

Let $\phi': \xi \stackrel{\simeq}{\to} \A^r_X$ be another isomorphism taking $\theta$ to the strict orientation of $\A^r_X$ multiplied by a square $\alpha^2$ of a unit $\alpha\in\sO(X)^\times$. The automorphism $\Theta = \phi'\circ \phi^{-1}: \A^r_X \stackrel{\simeq}{\to} \A^r_X$ preserves the strict orientations up to multiplication by $\alpha^2$; that is, an element of $\Hom_k(X, \GL_r)$ whose determinant is $\alpha^2$. Let $\Theta'\in \Hom_k(X, \GL_r)$ be the diagonal matrix with $\alpha^2$ on the first entry and $1$ on the others. Then $\Theta = \Theta'\circ \Theta''$ with $\Theta''\in \Hom_k(X,\SL_r)$. As $X$ is cohomologically trivial and as $\SL_r$ is $\A^1$-connected, a standard Postnikov tower argument shows that any morphism of $k$-schemes $X\to \SL_r$ is $\A^1$-homotopically trivial. This implies that the induced pointed morphism on Thom spaces 
\[
 Th(\Theta): \A^r/(\A^r-\{0\}) \wedge (X_+) \to \A^r/(\A^r-\{0\}) \wedge (X_+)
\]
induced by $\Theta$ is $\A^1$-homotopic to the one induced by $\Theta'$. Since $\A^r/(\A^r-\{0\})$ is isomorphic to $ (\A^1/(\A^1-\{0\}))\wedge (\A^{r-1}/(\A^{r-1}-\{0\}))$, the induced $\A^1$-homotopy class 
\[
Th(\Theta'): \A^r/(\A^r-\{0\}) \wedge (X_+) \to \A^r/(\A^r-\{0\}) \wedge (X_+) 
\]
is clearly the smash product of the identity of $\A^{r-1}/(\A^{r-1}-\{0\})$ with the morphism $\A^1/(\A^1-\{0\})\wedge (X_+) \stackrel{\simeq}{\to} \A^1/(\A^1-\{0\})\wedge (X_+)$ induced by $\alpha^2$. The latter morphism is $\A^1$-homotopic to the identity: this follows from the standard argument that $\P^1\wedge (X_+)\rightarrow \A^1/(\A^1-\{0\})\wedge (X_+)$ is an $\A^1$-weak equivalence, and that the automorphism $\P^1 \wedge (X_+) \stackrel{\simeq}{\to} \P^1 \wedge (X_+)$ given by $[x,y]\mapsto [\alpha^2 x,y]$ is equal to $[x,y]\mapsto [\alpha x, \alpha^{-1}y]$, which is $\A^1$-homotopically trivial, being induced by an element of $\SL_2(X)$.
\end{proof}

\begin{remark}
\label{rem cell} \hspace{1cm}
\begin{enumerate}
\item 
For an oriented cellular scheme $X\in Sm_k$ and a $T$-torsor $Y\to X$ over $X$, with $T$ a split $k$-torus, the morphism $Y\to X$ is clearly an oriented cellular morphism for the oriented cellular structure on $Y$ described in Lemma \ref{lemma torsor}. 

\item Using the notion of an oriented cellular morphism described in Definition \ref{def cellular morphism}, we may also extend the notion of an oriented cellular structure to ind-smooth $k$-schemes in an obvious way, by requiring that each object has a given (oriented) cellular structure and each morphism is (oriented) cellular.  For instance, the closed immersion $\P^{n-1}\subset \P^n$ (defined by $X_n = 0$) is an oriented cellular morphism; observe that $\Omega_{n-1}$ already contains $\P^{n-1}$, and the intersection of the $\Omega_i$'s, for $i\leq (n-1)$, with $\P^{n-1}$ is exactly the $\Omega_i$'s of $\P^{n-1}$.  In this way, we get an oriented cell structure on $\P^\infty$.

\item We may also adapt the notion of an oriented cellular structure to simplicial smooth $k$-schemes, by requiring that each term of the simplicial smooth $k$-scheme is given a (oriented) cellular structure and each face morphism is (oriented) cellular. For example, the classifying space $BT$ of any split torus $T$ has a canonical oriented cellular structure induced by the one on $T$.
\end{enumerate}
\end{remark}

\subsection{Cellular \texorpdfstring{$\A^1$}{A1}-chain complexes} \hfill 
\label{subsection cellular chain complex}

We are now set to introduce the notions of \emph{cellular $\A^1$-chain complex}, \emph{oriented cellular $\A^1$-chain complex} and the corresponding \emph{cellular $\A^1$-homology}. Let $X$ be a smooth $k$-scheme with a cellular structure. Let
\[
\emptyset = \Omega_{-1} \subset \Omega_0 \subset \Omega_1 \subset \cdots \subset \Omega_s = X  
\]
be the corresponding filtration by open subschemes as in Definition \ref{definition cellular}.  Let $\Omega_i - \Omega_{i-1} = \coprod_{j \in J_i} Y_{ij}$ be the decomposition of $\Omega_i - \Omega_{i-1}$ in irreducible components. We can now apply the motivic homotopy purity theorem \cite[Theorem 2.23, page 115]{Morel-Voevodsky} to get a canonical pointed $\A^1$-weak equivalence
\begin{equation}\label{eq1}
 \Omega_i/\Omega_{i-1} \simeq_{\A^1} Th(\nu_i)
\end{equation}
with $\nu_i$ the normal bundle of the closed immersion $X_i\hookrightarrow \Omega_i$ of smooth $k$-schemes. Consider the cofibration sequence
\[
\Omega_{i-1} \to \Omega_i \to \Omega_i/\Omega_{i-1}.
\]
In the corresponding long exact sequence of reduced $\A^1$-homology sheaves 
\[
\cdots \to \HAred_n(\Omega_{i-1}) \to \HAred_n(\Omega_i) \to \HAred_n(\Omega_i/\Omega_{i-1}) \xrightarrow{\delta} \HAred_{n-1}(\Omega_{i-1}) \to \cdots
\]
we can use the previous identification 
\[
\HAred_n(\Omega_n/\Omega_{n-1}) \simeq \HAred_n(Th(\nu_n))
\]
Define $\partial_n$ to be the composite
\[
\HAred_n(Th(\nu_n))\cong \HAred_n(\Omega_n/\Omega_{n-1}) \xrightarrow{\delta} \HAred_{n-1}(\Omega_{n-1}) \to \HAred_{n-1}(\Omega_{n-1}/\Omega_{n-2})\cong \HAred_{n-1}(Th(\nu_{n-1}))
\]
where the map $\HAred_{n-1}(\Omega_{n-1}) \to \HAred_{n-1}(\Omega_{n-1}/\Omega_{n-2})$ is induced by the previous cofibration sequence
\[
\Omega_{n-2} \to \Omega_{n-1} \to \Omega_{n-1}/\Omega_{n-2}. 
\]
It is easy to see that $\partial_{n-1} \circ \partial_n = 0$, for every $n$, thus we get a chain complex (in the category of strictly $\A^1$-invariant sheaves)
\[
 \Ccell_*(X) := (\HAred_n(Th(\nu_n)),\partial_n)
\]
The chain complex $\Ccell_*(X)$ is called the \emph{cellular $\A^1$-chain complex} associated to the cellular structure on $X$, and its homology sheaves $H_n(\Ccell_*(X))$ are called the \emph{cellular $\A^1$-homology} sheaves of $X$.

If we choose an orientation of the given cellular structure on $X$, the chosen orientations (and the canonical $\A^1$-weak equivalence $\A^i/(\A^i-\{0\}) = S^i \wedge \G_m^{\wedge i}$ ) together with the pointed $\A^1$-equivalence of Lemma \ref{lem:Thomiso} define canonical pointed $\A^1$-weak equivalences
\begin{equation}\label{eq2}
\Omega_i/\Omega_{i-1} \simeq_{\A^1} Th(\nu_i) \simeq_{\A^1}  S^i \wedge \G_m^{\wedge i} \wedge \big(\vee_{j \in J_i} (Y_{ij})_+ \big).
\end{equation}
Thus we can further identify, for each $n$:
\[
\HAred_n(\Omega_n/\Omega_{n-1}) \simeq \HAred_n(Th(\nu_n))  \simeq   \underset{j \in J_n}{\oplus} \HAred_n(S^n \wedge \G_m^{\wedge n} \wedge (Y_{nj})_+).
\]
Using the suspension isomorphisms in $\A^1$-homology, and the identification $\ZA(\G_m^{\wedge n}) = \KMW_n$, we get isomorphisms of strictly $\A^1$-invariant sheaves:
\[
\underset{j \in J_n}{\oplus} \HAred_n(S^n \wedge \G_m^{\wedge n} \wedge (Y_{nj})_+)\simeq
\underset{j \in J_n}{\oplus} \HAred_0(\G_m^{\wedge n} \wedge (Y_{nj})_+)\simeq \underset{j \in J_n}{\oplus} \KMW_n \otimes \ZA[Y_{nj}]
\]
(the tensor product being performed in the category of strictly $\A^1$-invariant sheaves). Thus using the orientations, the cellular chain complex $\Ccell_*(X)$ becomes isomorphic to the chain complex with terms in degree $n$: $\underset{j \in J_n}{\oplus} \KMW_n \otimes \ZA[Y_{nj}]$. 

Now we denote by $\tilde\partial_n$ the induced composite
\[
\underset{j \in J_n}{\oplus} \KMW_n \otimes \ZA[Y_{nj}] \simeq \HAred_n(Th(\nu_n)) \xrightarrow{\partial_n} \HAred_{n-1}(Th(\nu_{n-1})) \simeq \underset{j' \in J_{n-1}}{\oplus} \KMW_{n-1}\otimes \ZA[Y_{n-1,j'}]
\]

\begin{definition}
\label{definition cellular homology}
The chain complex (in the category of strictly $\A^1$-invariant sheaves)
\[
\Ctcell_*(X) := \left(\underset{j \in J_n}{\oplus} \KMW_n \otimes \ZA[Y_{nj}], \tilde\partial_n\right)_n
\]
is called the \emph{oriented cellular $\A^1$-chain complex} associated to the oriented cellular structure on $X$.
The $n$th homology sheaf of this complex will be called $n$-th \emph{oriented cellular $\A^1$-homology} of $X$.  After choosing the orientations, the $n$-th oriented cellular $\A^1$-homology of $X$ is seen to be isomorphic to the cellular $\A^1$-homology sheaves $\Hcell_n(X)$.
\end{definition}

\begin{remark}\label{rem cellular suslin} \hfill 
\begin{enumerate}
\item It is easy to check using the previous description that each term $\Ccell_n(X)$ is a projective object of $Ab_{\A^1}(k)$. Indeed, for any cohomologically trivial smooth $k$-scheme $Y$, the strictly $\A^1$-invariant sheaf $\ZA[Y]$ is clearly a projective object of $Ab_{\A^1}(k)$. More generally, for each $n\geq 1$, the tensor product $\KMW_n \otimes \ZA[Y]$ is also a projective object of $Ab_{\A^1}(k)$.  This can be seen by using the adjunction formula:
\[
\Hom_{Ab_{\A^1}(k)}(\KMW_n \otimes \ZA[Y], \bM) \cong\Hom_{Ab_{\A^1}(k)}(\ZA[Y], \bM_{-n}),
\]
for any $\bM\in Ab_{\A^1}(k)$ and the fact that the functor $\bM\mapsto \bM_{-1}$ is an exact functor, which follows from the fact that $\G_m$ is cohomologically trivial.

\item Let $X$ be a cellular smooth $k$-scheme. Without the choice of orientations, the cellular $\A^1$-chain complex $\Ccell_*(X)$ is practically uncomputable.  Its oriented version, $\Ctcell_*(X)$, depending on the choice of orientations is often more explicitly computable. See also the intrinsic interpretation given below in Section \ref{subsection intrinsic cellular chain complex}.

\item If one uses the variant of Suslin homology instead of $\A^1$-homology, the Thom isomorphisms exist without choice of orientations, so that the cellular Suslin chain complex $\CScell_*(X)$ is a well defined chain complex of $\A^1$-invariant sheaves with transfers in the sense of Voevodsky, and is defined for any cellular structure on $X$.
\end{enumerate}
\end{remark}

Let $f: X\to Y$ be a cellular morphism of cellular smooth $k$-schemes (see Definition \ref{def cellular morphism}). For each $i$, we denote by $X_i$ the reduced induced closed subscheme of $\Omega_i(X)$ on the closed complement $\Omega_i(X)-\Omega_{i-1}(X)$, and in the same way $Y_i$, the corresponding closed subschemes of $\Omega_i(Y)$. Then, as $f$  preserves the open filtrations, $f$ induces a morphism of smooth $k$-schemes $f:X_i\to Y_i$ for each $i$, and also a morphism on the induced normal bundles $\nu_f: \nu_{X_i}\to (f|_{X_i})^*(\nu_{Y_i})$. It follows easily that $f$ induces then a canonical morphism of cellular $\A^1$-chain complexes:
\[
 f^{\rm cell}_*:\Ccell_*(X)  \rightarrow  \Ccell_*(Y)
\]
If the cellular structures of $X$ and $Y$ are oriented, the morphism $f^{\rm cell}_*$ induces a morphism $\tilde{f}^{\rm cell}_*: \Ctcell_*(X)  \rightarrow  \Ctcell_*(Y)$ of oriented cellular $\A^1$-chain complexes, so that in degree $n$, $\tilde{f}^{\rm cell}_n$ is the morphism
\[
\KMW_n \otimes \ZA[X_n] \cong \HAred_n(Th(\nu_{X_n})) \xrightarrow{\HAred_n(Th(\nu_f))} \HAred_n(Th(\nu_{Y_n})) \cong \KMW_n\otimes \ZA[Y_n]
\]
induced by $\nu_f$ in the obvious way. 

If $f$ is an oriented cellular morphism, that is, if for each $n$ the induced morphism on normal bundles $\nu_f: \nu_{X_n}\to (f|_{X_n})^*(\nu_{Y_n})$ preserves the orientations, then $\tilde{f}^{\rm cell}_n$ is just the tensor product of $\HAred_0(f|_{X_n}): \ZA[X_n]\to \ZA[Y_n]$ with $\KMW_n$.

\subsection{Intrinsic interpretation in the derived category \texorpdfstring{$D(Ab_{\A^1}(k))$}{D(AbA1(k))}} \hfill 
\label{subsection intrinsic cellular chain complex}

Given a sheaf of abelian groups $\bM$ on $Sm_k$, we can construct the cohomological cellular cochain complex $C^*_{\rm cell}(X;\bM)$ which is the (cochain) complex of abelian groups with terms $C_{\rm cell}^n(X;\bM) := \Hom_{Ab_{\A^1}(k)} (\Ccell_n(X),\bM)$ and the obvious induced differentials.

\begin{proposition}
\label{prop: cohomology}
Let $X$ be a cellular smooth scheme over $k$.  For any strictly $\A^1$-invariant sheaf $\bM$ on $Sm_k$, we have isomorphisms
\[
H^n_{\rm Nis}(X, \bM) \xrightarrow{\sim} H^n(C_{\rm cell}^*(X; \bM)) \xrightarrow{\sim} \Hom_{D(Ab_{\A^1}(k))}(\Ccell_*(X), \bM[n]),
\]
natural in both $X$ and $\bM$.
\end{proposition}
\begin{proof} The first isomorphism follows from the cohomological spectral sequence computing $H^*_{\rm Nis}(X,\bM)$ associated with the filtration induced by the flag of open subsets 
\[
\emptyset = \Omega_{-1} \subsetneq \Omega_0 \subsetneq \Omega_1 \subsetneq \cdots \subsetneq \Omega_s = X  
\]
and the cohomological triviality of the $Y_{ij}$, which shows that the spectral sequence collapses from $E_2$ page onwards, yielding the isomorphism. The second isomorphism follows from the fact that each term $\Ccell_n(X)$ is a projective object of $Ab_{\A^1}(k)$, which follows again from the cohomological triviality
of the $Y_{ij}$.
\end{proof} 

\begin{remark} \hspace{1cm}
\begin{enumerate}
\item Note that the proof of the first isomorphism given above exactly parallels the way in which we introduced the cellular $\A^1$-chain complex: observe that the $E_1$ term of the spectral sequence considered in the proof of Proposition \ref{prop: cohomology} is the cochain complex $C_{\rm cell}^n(X;\bM)$ and the $E_2$ terms are $H^*_{\rm Nis}(X, \bM)$.

\item It is known from \cite{Morel-book} that for any strictly $\A^1$-invariant sheaf $\bM$ on $Sm_k$ and any smooth $k$-scheme $X$, the natural comparison morphism:
\[
H^n_{\rm Zar}(X, \bM) \cong H^n_{\rm Nis}(X, \bM)
\]
is an isomorphism for any $n\geq 0$.
\end{enumerate}
\end{remark}

We first observe the following easy consequence:

\begin{corollary}
\label{corollary cellular cohomology vanishing}
Let $X$ be a cellular smooth $k$-scheme. Then for any strictly $\A^1$-invariant sheaf $\bM$ on $Sm_k$ with $M_{-i} = 0$ for some $i\geq 0$ then, for $n\geq i$
\[
 H^n_{\rm Nis}(X, \bM) = 0
\]
\end{corollary}

\begin{proof}
Indeed, under the assumptions, by Remark \ref{rem cellular suslin} (1) the cohomological cellular cochain complex $C^*_{\rm cell}(X;\bM)$ vanishes in degrees $\geq i$.
\end{proof}

\begin{corollary}
\label{corollary cellular Hurewicz}
Let $n\geq 1$ and let $X$ be an $\A^1$-$(n-1)$-connected, cellular smooth $k$-scheme.  Then $\Hcell_0(X) = \Z$ and $\Hcell_i(X) = 0$ for $1 \leq i \leq n-1$, and there exists a canonical isomorphism 
\[
 \Hcell_{n}(X) \simeq \HA_{n}(X)
\]
and an epimorphism $\HA_{n+1}(X)\twoheadrightarrow \Hcell_{n+1}(X)$.
\end{corollary}
\begin{proof} If $X$ an $\A^1$-$(n-1)$-connected space and $\bM$ a strictly $\A^1$-invariant sheaf, one has $H^0(X;\bM)= \bM(k)$ and $H^i(X;\bM) = 0$ for $i\in\{1,\dots,n-1\}$. The first equality shows that the canonical morphism $\Hcell_0(X)\to \Z$ is an isomorphism. If $\tilde{C}^{\rm cell}_*(X)$ denotes the reduced cell chain complex of $X$ (that is, the kernel of $\Ccell_*(X) \to \Z$), then Proposition \ref{prop: cohomology} easily implies that $H_i(\tilde{C}^{\rm cell}_*(X)) = \Hcell_i(X) = 0$ for $i\in\{1,\dots,n-1\}$. Now, since
\[
H^{n}(C^*_{\rm cell}(X;\bM))=H^n(X;\bM) = \Hom_{Ab_{\A^1}(k)}(\HA_n(X),\bM)
\]
for every $\bM\in Ab_{\A^1}(k)$, the above vanishing gives an isomorphism $\Hcell_{n}(X) \simeq \HA_{n}(X)$. Now if $C_*$ is any chain complex in any abelian category $\sA$ with $H_iC_* = 0$ for $i\leq n-1$ and $M$ an object of $\sA$, then one has an exact sequence of the form
\[
0\to \Ext^1_\sA(H_nC_*;M) \to \Hom_{D\sA}(C_*;M[n+1])) \to \Hom_{\sA}(H_{n+1}C_*;M)\to \Ext^2_\sA(H_nC_*,M).
\]
Using Proposition \ref{prop: cohomology} and the above exact sequences for $\sA=Ab_{\A^1}(k)$ with $C_* = C_*^{\A^1}(X)$ as well as $\Ccell(X)$ we conclude by a diagram chase that the comparison morphism $$\Hom(\Hcell_{n+1}(X),\bM) \to \Hom(\HA_{n+1}(X),\bM)$$ is injective for every $\bM \in Ab_{\A^1}(k)$.  Consequently, the induced morphism $\HA_{n+1}(X)\rightarrow \Hcell_{n+1}(X)$ is an epimorphism.
\end{proof}

If $X$ and $Y$ are cellular smooth $k$-schemes, the product $X\times_k Y$ can be naturally endowed with a cellular structure defined by $\Omega_i(X\times_k Y) := \cup _{j = 0}^\infty \big(\Omega_j(X)\times_k \Omega_{i-j}(Y)\big)$.  In case the cellular structures on $X$ and $Y$ are oriented, the cellular structure on $X \times_k Y$ can be oriented using the obvious tensor products of appropriate orientations for $X$ with those for $Y$.  The proof of the following Lemma is easy and left to the reader.

\begin{lemma}
\label{lemma Kunneth formula}
Let $X$ and $Y$ be cellular smooth $k$-schemes.  The induced morphism of chain complexes (of strictly $\A^1$-invariant sheaves)
\[
\Ccell_*(X) \otimes \Ccell_*(Y)  \to \Ccell_*(X \times_k Y)
\]
is an isomorphism.  If the cellular structures on $X$ and $Y$ are both oriented, then the induced morphism of chain complexes (of strictly $\A^1$-invariant sheaves)
\[
\Ctcell_*(X) \otimes \Ctcell_*(Y)  \to \Ctcell_*(X \times_k Y)
\]
is an isomorphism.
\end{lemma} 

\begin{remark} 
Projective objects in the abelian category $Ab_{\A^1}$ need not be flat. For instance, the unramified Milnor-Witt $K$-theory sheaves $\KMW_n$ are projective objects, but not flat since the multiplication by an odd integer on $\KMW_1$ is never injective, although it is injective on $\Z$. However, in the case of cellular smooth $k$-schemes, the tensor product $\Ccell_*(X) \otimes \Ccell_*(Y)$ is the derived tensor product of the two complexes $\Ccell_*(X)$ and $\Ccell_*(Y)$, as the two cellular $\A^1$-chain complexes consist of projective objects in each degree.
\end{remark}

We will next see that up to chain homotopy, the correspondence $X \mapsto \Ccell_*(X)$ is functorial on the category of smooth $k$-schemes admitting a cellular structure.

Let $X$ be a cellular smooth $k$-scheme and $C_*$ be a chain complex of strictly $\A^1$-invariant sheaves. The cellular $\A^1$-chain complex $\Ccell_*(X)$ is bounded and consists of terms that are projective objects in $Ab_{\A^1}(k)$. Thus, the group 
\[                                                                                                                                                                                                                                                               
\Hom_{D(Ab_{\A^1}(k))}(\Ccell_*(X),C_*)                                                                                                                                                                                                                                                             
\]
of morphisms in the derived category of $Ab_{\A^1}(k)$ is the group of chain homotopy classes of morphisms of chain complexes from $\Ccell_*(X)$ to $C_*$. On the other hand we may consider the hypercohomology $H^n_{\rm Nis}(X;C_*)$ of $X$ with coefficients in $C_*$, which is defined by
\[
 H^n_{\rm Nis}(X;C_*): = \Hom_{D(Ab(k))}(\Z(X),C_*[n]).
\]
Clearly, $C_*$ is $\A^1$-local as it consists of strictly $\A^1$-invariant sheaves, so its homology sheaves are strictly $\A^1$-invariant sheaves as well. Thus, we get an identification:
\[
H^n_{\rm Nis}(X;C_*) = \Hom_{D(Ab_{\A^1}(k))}(C_*^{\A^1}(X),C_*[n]).
\]
Proposition \ref{prop: cohomology} above states that if $C_*$ is a sheaf $M$ concentrated in degree $0$, then there is a canonical isomorphism
\[
H^n_{\rm Nis}(X;C_*) \cong \Hom_{D(Ab_{\A^1}(k))}(C_*^{\A^1}(X),M[n]).
\]
It is thus tempting to extend this isomorphism to any complex $C_*$ of strictly $\A^1$-invariant sheaves.

\begin{definition}
We say that a chain complex $D_*$ of abelian sheaves of groups is {\em strictly $\A^1$-resolvable} if there is a chain complex $C_*$ of strictly $\A^1$-invariant sheaves and a morphism 
\[
\phi: D_* \to C_*
\]
in $D(Ab_{\A^1}(k))$ such that for any chain complex $C'_*$ of strictly $\A^1$-invariant sheaves, the morphism
\[
\Hom_{D(Ab_{\A^1}(k))}(C_*,C'_*) \to \Hom_{D_{\A^1}(k)}(D_*,C'_*)
\]
induced by $\phi$ is an isomorphism. We call the pair $(C_*, \phi)$ a {\em strict $\A^1$-resolution} of $D_*$.
\end{definition}

\begin{remark}
\label{rem strict A^1 resolvable}
In other words, $D_*$ is strictly $\A^1$-resolvable if the functor $D(Ab_{\A^1}(k)) \to Ab$ defined by $C'_*\mapsto \Hom_{D_{\A^1}(k)}(D_*,C'_*)$ is representable. If this is the case, then the pair $(C_*, \phi)$ consisting of $C_*$ and the morphism $D_* \to C_*$ is clearly unique up to a canonical isomorphism. Moreover, the correspondence $D_*\mapsto C_*$ defines a functor on the full subcategory of strictly $\A^1$-resolvable chain complexes.
\end{remark}


\begin{lemma} 
\label{lem crit strict A^1 res}
Let $D_*$ be a connective chain complex of sheaves of abelian groups, in the sense that there exists $n_0\in \Z$ such that $H_n(D_*) = 0$ for $n\leq n_0$. Let $\phi: D_* \to C_*$ be a morphism in $D_{\A^1}(k)$, where $C_*$ is a chain complex of strictly $\A^1$-invariant sheaves. Then $(C_*,\phi)$ is strict $\A^1$-resolution of $D_*$ if and only if for any strictly $\A^1$-invariant sheaf $M$ and any $n\in \Z$ the morphism
\[
\Hom_{D(Ab_{\A^1}(k))}(C_*,M[n]) \to \Hom_{D_{\A^1}(k)}(D_*,M[n])
\]
induced by $\phi$ is an isomorphism.
\end{lemma}

\begin{proof} 
For any chain complex $C'_*$ of strictly $\A^1$-invariant sheaves, the group $\Hom_{D_{\A^1}(k)}(D_*,C'_*)$ is clearly equal to the group $\Hom_{D_{\A^1}(k)}(D_*,\tau_{\geq n_0} C'_*)$, where $\tau_{\geq n_0} C'_*\to C'_*$ denotes the truncation of $C'_*$ with same homology in degrees $\geq n_0$ and with trivial homology in degrees $< n_0$.  Consider the successive truncations $\tau_{\geq r} C'_*$ for $r\geq n_0$.  The cone of $\tau_{\geq r+1} C'_* \to \tau_{\geq r} C'_*$ is a complex quasi-isomorphic to $H_r (C'_*)$ placed in degree $r$. Let $C'_*/\tau_{\geq r} C'_*$ denotes the cone of $\tau_{\geq r} C'_* \to C'_*$.  Inductively, using the assumptions, we may thus show that for any $r\geq n_0$
\[
\Hom_{D(Ab_{\A^1}(k))}(C_*,C'_*/\tau_{\geq r} C'_*) \to \Hom_{D_{\A^1}(k)}(D_*,C'_*/\tau_{\geq r} C'_*) 
\]
is an isomorphism.  We conclude using the fact that the morphism 
\[
C'_* \to \holim_r C'_*/\tau_{\geq r} C'_*
\]
is a quasi-isomorphism (see \cite[Definition 1.31, page 58 and Theorem 1.37, page 60]{Morel-Voevodsky} for the analogous statement for simplicial sheaves of sets).
\end{proof}

\begin{example} 
\label{ex strict A^1-resolution}
Let $X$ be a smooth, cohomologically trivial $k$-scheme. Then $\HA_0(X)$ is a projective object of $Ab_{\A^1}(k)$ and the pair $(\HA_0(X),\phi_X)$, where $\phi : C_*^{\A^1}(X) \to \HA_0(X)$ is the obvious projection, is a strict $\A^1$-resolution of $C_*^{\A^1}(X)$ by Lemma \ref{lem crit strict A^1 res}.

More generally, for any integer $n\geq 0$, $(\A^{n}/(\A^{n}-\{0\}))\wedge (X_+)$ is strictly $\A^1$-resolvable by $((\KMW_n\otimes \HA_0(X))[n],\phi_n\otimes \phi_X) $, where $\phi_n: \tilde{C}_*^{\A^1}(\A^{n}/(\A^{n}-\{0\})) \to \KMW_n[n]$ is the obvious projection. This can be verified by an easy direct computation and the above observations.
\end{example}

\begin{proposition}  
Let $X$ be a cellular smooth $k$-scheme. Then $C_*^{\A^1}(X)$ is strictly $\A^1$-resolvable by its cellular $\A^1$-chain complex. More precisely, there exists a (unique) morphism 
\[
 \phi_X : C_*^{\A^1}(X)  \to \Ccell_*(X)
\]
in $D_{\A^1}(k)$ which is a strict $\A^1$-resolution.
\end{proposition}

\begin{proof} 
We proceed by induction on the length of the filtration by open subsets of $X$
\[
\emptyset = \Omega_{-1} \subsetneq \Omega_0 \subsetneq \Omega_1 \subsetneq \cdots \subsetneq \Omega_s = X  
\]
which defines the cellular structure. If $s = 0$, the statement holds by Example \ref{ex strict A^1-resolution}.

Now assume $s\geq 1$. Let 
\[
\Omega_{s-1}\to X \to Th(\nu_s)
\]
be the $\A^1$-cofibration sequence deduced from the cellular structure (in degree $s$). This induces an exact triangle 
\[
C_*^{\A^1}(\Omega_{s-1}) \to C_*^{\A^1}(X)  \to \tilde{C}_*^{\A^1}(Th(\nu_s)).  
\]
Now, the Thom space $Th(\nu_s)$ is (non-canonically) isomorphic to $\A^{s}/(\A^{s}-\{0\})\wedge ((X_s)_+)\simeq_{\A^1} S^{s}\wedge \G_m^{s}\wedge ((X_s)_+)$ by Lemma \ref{lemma coh trivial}, since $X_s$ is cohomologically trivial.  We consider the connecting morphism 
\begin{equation}
\label{eq connecting morphism}
\tilde{C}_*^{\A^1}(Th(\nu_s) [-1] \to C_*^{\A^1}(\Omega_{s-1}) 
\end{equation}
in $D_{\A^1}(k)$ obtained from the above triangle. Both the source and the target of this morphism are strictly $\A^1$-resolvable. The target is strictly $\A^1$-resolvable by the induction hypothesis. The source is strictly $\A^1$-resolvable by Example \ref{ex strict A^1-resolution} above, by remembering that $Th(\nu_s)$ is (non-canonically) isomorphic in the pointed $\A^1$-homotopy category to $(\A^s/(\A^s-\{0\}))\wedge ((Y_s)_+)$, so that the strict $\A^1$-localization is (non-canonically) $\KMW_s\otimes \HA_0(Y_s)[s]$.

It follows from Remark \ref{rem strict A^1 resolvable} that the morphism \eqref{eq connecting morphism} induces a unique morphism in $D(Ab_{\A^1}(k))$ between the corresponding strict $\A^1$-resolutions:
\begin{equation}
\label{eq connecting morphism 2}
\HA_s(Th(\nu_s))[s-1] \to \Ccell_*(\Omega_{s-1}),
\end{equation}
where $\Ccell_*(\Omega_{s-1})$ is the cellular $\A^1$-chain complex corresponding to the cellular structure on $\Omega_{s-1}$ induced by the one on $X$.  Since the chain complex $\Ccell_*(\Omega_{s-1})$ is concentrated in degrees $\leq s-1$ and since $\HA_s(Th(\nu_s))$ is a strictly $\A^1$-invariant sheaf, the morphism \eqref{eq connecting morphism 2} is determined by the morphism it induces on $H_{s-1}$.  Thus, the morphism \eqref{eq connecting morphism 2} corresponds exactly to a morphism in $Ab_{\A^1}(k)$ given by the composition 
\[
\HA_s(Th(\nu_s)) \cong \KMW_s\otimes \HA_0(Y_s) \to Z_{s-1}^{\rm cell}(\Omega_{s-1}) \hookrightarrow \Ccell_{s-1}(\Omega_{s-1}) \cong \HA_{s-1}(Th(\nu_{s-1}))
\]
by the construction of the cellular $\A^1$-chain complex.  This induced morphism $\HA_{s}(Th(\nu_{s}))\to\HA_{s-1}(Th(\nu_{s-1}))$ is a differential of the complex $\Ccell_*(\Omega_{s-1})$.  This follows by functoriality of the strict $\A^1$-resolution applied to the composition 
\[                                                                                                                                                                                                                                                                                             
\tilde{C}_*^{\A^1}(Th(\nu_{s}))[s-1] \to C_*^{\A^1}(\Omega_{s-1}) \to \tilde{C}_*^{\A^1}(\Omega_{s-1}) \cong \tilde{C}_*^{\A^1}(Th(\nu_{s-1}))                                                                                                                                                                                                                                                                                            \]
and the definition of the differential as the morphism $\HA_{s}(Th(\nu_s))\to \HA_{s-1}(Th(\nu_{s-1}))$ induced by the connecting morphism in the $\A^1$-cofiber sequence
\[
Th(\nu_{s-1}) \cong \Omega_{s-1}/\Omega_{s-2} \subset \Omega_{s}/\Omega_{s-2} \to \Omega_{s}/\Omega_{s-1}\cong Th(\nu_{s}).
\]
Therefore, the cone of the morphism $\HA_s(Th(\nu_s))[s-1] \to \Ccell_*(\Omega_{s-1})$ in $D(Ab_{\A^1}(k))$ is quasi-isomorphic to $\Ccell_*(X)$.  

This allows us to extend the morphism $\phi_{\Omega_{s-1}}$ to a morphism $\phi_X$ in $D(Ab_{\A^1}(k))$, which defines a morphism of exact triangles equal to $\phi_{\Omega_{s-1}}$ on one extremity and $\phi_{Th(\nu_{s})}$ on the other one. It is then easy to check, using the induction hypothesis for $\Omega_{s-1}$ that $\phi_X$ satisfies the criterion given by Lemma \ref{lem crit strict A^1 res}.  This completes the proof.
\end{proof}

We get the following interpretation of $\Ccell_*(X)$ up to quasi-isomorphism in $D(Ab_{\A^1}(k))$ from the above discussion. 

\begin{corollary} 
\label{cor: instrinsic}
Let $X$ be a cellular smooth $k$-scheme. Then for any chain complex $C_*$ of strictly $\A^1$-invariant sheaves on $Sm_k$, one has a canonical isomorphism
\[
\Hom_{D_{\A^1}(k)}(C_*^{\A^1}(X),C_*) \cong \Hom_{D(Ab_{\A^1}(k))}(\Ccell_*(X),C_*).
\]
In other words, if $X$ admits a cellular structure, the functor
\[
D(Ab_{\A^1}(k)) \to Ab; \quad C_*\mapsto \Hom_{D_{\A^1}(k)}(C_*^{\A^1}(X),C_*)
\]
is represented by $\Ccell_*(X)$.
\end{corollary}

\begin{remark}
\label{remark pro-object}
If $X$ is not cellular, then in general the functor in Corollary \ref{cor: instrinsic} is not representable by a complex, but by a pro-object of $D(Ab_{\A^1}(k))$. This follows directly from standard ``Brown representability type'' results. One of the main points here is that for smooth $k$-schemes admitting a cellular structure, this representing pro-object in the category $D(Ab_{\A^1}(k))$ is in fact constant and the cellular $\A^1$-chain complex is an explicit model for it.  Consequently, the homology groups of this pro-object are also constant in the category of pro-objects in $Ab(k)$.
\end{remark}

The following lemma, quite analogous to the classical properties of cellular chain complexes associated to a CW-structure on a topological space $X$, is an easy consequence of the above observations.

\begin{lemma}
\label{lemma cellular functoriality}
Let $f: X\rightarrow Y$ be a morphism between smooth $k$-schemes in the $\A^1$-homotopy category $\sH(k)$ and assume that $X$ and $Y$ are both equipped with cellular structures. Then there exists a canonical homotopy class of morphisms of cellular $\A^1$-chain complexes
\[
[f]^{\rm cell}_*:\Ccell_*(X)  \rightarrow  \Ccell_*(Y)
\]
which only depends on the class of $f$ in $\Hom_{{\sH(k)}}(X,Y)$.  In particular $f$ induces a {\em canonical morphism}
\[
 \Hcell(f)_*: \Hcell_*(X) \rightarrow \Hcell_*(Y)
\]
\end{lemma}

\begin{remark} 
In case $f$ is a cellular morphism of smooth $k$-schemes, in the sense of Definition \ref{def cellular morphism}, then the canonical morphism of cellular $\A^1$-chain complexes $f^{\rm cell}_*:\Ccell_*(X) \rightarrow  \Ccell_*(Y)$ induced by $f$ has $[f]^{\rm cell}_*$ for its chain homotopy class.
\end{remark}

We also observe two trivial consequences of Lemma \ref{lemma cellular functoriality}, obtained by applying it to $f=\Id_X:X\to X$).

\begin{corollary} 
Let $X$ be a smooth $k$-scheme endowed with two cellular structures.  Then there exists a canonical homotopy equivalence between the two corresponding cellular $\A^1$-chain complexes, and consequently, a canonical isomorphism between the cellular $\A^1$-homologies of $X$ corresponding to the two structures.
\end{corollary}

\begin{corollary}
\label{cor functor cell}
The correspondence $X\mapsto \Ccell_*(X)$ defines a functor from the full subcategory of $\sH(k)$ with objects the smooth $k$-schemes admitting at least one cellular structure to the homotopy category of (non-negative) chain complexes of projective objects in $Ab_{\A^1}(k)$.  Consequently, it induces a functor $X \mapsto \Hcell_*(X)$ from the full subcategory of $\sH(k)$ with objects the smooth $k$-schemes admitting at least one cellular structure to the category of graded strictly $\A^1$-invariant sheaves. 
\end{corollary}

\begin{remark} 
As noted in Remark \ref{remark pro-object}, the functor $X\mapsto \Ccell_*(X)$ in Corollary \ref{cor functor cell} extends to a functor defined on the category $\sH(k)$, but with target the category of pro-objects in the homotopy category of (non-negative) chain complexes of strictly $\A^1$-invariant sheaves:
\[
Sm_k \to \text{pro-}D(Ab_{\A^1}(k)); \quad X \mapsto C^{\text{st-}\A^1}_*(X)
\]
and the corresponding homology $H^{\text{st-}\A^1}_*$, called \emph{strict $\A^1$-homology}, takes its values in the category $\text{pro-} Ab_{\A^1}(k)$. Observe that $H^{\text{st-}\A^1}_n(X) = 0$ for $n>\dim(X)$ by the analogue of Proposition \ref{prop: cohomology}, since the Nisnevich cohomological dimension of a smooth scheme $X$ is $\dim(X)$.  These objects seem to be much more computable compared to the corresponding $\A^1$-homology sheaves. In fact, there are reasons to believe that for a smooth (projective) $k$-scheme $X$, the pro-objects $H^{\text{st-}\A^1}_*(X)$ are always constant.
\end{remark}

\subsection{Cellular \texorpdfstring{$\A^1$}{A1}-homology of projective spaces and punctured affine spaces}\hfill 
\label{subsection homology of Pn}

In this section, we explicitly compute the cellular $\A^1$chain complex and the cellular $\A^1$-homology sheaves of the $n$-dimensional projective space $\P^n$ with respect to the canonical oriented cellular structure described in Example \ref{ex proj space}.  If $\Omega_{\bullet}$ denotes the resulting filtration of $\P^n$ by open subschemes, we have $\Omega_{n-1} = \P^n-\{[0,\dots,0,1]\}$.  Thus, $\Omega_{n-1}= \A^{n}-\{0\} \times_{\G_m} \A^1$ is the total space of the canonical line bundle over $\P^{n-1}$.  The projection morphism $\Omega_{n-1} \to \P^{n-1}$ is cellular with respect to the cellular structure on $\Omega_{n-1}$ induced by the one on $\P^n$ and the canonical cellular structure on $\P^{n-1}$.  In fact, the morphism $\pi_{n-1}:\Omega_{n-1} \to \P^{n-1}$ is oriented cellular in the sense of Definition \ref{def cellular morphism} by Lemma \ref{lemma torsor}.  Since $\pi_{n-1}$ is a vector bundle of rank one, the induced morphism
\[
(\tilde{\pi}_{n-1}^{\rm cell})_*: \Ctcell_*(\Omega_{n-1})  \rightarrow  \Ctcell_*(\P^{n-1})
\]
is an isomorphism of (oriented) cellular $\A^1$-chain complexes in $D(Ab_{\A^1}(k))$. 

Let $\~\Omega_{\bullet}$ denote the filtration of $\A^{n+1}-\{0\}$ by open subschemes obtained by pulling back $\Omega_{\bullet}$ by the canonical $\G_m$-torsor $\A^{n+1}-\{0\} \to \P^n$ as described in Lemma \ref{lemma torsor}.  In exactly the same way as above, one can verify that the canonical projection $\tilde{\Omega}_{n-1}\to \A^{n}-\{0\}$, which is the projection of a (trivial) vector bundle of rank one, induces an isomorphism 
\[
\Ctcell_*(\tilde{\Omega}_{n-1}) \xrightarrow{\simeq} \Ctcell_*(\A^{n}-\{0\})
\]
of (oriented) cellular $\A^1$-chain complexes.  Observe that $\tilde{\Omega}_i= (\A^{i+1}-\{0\})\times \A^{n-i}$, and consequently, it follows that $\tilde{\Omega}_i-\tilde{\Omega}_{i-1} = \A^i\times \G_m \times \A^{n-i}$.  We now point out that for the morphism $\A^{n+1}-\{0\} \to \P^n$ to be oriented cellular, the orientations for $\A^{n+1}-\{0\}$ have to be the ones induced by those on $\P^n$, which are described in example \ref{ex proj space}.  If $X_0, \ldots, X_{n}$ denote the $n+1$ coordinate functions on $\A^{n+1}-\{0\}$, then the normal bundle $\nu_i$ of the closed immersion of $\tilde{\Omega}_i-\tilde{\Omega}_{i-1}$ in $\tilde{\Omega}_i$ is trivialized (or rather, oriented) by the functions $\frac{X_j}{X_n}$ (and not by the functions $X_j$) for $j\in\{0,\dots, i-1\}$. Thus, 
\begin{equation}\label{eq An-0}
\Ctcell_i(\A^{n+1}-\{0\}) \cong \KMW_i \otimes \ZA[\G_m] 
\end{equation}
for each $i\leq n$, where the isomorphism $\KMW_i\otimes \ZA[\G_m] \cong \HA_{i} (Th(\nu_i))$ is coming from the above trivializations.

\begin{remark} 
\label{rem H} 
Let $\bH:=\ZA[\G_m]$ denote the free strictly $\A^1$-invariant sheaf on $\G_m$.  Clearly, $\bH$ is a commutative and cocommutative Hopf algebra in $Ab_{\A^1}$ for the $\A^1$-tensor product.  The structure morphism $\G_m\to \Spec k$ induces an augmentation morphism, the counit of the structure, which is split by the unit of $\G_m$.  The augmentation ideal of $\bH$, that is, the kernel of this augmentation morphism, is canonically $\ZA(\G_m) = \KMW_1$ (where we use $1$ as the base point of $\G_m$).  The above splitting thus induces a decomposition
\[
 \bH = \KMW_1 \oplus \Z.
\]
Recall that $\otimes$ denotes the symmetric monoidal structure on $Ab_{\A^1}(k)$.  With the above decomposition, the product $\bH\otimes\bH\to \bH$ is easily computed to be the following: it is the obvious morphism on the factor 
\[
\Z \oplus (\KMW_1\otimes\Z) \oplus (\Z\otimes \KMW_1) \to \bH
\] 
and it is given by
\[
\eta:\KMW_1\otimes \KMW_1 =  \KMW_2 \to \KMW_1\subset \bH 
\]
on the remaining factor $\KMW_1\otimes \KMW_1 \subset \bH$; this is in fact precisely the definition of $\eta$ (see \cite{Morel-book}). The diagonal $\Psi: \bH\to\bH\otimes \bH$ giving the Hopf algebra structure is easily checked to be the inclusion 
\[
\Z=\Z\otimes\Z \subset \bH\otimes \bH
\]
on the factor $\Z \subset \bH$ and on the factor $\KMW_1\subset \bH$, it is given by the morphism
\[
\begin{split}
\KMW_1  &\to (\KMW_1 \otimes \Z) \oplus (\Z \otimes \KMW_1) \oplus (\KMW_1\otimes \KMW_1) \subset \bH \otimes \bH \\
(t) &\mapsto (t) \otimes 1 + 1 \otimes (t) + (-1) \otimes (t).
\end{split}
\]
Observe that for a unit $t$, we have $(t)\otimes(t) = (-1)\otimes(t) = (t)\otimes(-1)$ in $\KMW_2$. 

Using this structure, if $\bM$ and $\bN$ are strictly $\A^1$-invariant sheaves with an $\bH$-module structure, one defines as usual a $\bH$-module structure on the ($\A^1$-) tensor product $\bM\otimes \bN$ by the composition:
\[
\bM \otimes \bN \otimes \bH  \xrightarrow{\Id_{\bM} \otimes \Id_{\bN} \otimes \Psi} \bM \otimes \bN \otimes \bH  \otimes \bH \cong (\bM \otimes \bH)  \otimes (\bN \otimes \bH)  \to \bM\otimes \bN.
\]
For instance, if $\KMW_1$ is endowed with the $\bH$-module structure as the kernel of the augmentation $\bH\to\Z$, then the $n$-th tensor product $(\KMW_1)^{\otimes n} = \KMW_n$ admits an induced $\bH$-module structure. 
\end{remark}

The action of $\G_m$ on $\A^{n+1}-\{0\}$ preserves the oriented cellular structure by Lemma \ref{lemma torsor}, and induces an action of $\bH$ on the entire cellular $\A^1$-chain complex. Thus, $\Ctcell_*(\A^{n+1}-\{0\})$ (and also $\Ccell_*(\A^{n+1}-\{0\})$) is a complex of $\bH$-modules in $Ab_{\A^1}(k)$.  One may show that the isomorphism (\ref{eq An-0}) above is an isomorphism of $\bH$-modules where the action of $\bH$ on $\KMW_i$ is the trivial one. This is contained in the following result, which computes the oriented cellular $\A^1$-chain complex of $\A^{n+1}-\{0\}$ and its cellular $\A^1$-homology.

\begin{theorem}
\label{thm An-0} 
The oriented cellular $\A^1$-chain complex $\Ctcell_*(\A^{n+1}-\{0\})$ of $\A^{n+1}-\{0\}$, as a complex of right $\bH$-modules, is canonically isomorphic to:
\[
\KMW_n \otimes \bH \stackrel{\partial_n}{\to} \KMW_{n-1} \otimes\bH \stackrel{\partial_{n-1}}{\to} \dots \stackrel{\partial_{i+1}}{\to} \KMW_i\otimes \bH \stackrel{\partial_i}{\to} \KMW_{i-1}\otimes \bH \dots \KMW_1\otimes \bH \stackrel{\partial_1}{\to} \bH
\]
where for each $i$ the $\bH$-module structure on $\KMW_i\otimes \bH$ is the tensor product of the trivial structure on $\KMW_i$ and the canonical one on $\bH$ and where the differential 
\[
\partial_i : \big(\KMW_{i+1} \oplus \KMW_{i}\big)  \cong \KMW_{i} \otimes \bH \to \KMW_{i-1} \otimes \bH \cong \big(\KMW_i \oplus \KMW_{i-1} \big)
\]
is given for $i$ odd by the matrix (with obvious notations)
\[
\begin{pmatrix} \eta & 0 \\ \Id_{\KMW_i} & 0\end{pmatrix}
\]
and for $i$ even by the matrix 
\[
\begin{pmatrix} 0 & 0 \\ \<-1\> \cdot \Id_{\KMW_i} & \eta \end{pmatrix}.
\]
\end{theorem}

\begin{remark} 
\label{rem H structure}
It immediately follows from Theorem \ref{thm An-0} that the complex $\Ctcell_*(\A^{n+1}-\{0\})$ is a complex of ``free'' $\bH$-modules. Moreover, it is acyclic in degree $\neq 0$ and $n$, because the space $\A^{n+1}-\{0\}$ is $(n-1)$-connected. In other words $\Ctcell_*(\A^{n+1}-\{0\})$ is an $(n+1)$-extension of $\Z$ by $\KMW_{n+1} = \Hcell_n(\A^{n+1}-\{0\})$. One has a chain homotopy equivalence
\[
\Ctcell_*(\A^{n+1}-\{0\}) \otimes_{\bH} \Z  \cong \Ctcell_*(\P^{n}).
\]
\end{remark}

We will need some preliminaries before embarking upon the proof of Theorem \ref{thm An-0}.  Consider the automorphism of the rank $n$ trivial vector bundle over $\G_m$ defined by 
\[
\begin{split}
\Theta_n : \A^{n}\times \G_m &\xrightarrow{\simeq} \A^{n}\times \G_m \\
((\lambda_0,\ldots,\lambda_{n-1}),t) &\mapsto ((\lambda_0 t,\ldots,\lambda_{n-1} t),t). 
\end{split}
\]
and let 
\[
Th(\Theta_n) : \big(\A^{n}/(\A^{n}-\{0\})\big)\wedge (\G_m)_+ \cong \big(\A^{n}/(\A^{n}-\{0\})\big)\wedge  (\G_m)_+ 
\]
denote the associated isomorphism of Thom spaces.  The space $\left(\A^{n}/(\A^{n}-\{0\})\right)\wedge (\G_m)_+$ is clearly $(n-1)$-connected and the K\"unneth morphism 
\begin{equation}
\label{eq A/A-0}
\KMW_n \otimes \bH \xrightarrow{\simeq} \HA_{n}(\big(\A^{n}/(\A^{n}-\{0\})\big)\wedge (\G_m)_+)
\end{equation}
is an isomorphism (of strictly $\A^1$-invariant sheaves) because of the Hurewicz Theorem \cite[\S 6.3]{Morel-book}. Now using the cofibration sequence 
\[
(\A^{n}-\{0\})\times \G_m \to \A^{n}\times \G_m \to \big(\A^{n}/(\A^{n}-\{0\})\big)\wedge (\G_m)_+
\]
and its associated long exact sequence in $\A^1$-homology, we deduce an exact sequence\footnote{It is conjectured \cite[Conjecture 6.34]{Morel-book} that $\HA_{n}(\G_m) = 0$ for $n\geq1$}
\begin{equation}
\label{eq H A/A-0}
\HA_{n}(\G_m) \to \KMW_n \otimes \bH \to \HA_{n-1}((\A^n-\{0\})\times \G_m) \to \HA_{n-1}(\G_m) \to 0.
\end{equation}
The automorphism 
\[
(\A^{n}-\{0\})\times \G_m \xrightarrow{\simeq} (\A^{n}-\{0\}) \times \G_m; \quad ((\lambda_0,\ldots,\lambda_{n-1}),t)\mapsto ((\lambda_0 t,\ldots,\lambda_{n-1} t),t) 
\]
given by restricting $\Theta_n$ induces an automorphism of the exact sequence \eqref{eq H A/A-0} (the action on the factors  $\HA_{n}(\G_m)$ and $\HA_{n-1}(\G_m)$ being trivial). Thus, the induced automorphism
\begin{equation}
\label{eq theta_n}
\theta_n: \KMW_n \otimes \bH \cong \KMW_n \otimes \bH
\end{equation}
is the same as the one induced by $Th(\Theta_n)$ and the identification \eqref{eq A/A-0}.  The projection of $\theta_n$ onto the factor $\KMW_n \otimes \Z = \KMW_{n}$ is clearly the morphism induced by the projection of $\Theta_n$ on the first factor
\begin{equation}
\label{eq Theta_n projection}
(\A^{n}-\{0\})\times \G_m \to \A^{n}-\{0\}; \quad((\lambda_0,\ldots,\lambda_{n-1}),t) \mapsto (\lambda_0 t,\ldots,\lambda_{n-1} t)  
\end{equation}
on $\HA_{n-1}$:
\[
\KMW_n \otimes \bH \to \KMW_n. 
\]
Using this and the fact that $\A^{n}/(\A^{n}-\{0\})$ is canonically isomorphic to $\big(\A^{1}/(\A^{1}-\{0\})\big)^{\wedge n}$ as a $\G_m$-space, we deduce that the above morphism is exactly the morphism defining the $\bH$-module structure on $\KMW_n$ corresponding to the $n$-th tensor product of the one on $\KMW_1$.  We now explicitly describe this $\bH$-module structure on $\KMW_n$.

\begin{lemma}  
\label{lem n.eta} \hfill 
\begin{enumerate}[label=$(\alph*)$]
\item The morphism 
\[
\KMW_n\otimes\bH \to \KMW_n
\]
corresponding to the canonical $\bH$-module structure on $\KMW_n$ is given as follows. Using the decomposition 
\[
\KMW_n\otimes\bH = (\KMW_n\otimes \Z) \oplus (\KMW_n\otimes\KMW_1) = \KMW_n \oplus \KMW_{n+1}, 
\]
on the factor $\KMW_n\otimes\Z = \KMW_n$ it is the identity and on the factor $\KMW_n\otimes\KMW_1 = \KMW_{n+1}$ it is given by 
\[
 n_\epsilon \cdot \eta : \KMW_{n+1} \to \KMW_n
\]
where $n_\epsilon =  \underset{{i=1}}{\overset{n}{\sum}} \<(-1)^{i-1}\> \in \KMW_0(k) = \GW(k)$. In particular, this morphism is $0$ if $n$ is even, and is multiplication by $\eta$ if $n$ is odd.

\item The automorphism of $\KMW_{n+1}\oplus \KMW_{n}$ defined by the commutative diagram
\[
\begin{xymatrix}{
\KMW_n \otimes \bH \ar[d]^{=}\ar[r]^-{\theta_n}& \KMW_n \otimes \bH \ar[d]^{=}\\
\KMW_{n+1}\oplus \KMW_{n} \ar[r] & \KMW_{n+1}\oplus \KMW_{n}
}\end{xymatrix}
\]
is given by multiplication by the matrix
\[
\begin{pmatrix} \<(-1)^n\> & n_\epsilon \cdot \eta \\ 0 & 1 \end{pmatrix}. 
\]
\end{enumerate}
\end{lemma}
 
\begin{proof}
We only outline the argument and leave the details to the reader.
\begin{enumerate}[label=$(\alph*)$]
\item Recall from Remark \ref{rem H} that the canonical $\bH$-module structure on $\KMW_1$ is given by the morphism $\KMW_1 \otimes \bH \to \KMW_1$, which on the factor $\KMW_1 \otimes \Z$ is the obvious homomorphism and is multiplication by $\eta$ on the factor $\KMW_1 \otimes \KMW_1$.  We will use the description of the diagonal $\Psi: \bH\to\bH\otimes \bH$ of $\bH$ and the definition of the $\bH$-module structure on a tensor product.  The $\bH$-module structure on $\KMW_2$ is given by the composition
\[
\begin{split}
\KMW_2 \otimes \bH & \to \KMW_1 \otimes \KMW_1 \otimes \bH  \xrightarrow{\Id \otimes \Psi} \KMW_1 \otimes \KMW_1 \otimes \bH  \otimes \bH \\ 
&  \cong (\KMW_1 \otimes \bH)  \otimes (\KMW_1 \otimes \bH) \to \KMW_1 \otimes \KMW_1 = \KMW_2,
\end{split}
\]
which is easily seen to be identity on the factor $\KMW_2 \otimes \Z$.  On the factor $\KMW_2 \otimes \KMW_1$, it is given by 
\[
\begin{split}
(\lambda, \mu) \otimes (t) \mapsto & ~(\lambda)\otimes(\mu)\otimes \left( (t) \otimes 1 + 1\otimes(t) + (-1)\otimes(t)\right) \\
\mapsto & ~\eta \cdot (\lambda)\cdot (t) \otimes (\mu) +  (\lambda) \otimes \eta \cdot (\mu) \cdot (t) + \eta \cdot (\lambda)\cdot(-1) \otimes \eta \cdot (\mu)\cdot (t) \\
\mapsto & ~\epsilon \cdot \eta \cdot (\lambda)\cdot(\mu)\cdot(t) +  \eta \cdot (\lambda)\cdot(\mu)\cdot(t) + \epsilon \cdot \eta^2 \cdot(-1) \cdot (\lambda)\cdot(\mu)\cdot(t) \\
& ~= \eta \cdot (\lambda) \cdot (\mu) \cdot (t) - \epsilon \cdot \eta \cdot (\lambda) \cdot (\mu) \cdot (t)= 2_{\epsilon} \cdot \eta \cdot (\lambda) \cdot (\mu) \cdot (t), 
\end{split}
\]
where $\epsilon = -\<-1\> \in \KMW_0(k)$ and we have used the $\epsilon$-graded commutativity of Milnor-Witt $K$-theory and the defining relation $\eta^2\cdot (-1) + 2\eta = 0$.
An easy induction now finishes the proof in the general case.

\item We use the fact that $\theta_n$ is induced by the morphism $(\A^{n}-\{0\})\times \G_m \cong (\A^{n}-\{0\})\times \G_m$, which is the product of the morphism $(\A^{n}-\{0\})\times \G_m \to \A^{n}-\{0\}$ given in \eqref{eq Theta_n projection} and of the projection $(\A^{n}-\{0\})\times \G_m \to \G_m$. 
The former morphism on $\HA_{n-1}$ is just the $\bH$-module structure on $\KMW_n$ determined in part $(a)$ above.  Now using the K\"unneth formula, the fact that the diagonal of the Hopf algebra $\HA_{n-1}(\A^n-\{0\}) = \KMW_n$ is given by the formula $x\mapsto x\otimes 1 \oplus 1\otimes x$, and the description of the diagonal of $\HA_0(\G_m) = \bH$, the desired automorphism can be explicitly computed. The details are left to the reader.
\end{enumerate}
\end{proof} 

\begin{remark}
\label{rem important} 
The automorphism $\theta_n$ described in Lemma \ref{lem n.eta} is of order $2$: $\theta_n \circ \theta_n = \Id$. This is easy to check by using the identity $$\<(-1)^n\> \cdot n_\epsilon \cdot \eta +  n_\epsilon \cdot \eta  = 0.$$  Thus, $\theta_n$ is also the automorphism $\HA_{n-1}(\A^{n}-\{0\}\times \G_m) \cong \HA_{n-1}(\A^{n}-\{0\}\times \G_m)$ induced by the morphism 
\[
(\A^{n}-\{0\})\times \G_m \to (\A^{n}-\{0\})\times \G_m; \quad ((\lambda_0,\dots,\lambda_{n-1}),t)\mapsto ((\lambda_0 t^{-1},\dots,\lambda_{n-1} t^{-1}),t). 
\]
This fact will be used in the proof of Theorem \ref{thm An-0} below.
\end{remark}

\begin{proof}[\bf Proof of Theorem \ref{thm An-0}]
We proceed by induction on $n$. The case $n = 0$ is trivial.  
Now we assume that $n\geq 1$ and that the theorem is proven for $\A^{n}-\{0\}$. 

With our notation described at the beginning of this subsection, we have $\tilde{\Omega}_{n-1}(\A^{n+1}-\{0\}) = (\A^n-\{0\})\times \A^1$ and the filtration  $$\emptyset = \tilde{\Omega}_{-1}(\A^{n+1}-\{0\}) \subset \tilde{\Omega}_0 (\A^{n+1}-\{0\}) \subset \cdots \subset \tilde{\Omega}_{n-1}(\A^{n+1}-\{0\}) = (\A^n-\{0\})\times \A^1$$ is exactly the product with $\A^1$ (on the right) of the filtration  
\[
\emptyset = \tilde{\Omega}_{-1}(\A^{n}-\{0\}) \subset \tilde{\Omega}_0 (\A^{n}-\{0\}) \subset \cdots \subset \tilde{\Omega}_{n-1}(\A^{n}-\{0\}) = (\A^n-\{0\}).
\]
Thus, by $\A^1$-invariance of $\A^1$-homology, the truncation $\tau_{\leq n-1}\Ctcell_*(\A^{n+1}-\{0\})$ in degrees $\leq n-1$ equals $\Ctcell_*(\A^{n}-\{0\})$.  By induction, we thus already know the theorem for the oriented cellular $\A^1$-chain complex of $\A^{n+1}-\{0\}$ in degree $\leq n-1$.  So we only need to identify $\Ctcell_n(\A^{n+1}-\{0\})$ and the differential 
\begin{equation}
\label{eq partial_n}
\partial_n : \Ctcell_n(\A^{n+1}-\{0\}) \to \Ctcell_{n-1}(\A^{n+1}-\{0\}).
\end{equation}
Now, $(\A^{n+1}-\{0\} )- \tilde{\Omega}_{n-1}(\A^{n+1}-\{0\}) = \{0\}\times \G_m \subset \A^n \times \G_m$, so by construction $\Ccell_n(\A^{n+1}-\{0\}) = \HA_n(Th(\nu_n))$, where $\nu_n$ is the normal bundle of the closed immersion 
\[
\{0\}\times \G_m\subset \A^{n+1}-\{0\}.
\]
Since the open subset $\A^n \times \G_m\subset \A^{n+1}-\{0\}$ contains $\{0\}\times \G_m$, it defines a trivialization of $\nu_n$ and produces an isomorphism of pointed spaces $Th(\nu_n) \cong (\A^n/(\A^n-\{0\}))\wedge (\G_m)_+$.  Hence, 
\[
\begin{split}
\Ccell_n(\A^{n+1}-\{0\}) = \HA_n(Th(\nu_n)) &= \HA_n(\A^n/(\A^n-\{0\}))\wedge (\G_m)_+) \\ 
&= \HA_{n-1}((\A^n-\{0\})\times \G_m),  
\end{split}
\]
where the last equality follows from the fact that $\A^n-\{0\}$ is $\A^1$-weak equivalent to $S^{n-1} \wedge \G_m^{\wedge n}$.
The desired differential $\partial_n$ in \eqref{eq partial_n}  of the underlying cellular $\A^1$-chain complex is, by construction, given by the composition
\begin{equation}
\label{eq partial_n description}
\begin{split}
&\Ccell_n(\A^{n+1}-\{0\}) = \HA_{n-1}((\A^n-\{0\})\times \G_m)  \to \HA_{n-1}(\tilde{\Omega}_{n-1})
= \HA_{n-1}(\A^n-\{0\}) \\
&\to \HA_{n-1}((\A^n-\{0\})/\tilde{\Omega}_{n-2}(\A^n-\{0\})) = \Ccell_{n-1}(\A^n-\{0\})  = \Ccell_{n-1}(\A^{n+1}-\{0\}).
\end{split}
\end{equation}
This follows from the fact that $\tilde{\Omega}_{n-1} \cap \left(\A^n \times \G_m \right) = (\A^n-\{0\}) \times \G_m$ and from the induced commutative diagram
\[
\minCDarrowwidth12pt
\begin{CD}
(\A^n-\{0\}) \times \G_m @>>>  \A^n \times \G_m @>>> \dfrac{\A^n}{\A^n-\{0\}}\wedge (\G_m)_+ @>>> S^1 \wedge \big((\A^n-\{0\}) \times \G_m\big)_+  \\
@VVV @VVV @VVV @VVV \\
\tilde{\Omega}_{n-1} @>>> \A^{n+1}-\{0\} @>>> Th(\nu_n) @>>> S^1 \wedge ((\tilde{\Omega}_{n-1})_+).
\end{CD} 
\]
Now, the composition $\HA_{n-1}((\A^n-\{0\})\times \G_m)  \to \HA_{n-1}(\A^n-\{0\}) = \KMW_n$ of the first row in \eqref{eq partial_n description} is just induced by the projection $(\A^n-\{0\})\times \G_m \to \A^n-\{0\}$.

However, in order to compute the differential $\partial_n : \Ctcell_n(\A^{n+1}-\{0\}) \to \Ctcell_{n-1}(\A^{n+1}-\{0\})$ of the oriented cellular $\A^1$-chain complex, recall that we need to use the induced trivialization from $\P^n$ described in Example \ref{ex proj space} and Lemma \ref{lemma torsor}.  If the $X_i$ for $i\in\{0,\dots,n\}$ are the $n+1$-coordinate functions on $\A^{n+1}-\{0\}$, then the normal bundle $\nu_n$ is oriented by the functions $\frac{X_i}{X_n}$ for $i\in\{0,\dots, n-1\}$ (which produces the isomorphism of pointed spaces $Th(\nu_n) \cong (\A^n/(\A^n-\{0\})\wedge ((\G_m)_+)$). Thus, we have to compose the previous identifications with the isomorphism of vector bundles over $\G_m$
\[
\A^n\times \G_m \to \A^n \times \G_m; \quad ((\lambda_0,\ldots,\lambda_{n-1}),t)\mapsto (\lambda_0 t,\ldots,\lambda_{n-1} t),
\]
which takes the corrected trivialization given by the functions $\frac{X_i}{X_n}$ to the one coming from the functions $X_i$, $i\in\{0,\dots, n-1\}$.  Therefore, the differential $\partial_n$ in the oriented cellular $\A^1$-chain complex \eqref{eq partial_n} is the composition 
\[
\HA_{n-1}((\A^n-\{0\})\times \G_m) \xrightarrow{\simeq} \HA_{n-1}((\A^n-\{0\})\times \G_m) \to \HA_{n-1}(\A^n-\{0\}) \to \Ccell_{n-1}(\A^n-\{0\}), 
\]
where the leftmost isomorphism is induced by the above automorphism of $\A^n \times \G_m$.  Now, the composition  
\[
\KMW_{n+1}\otimes \bH \to \HA_{n-1}((\A^n-\{0\})\times \G_m) \to \HA_{n-1}(\A^n-\{0\}) = \KMW_n
\]
is precisely the morphism described in the Lemma \ref{lem n.eta}, where the leftmost arrow is the one in \eqref{eq H A/A-0}.  The canonical morphism $\KMW_n = \HA_{n-1}(\A^{n}-\{0\})\to \Ctcell_{n-1}(\A^{n}-\{0\})$ is given as follows: 
\[
\KMW_n \xrightarrow{\begin{pmatrix} \<-1\> & \eta \end{pmatrix}}   \KMW_n \oplus \KMW_{n-1},
\]
if $n$ is even and 
\[
\KMW_n \xrightarrow{\begin{pmatrix} \Id & 0 \end{pmatrix}}   \KMW_n \oplus \KMW_{n-1},
\]
if $n$ is odd.  To see this, first use the fact that $\KMW_n = \HA_{n-1}(\A^{n}-\{0\})\to \Ctcell_{n-1}(\A^{n}-\{0\})$ is the composition $$\KMW_n = \HA_{n-1}(\A^{n}-\{0\})\to \HA_{n-1}(Th(\nu_{n-1})) \cong \HA_{n-1}((\A^{n-1}/(\A^{n-1}-\{0\})\wedge ((\G_m)_+)),$$ in which the rightmost isomorphism is induced by the trivialization coming from the functions $\frac{X_i}{X_n}$ for $i\in\{0,\dots, n-2\}$.  If one uses the trivialization coming from the functions $X_i$ instead, then it is easy to see that
\begin{equation}
\label{eq KMW to H}
\KMW_n \to \HA_{n-1}((\A^{n-1}/(\A^{n-1}-\{0\})\wedge ((\G_m)_+)) = \KMW_n\oplus \KMW_{n-1} 
\end{equation}
is the canonical inclusion.  Thus, in order to compute the morphism in \eqref{eq KMW to H} for the trivialization coming from the functions $\frac{X_i}{X_n}$, we need to compose the morphism \eqref{eq KMW to H} with the automorphism 
\[
\KMW_n\oplus \KMW_{n-1} = \HA_{n-2}(\A^{n-1}-\{0\}\times \G_m) \cong \HA_{n-2}(\A^{n-1}-\{0\}\times \G_m) = \KMW_n\oplus \KMW_{n-1}
\]
induced by $((\lambda_0,\dots,\lambda_{n-2}),t)\mapsto ((\lambda_0 t^{-1},\dots,\lambda_{n-2} t^{-1}),t)$.  This morphism is computed in Lemma \ref{lem n.eta} $(b)$ and Remark \ref{rem important} above.  Using the induction, it is now easy to conclude the proof of the theorem.
\end{proof}

\begin{remark}
If we let $n$ go to infinity in the statement of Theorem \ref{thm An-0}, one obtains the oriented cellular $\A^1$-chain complex of the oriented cellular (ind-)space $\A^\infty-\{0\}$, which is unbounded below with the differentials given by the same formula as in Theorem \ref{thm An-0}. It is periodic of period $2$, and it is a free $\bH$-resolution of $\Z$.  This suggests that $\G_m$ behaves like a ``cyclic group of order $2$''.  Indeed, consider the standard ``small'' free resolution of $\Z$ by free $\Z[\Z/2]$-modules of rank $1$:
\[
\cdots \Z[\Z/2] \stackrel{\partial_n}{\to} \Z[\Z/2] \to \cdots \to \Z[\Z/2] \stackrel{\partial_2}{\to} \Z[\Z/2] \stackrel{\partial_1}{\to} \Z[\Z/2] \to \Z
\]
where $\Z[\Z/2]$ is the group ring of $\Z/2$, with $\sigma$  the nontrivial element of $\Z/2$, and 
\[
\partial_n =
\begin{cases}
1-\sigma, &\text{ if $n$ is odd; and}\\
1+\sigma, &\text{ if $n$ is even.} 
\end{cases}
\]
Here $\Z[\Z/2]$ replaces $\bH$ (in fact, it is the topological real realization of $\bH$). The augmentation ideal $\Z(\Z/2)\subset \Z[\Z/2]$ of $\Z[\Z/2]$ is now the free subgroup with generator $[1]-[\sigma]$ (with obvious notations). The decomposition $\Z[\Z/2] = \Z\oplus \Z(\Z/2)$ is analogous to $\bH = \Z\oplus \KMW_1$ and the differential 
\[
\partial_1 : \Z(\Z/2) \oplus \Z \xrightarrow{1 - \sigma} \Z(\Z/2) \oplus \Z
\]
can be interpreted with respect to this analogy: the summand $\Z$ is mapped isomorphically to $\Z(\Z/2)$ by $1\mapsto [1] - [\sigma]$, and $\Z(\Z/2)$ maps to itself through the multiplication by $2$ which replaces here $\eta$. The differential 
\[
\partial_2 : \Z(\Z/2) \oplus \Z \xrightarrow{1 + \sigma} \Z(\Z/2) \oplus \Z
\]
maps the factor $\Z(\Z/2)$ to $0$ and since $$(1+\sigma)[1] = [1] + [\sigma] = 2 - ([1]-[\sigma]),$$ we see that $[1]$ maps to $-([1]-[\sigma]) + 2 \cdot [1]$. Thus, the topological real realization of the complex $\Ctcell_*(\A^\infty-\{0\})$ is precisely the $\Z[\Z/2]$-resolution above (observe that the topological real realizations of the $\KMW_n$'s appearing in Theorem \ref{thm An-0} are $\Z$ with trivial $\Z/2$-action).

Of course, it is also possible to interpret the above computations through a complex embedding by replacing $\G_m$ with $S^1$ and $\bH$ with the singular homology group $H_*^{\rm sing}(S^1)$.  We leave the details to the reader.
\end{remark}

As a consequence of Theorem \ref{thm An-0}, we obtain the oriented cellular $\A^1$-chain complex of the projective space $\P^n$ with the canonical oriented cellular structure described above.

\begin{corollary}
\label{cor projective space} 
For any integer $n\geq 1$, the oriented cellular $\A^1$-chain complex $\Ctcell_*(\P^n)$ has the form
\[
\KMW_n \xrightarrow{\partial_n} \cdots \xrightarrow{\partial_{i+1}} \KMW_i \xrightarrow{\partial_i} \cdots \to \KMW_1 \xrightarrow{\partial_1} \Z
\]
with
\[
\partial_i=
\begin{cases}
0, &\text{ if $i$ is odd; and}\\
\eta, &\text{ if $i$ is even.} 
\end{cases}
\]
It follows that
\[
\Hcell_i(\P^n)=
\begin{cases}
\Z, &\text{ if $i = 0$;}\\
\KMW_i/\eta, &\text{ if $i<n$ is odd;} \\ 
{}_\eta\KMW_i, &\text{ if $i<n$ is even;}
\end{cases}
\]
and
\[
\Hcell_n(\P^n)=
\begin{cases}
\KMW_n, &\text{ if $n$ is odd;}\\
{}_\eta\KMW_n, &\text{ if $n$ is even.}
\end{cases} 
\]
\end{corollary}

\begin{proof} 
Since the oriented cellular structure of $\A^{n+1}-\{0\}$ is obtained by pulling back that of $\P^n$ (also see Remark \ref{rem H structure}), we observe that $\Ctcell_*(\P^n)$ is the quotient of $\Ctcell_*(\A^{n+1}-\{0\})$ by the action of $\bH$:
\[
\Ctcell_*(\A^{n+1}-\{0\}) \otimes_{\bH} \Z  \cong \Ctcell_*(\P^{n}).
\]
The corollary now follows by the description of the differentials in $\Ctcell_*(\A^{n+1}-\{0\})$ determined in Theorem \ref{thm An-0}.
\end{proof}

\begin{remark} 
Clearly, the Suslin homology version of the cellular $\A^1$-chain complex of $\P^n$ is the complex with differential $0$ everywhere:
\[
\CScell_*(\P^n): \KM_n \stackrel{0}{\to} \dots \to \KM_i \stackrel{0}{\to} \dots \to \KM_1 \stackrel{0}{\to} \Z,
\]
whose homology is $\HScell_i(\P^n) = \KM_i$ for every $i$. Observe that the cellular Suslin homology of $\P^n$ is much simpler than the Suslin homology of $\P^n$ (which is unknown in general).
\end{remark}


We end this section by giving an application of the abstract characterization of cellular $\A^1$-chain complexes discussed in Section \ref{subsection intrinsic cellular chain complex} and the above computations to degree of morphisms from $\P^n \to \P^n$.

Let $n$ be an odd integer. For any morphism $f: \P^n\to \P^n$ in the $\A^1$-homotopy category $\sH(k)$, we may define its degree $\deg(f)\in \GW(k) = \KMW_0(k)$ as follows. By Lemma \ref{lemma cellular functoriality}, one may define a canonical homotopy class of morphism of cellular $\A^1$-chain complexes $[f]^{\rm cell}_*:\Ccell_*(\P^n)  \rightarrow  \Ccell_*(\P^n)$ and $\Hcell_n(f)$ is an endomorphism of sheaves of $\Hcell_n(\P^n)$. Since $n$ is odd, $\Hcell_n(\P^n) = \KMW_n$ by Corollary \ref{cor projective space}.  Since $\Hom_{Ab_{\A^1}(k)}(\KMW_n,\KMW_n) = \KMW_0(k) = \GW(k)$, it follows that $\Hcell_n(f)$ defines an element of $\GW(k)$. 

\begin{definition} The element of $\GW(k)$ defined above for any morphism $f: \P^n\to \P^n$ for odd $n$ in the $\A^1$-homotopy category $\sH(k)$ is denoted by $\deg(f)$ and is called the \emph{degree} of $f$.
\end{definition}

\begin{remark}
The possibility of defining the degree of $f$ as above is related to the fact that $\P^n$, for $n$ odd, is orientable. This construction can be easily generalized to other examples of cellular orientable smooth projective $k$-schemes, for instance obtained by convenient blow-up of $\P^n$'s.
\end{remark}

\section{Cellular \texorpdfstring{$\A^1$}{A1}-chain complexes associated with the Bruhat decomposition}
\label{section Bruhat decomposition}

Unless otherwise specified, we will fix and use the following notations for the rest of the article.  The base field will be denoted by $k$ and will be assumed to be perfect.

\begin{longtable}{ll}
$G$ & a split reductive (connected) algebraic group over $k$ \\
$G_{\rm der}$ & the derived group of $G$ \\
 & (it is a normal, semisimple subgroup scheme of $G$ with trivial radical) \\
$G_{\rm sc}$ & the simply connected central cover of $G_{\rm der}$ in the sense of algebraic groups\\
$T$ & a split maximal torus of $G$ \\
$B$ & a Borel subgroup of $G$ containing $T$ \\
$W$ & the Weyl group $N_G(T)/T$ of $G$ \\
$w_0$ & the element of longest length in $W$ \\
$X$ & the character group $\Hom(T, \G_m)$ of $T$ \\
$Y$ & the cocharacter group $\Hom(\G_m, T)$ of $T$\\
$\Phi\subset X$ & the set of roots \\
$\Phi^+$ & the set of positive roots corresponding to $B$\\
$\Phi^{-}$ & the complement of $\Phi^+$ in $\Phi$\\
$\Delta\subset\Phi^+$ & the set of simple roots (or a basis) of $\Phi^+$\\
$\alpha^{\vee}\in Y$ & the coroot corresponding to a root $\alpha$; it satisfies $\langle \alpha, \alpha^{\vee}\rangle = 2$ \\
$\Phi^\vee\subset Y$ & the set of coroots\\
$\varpi_{\alpha}$ & the dual of $\alpha^{\vee}$ corresponding to the inner product $\langle , \rangle$\\ &(that is, the fundamental weight corresponding to $\alpha$) \\
$s_{\alpha}$ & the element of $W$ corresponding to the reflection with respect to $\alpha$ \\
$U_{\alpha}$ & the root subgroup corresponding to $\alpha$ \\
$\mathfrak{u}_\alpha$ & the tangent space of $U_\alpha$ at the neutral element (or the Lie algebra of $U_{\alpha}$)
\\
& (it is non-canonically isomorphic to $\G_a$, on which $T$ acts through $\alpha$) \\
$U$ & the product of the $U_\alpha$'s for $\alpha\in \Phi^+$ (in any order); $B$ is the product $T U$\\
$S_{\alpha}$ & the subgroup of $G$ generated by $\alpha^{\vee}(\G_m)$, $U_{\alpha}$ and $U_{-\alpha}$ \\
& ($S_{\alpha} \simeq \SL_2$ or $\PGL_2$; we always have $S_{\alpha} \simeq \SL_2$ if $G$ is simply connected)\\
$P_{\lambda}$ & the parabolic subgroup of $G$ corresponding to a cocharacter $\lambda$ \\
\end{longtable}

The aim of this section is to explicitly describe a cellular structure on any split reductive group $G$ over $k$ and give some information about the associated cellular $\A^1$-chain complex. We will freely use the standard results on the associated root datum $(X,\Phi,Y,\Phi^\vee)$ and Bruhat decomposition (see \cite[Chapter 7 and 8]{Springer}, for example). 
For any $w\in W$ we choose a lifting $\dot{w}$ of $w$ in $N_G(T)$. The Bruhat decomposition of $G$ is the partition
\[
G = \underset{w \in W}\coprod B\dot{w}B.
\]
Observe that $B\dot{w}B$ doesn't depend on the choice of $\dot{w}$. This induces the stratification by open subsets
\begin{equation}
\label{eqn bruhat dec}
\emptyset = \Omega_{-1} \subsetneq \Omega_0 \subsetneq \Omega_1 \subsetneq \cdots \subsetneq \Omega_{\ell} = G, 
\end{equation}
where $\ell$ is the length of the longest element $w_0$ in $W$ and for every $i$, one has
\[
\Omega_i - \Omega_{i-1}= \underset{\ell(w) = \ell-i}{\coprod} B \dot{w}B.
\]
Another way to define $\Omega_i$ is to observe that
\[
\Omega_i = \underset{w\in W; ~ \ell(w)\geq \ell-i}{\bigcup} \dot{w}(\Omega_0),
\]
where $\dot{w}(\Omega_0)$ denotes the translate of $\Omega_0$ by $\dot{w}$.  It is well known (see Section \ref{subsection complex in rank 1} below for a precise recollection) that the $k$-schemes $B \dot{w}B$ are isomorphic to a product of an affine space and $T$. Thus, these are affine, smooth and cohomologically trivial and consequently, we obtain a cellular structure on $G$.

\begin{definition}
\label{def cell G} 
For a split reductive (connected) algebraic $k$-group $G$, with fixed split maximal torus $T$ and Borel subgroup $T\subset B\subset G$, we call the above cellular structure the \emph{canonical} or \emph{Bruhat cellular structure} on $G$.
\end{definition}

For our computations, we will need to orient this cellular structure. We first treat the case where the semi-simple rank of $G$ is one.

\subsection{The rank 1 case} \hfill 
\label{subsection complex in rank 1}

Let $G$ be a split semisimple group of rank $1$, $T$ be a split maximal torus of $G$ and $B$ be a Borel subgroup of $G$ containing $T$.  In this case, the Weyl group $W =  N_G(T)/T = \langle s_{\alpha} \rangle \simeq \Z/2$ has order $2$, where the generating reflection $s_{\alpha}$ corresponds to the positive simple root $\alpha: T \to \G_m$. We denote the corresponding coroot as usual by $\alpha^{\vee}: \G_m \to T$. We also denote by $U_\alpha$ and $U_{-\alpha}$ the root subgroups.  Let $\dot{s}_{\alpha}\in N_G(T)$ be an element whose class in $W$ is $s_\alpha$.

The Bruhat decomposition of $G$ has the form $G = B\dot{s}_{\alpha}B ~ \amalg ~ B$ and induces the increasing filtration of $G$ by open subsets of the form:
\[
\Omega_0 \hookrightarrow \Omega_1 = G
\]
with $\Omega_0 = B \dot{s}_{\alpha} B = U_\alpha\dot{s}_{\alpha}TU_\alpha$.  Its complement $G - \Omega_0$ with the reduced induced subscheme structure is $B$.  The morphisms induced by the product in $G$ 
$$U_\alpha\times T\times U_\alpha \rightarrow B\dot{s}_{\alpha}B; \quad (\lambda,t,\mu)\mapsto \lambda \dot{s}_{\alpha} t \mu$$ 
and 
$$T\times U_\alpha \rightarrow B, (t,\mu)\mapsto t \mu$$ 
are isomorphisms of $k$-schemes. We then observe that the open subscheme 
\[
\dot{s}_{\alpha} (\Omega_0) = U_{-\alpha}  T U_\alpha 
\]
contains $B$; it is thus an open neighborhood of $B= T U_\alpha$.  Clearly, the normal bundle of $B \subset G$ is $\mathfrak{u}_{-\alpha}\times B$, where we let $\mathfrak{u}_{-\alpha}$ denote the tangent space of $U_{-\alpha}$ at $0$, which is a $1$-dimensional $k$-vector space (in other words, $\mathfrak{u}_{-\alpha}$ is the Lie algebra of $U_{-\alpha}$).  This identification is independent of the choice of $\dot{s}_{\alpha}$.

Giving a trivialization of the normal bundle of $B$ in $G$ is thus equivalent to giving an isomorphism $\Ga\cong U_{-\alpha}$ of $k$-groups, or equivalently, a nonzero element of $\mathfrak{u}_{-\alpha}$. By \cite[exp. XX, Th\'eor\`eme 2.1]{SGAIII}, the root groups $U_{\alpha}$ and $U_{-\alpha}$ are canonically in duality; and consequently, giving an isomorphism $u_{-\alpha}: \Ga\cong U_{-\alpha}$ is equivalent to giving an isomorphism $u_{\alpha}:\Ga\cong U_{\alpha}$.  This constitutes a ``pinning'' of $G$ in the sense of \cite[XXIII \S 1 \'Epinglages]{SGAIII} and \cite{Demazure}. 
As $G$ is the union of the two open subschemes $\Omega_0$ and $\dot{s}_{\alpha} (\Omega_0)$, we get a canonical commutative diagram of smooth $k$-schemes and spaces 
\begin{equation}
\label{diagram rank 1}
\begin{xymatrix}{
\Omega_0 \ar[r] & G \ar[r]& G/\Omega_0 \ar[r]& S^1 \wedge ((\Omega_0)_+)\\
\Omega_0 \cap \dot{s}_{\alpha} (\Omega_0) \ar[u] \ar[r] & \dot{s}_{\alpha} (\Omega_0) \ar[u] &  & 
}
\end{xymatrix} 
\end{equation}
in which the top row is a cofiber sequence.  Observe that the bottom row in the above diagram does not depend on the choice of $\dot{s}_{\alpha} \in N_G(T)$.  By the motivic homotopy purity theorem \cite[Theorem 2.23, page 115]{Morel-Voevodsky}, $G/\Omega_0$ is canonically $\A^1$-weakly equivalent to the Thom space of the vector bundle $U_{-\alpha}\times B$ over $B$, that is, $U_{-\alpha}/(U_{-\alpha}^\times)\wedge (B_+)$. Using the pinning and the fact that $B\rightarrow T$ is an $\A^1$-weak equivalence, we see that $G/\Omega_0$ is canonically weakly $\A^1$-equivalent to $\A^1/(\A^1-\{0\})\wedge (T_+) \cong S^1\wedge \G_m \wedge (T_+)$. The cellular $\A^1$-chain complex of $G$ for the chosen (strict) orientation thus has the form
\begin{equation}
\label{eqn Ccell SL2}
\Ccell_*(G): \ZA(\G_m) \otimes \ZA[T] \xrightarrow{\partial} \ZA[\dot{s}_{\alpha}T].
\end{equation}
Now we explicitly compute the differential. At this point, in order to identify $\ZA[\dot{s}_{\alpha}T]$ with $\ZA[T]$, we need to specify an explicit choice of lifting of $s_\alpha$.  By \cite[XX, 3.1]{SGAIII}, given a pinning $u_\alpha$ (and the  corresponding $u_{-\alpha}$), the element 
\begin{equation}\label{eq:walpha}
w_\alpha(X): = u_{\alpha}(X)u_{-\alpha}(-X^{-1})u_\alpha(X)
\end{equation}
belongs to $N_G(T)$ and is a lifting of $s_\alpha$, for every unit $X$ in $k$.  

\begin{convention}
\label{convention lifting}
We will take $X = -1$, and $\dot{s}_\alpha = w_\alpha(-1)$ for our choice of the lift of $s_{\alpha} \in W$ to $N_G(T)$ in what follows. 
\end{convention}

\begin{example}
For $G = \SL_2$ and its standard pinning, we have $\dot{s}_\alpha = \begin{pmatrix} 0 & -1 \\ 1 & 0\end{pmatrix}$. 
\end{example}

With the choice of $\dot{s}_{\alpha}$ as in Convention \ref{convention lifting}, the isomorphism $\dot{s}_{\alpha}T \cong T$ of smooth $k$-schemes induces an isomorphism of sheaves between $\ZA[\dot{s}_{\alpha}T]$ and $\ZA[T]$.  It is easy to see that $\Omega_0 \cap \dot{s}_{\alpha} (\Omega_0) \subset \dot{s}_{\alpha} (\Omega_0) = U_{-\alpha}  T U_\alpha$ is exactly the open subscheme $U_{-\alpha}^\times  T U_\alpha$.  

\begin{lemma}\label{lem:comp} 
Under the above identifications, the vertical morphism in Diagram \eqref{diagram rank 1} 
\[
U_{-\alpha}^\times  T U_\alpha \cong \Omega_0 \cap \dot{s}_{\alpha} (\Omega_0) \subset \Omega_0 = U_\alpha \dot{s}_{\alpha} T U_\alpha                                                           
\]
is the right $TU_\alpha$-equivariant morphism induced by the morphism: 
\[
U_{-\alpha}^\times \to \Omega_0; \quad \lambda\mapsto \lambda^{-1} \cdot \dot{s}_{\alpha} \cdot \alpha^\vee(\lambda) \cdot \lambda^{-1}.
\]
\end{lemma}

\begin{proof} By \cite{SGAIII}, there is a unique morphism of $k$-groups $ \Phi_{u_\alpha}: \SL_2\rightarrow G$ which is compatible with the pinnings on both sides.  Recall that in $\SL_2$, the root subgroup $U_\alpha$ is the subgroup of matrices of the form
\[
\begin{pmatrix} 1 & \lambda \\ 0 & 1\end{pmatrix}
\]
and $U_{-\alpha}$ is that of matrices of the form
\[
 \begin{pmatrix} 1 & 0 \\ \lambda & 1\end{pmatrix}.
\]
The induced morphism $\G_m\rightarrow T$ is then the coroot 
\[
 \alpha^\vee: \G_m\to \SL_2; \quad u\mapsto\begin{pmatrix} u & 0 \\ 0 & u^{-1}\end{pmatrix}
\]
and the matrix $\begin{pmatrix} 0 & -1 \\ 1 & 0\end{pmatrix}$ is sent precisely to $\dot{s}_\alpha$ by $\Phi_{u_{\alpha}}$. 
Thus, it suffices to prove the lemma for the standard pinning on $\SL_2$.  The proof now follows from the straightforward computation 
\[
\begin{pmatrix} 1 & \lambda^{-1} \\ 0 & 1\end{pmatrix} \cdot \begin{pmatrix} 0 & -1 \\ 1 & 0\end{pmatrix} \cdot \begin{pmatrix} \lambda & 0 \\ 0 & \lambda^{-1}\end{pmatrix} \cdot \begin{pmatrix} 1 & \lambda^{-1} \\ 0 & 1\end{pmatrix} = \begin{pmatrix} 1 & 0 \\ \lambda & 1\end{pmatrix}. 
\]
\end{proof}

\begin{remark} \label{rem:dots} \hfill
\begin{enumerate}
\item The proof of Lemma \ref{lem:comp} is sensitive to the choice of $\dot{s}_{\alpha}$ made in Convention \ref{convention lifting}. One may verify that choosing $w_\alpha(1)$ instead of $w_\alpha(-1)$ leads to an inclusion up to a sign.

\item Observe also that a change of the pinning of $G$ amounts to multiplying $u_\alpha$ by a unit $\rho$ (and $u_{-\alpha}$ by $\rho^{-1}$).  Consequently, it amounts to multiplying $\dot{s}_{\alpha} = w_\alpha(-1)$ also by the image of $\rho$ in $T$ under $\alpha^\vee:\G_m\to T$.

\item Choosing a pinning $u_\alpha$ of $G$ induces by the formula \eqref{eq:walpha} for $X = -1$ a choice of a lift $\dot{s}_\alpha$ in $N_G(T)$. We observe that the converse holds. Choose a lift $\dot{s}_\alpha$ in $N_G(T)$. We see that the rational map $\dot{x}:G\dashrightarrow \P^1$ corresponding to the projection $\Omega_0 = U_\alpha\dot{s} T U_\alpha\to T$ followed by the dual character $\varpi:T\to \G_m$ of $\alpha^\vee$ can be used to define a pinning using the rational map 
\[
\dot{\gamma} = \frac{\dot{x}}{\dot{x}(\dot{s}_\alpha. )}:G\dashrightarrow \P^1,
\]
which induces an isomorphism $U_\alpha\cong \A^1\subset \P^1$. In the case of $G = \SL_2$, this rational map is given by $\dfrac{b}{a}$, where the elements of $\SL_2$ are written $\begin{pmatrix} a & c \\ b & d\end{pmatrix}$.  One easily verifies that this defines a bijection between the sets 
\[
\{\text{Lifts of $s_{\alpha}$ in $N_G(T)$ }\} \xrightarrow{\simeq} \{\text{Pinnings of $G$}\}
\]
with inverse sending $u_\alpha$ to $\dot{s}_\alpha$. Observe also that both these sets admit an action of $k^\times$, which is preserved by the above bijection.
\end{enumerate}
\end{remark}

Using Lemma \ref{lem:comp}, it follows that the cellular $\A^1$-chain complex (\ref{eqn Ccell SL2}) of $G$ has the form
\begin{equation}
\label{eq:descriptionpartial}
\ZA(\G_m)\otimes\ZA[T]  \xrightarrow{\partial} \ZA[\dot{s}_\alpha T]
\end{equation}
with the differential $\partial$ given by the composition 
\[
\ZA(\G_m)\otimes\ZA[T]  \subset \ZA[\G_m\times T] \xrightarrow{\alpha^\vee \times \Id_T} \ZA[T \times T] \xrightarrow{\mu} \ZA[T] \to \ZA[\dot{s}_{\alpha} T],
\]
where $\mu$ is induced by the product $\mu: T\times T \rightarrow T$ and the rightmost arrow is the isomorphism of sheaves induced by the product on the left with $\dot{s}_\alpha$. We will use this formula several times.

\begin{example}
We now assume that $G$ is simply connected; in other words, $\Phi_{u_\alpha}: \SL_2 \stackrel{\simeq}{\to} G$ is an isomorphism. Then the coroot $\alpha^\vee: \G_m \cong T$ is also an isomorphism. Under the identifications $\ZA[T]=\KMW_1 \oplus \Z$ and $\ZA(\G_m)\otimes\ZA[T]=\KMW_2 \oplus \KMW_1$,  the cellular $\A^1$-chain complex of $G$ with respect to the pinning of $G$ given by $u_{\alpha}$ is isomorphic to 
\begin{equation}
\label{eqn2 Ccell SL2}
\Ctcell_*(\SL_2): \KMW_2 \oplus \KMW_1 \xrightarrow{\begin{pmatrix} \eta & 0 \\ \Id & 0 \end{pmatrix}} \KMW_1 \oplus \Z,
\end{equation}
because $\mu$ is the direct sum of $\eta$ and the identity.  As a consequence, we obtain $$\Hcell_0(G) = \HA_0(G) = \Z$$ and $$\Hcell_1(G) = \HA_1(G) = \KMW_2,$$ where $\KMW_2$ is identified as the subsheaf of $\KMW_2 \oplus \KMW_1$ given by the image of the morphism 
\[
\KMW_2\rightarrow \KMW_2 \oplus \KMW_1; \quad \alpha \mapsto (\alpha,-\eta  \alpha). 
\] 
\end{example}

\begin{example} Now assume that $G$ is not simply connected; in other words, $\Phi_{u_\alpha}: \SL_2 \rightarrow G$ has kernel $\mu_2$ and induces an isomorphism $\PGL_2 \cong G$. Then the root $\alpha: T \cong \G_m$ is an isomorphism and the coroot $\alpha^\vee: \G_m \rightarrow T \cong \G_m$ is the squaring map. In this case, the cellular $\A^1$-chain complex of $G$ is given by 
\begin{equation}
\label{eqn3 Ccell SL2}
\Ctcell_*(\SL_2): \KMW_2 \oplus \KMW_1 \xrightarrow{\begin{pmatrix} 0 & 0 \\ h & 0 \end{pmatrix}} \KMW_1 \oplus \Z,
\end{equation}
because the morphism $\ZA(\G_m)\to \ZA(\G_m)$ induced by the squaring map $\G_m\to \G_m$ is multiplication by $h = 1 + \<-1\> \in \KMW_0(k)$, (see \cite[Lemma 3.14 p. 55]{Morel-book}) and $\eta  h = 0$. Therefore, in this case we obtain $$\Hcell_0(G) = \HA_0(G) = \Z \oplus \KMW_1/h = \KMW_0,$$ since $\KMW_1/h = \mathbf I$, the fundamental ideal sheaf in $\KMW_0$ and $$\HA_1(G) \twoheadrightarrow \Hcell_1(G) = \KMW_2 \oplus \ _h\KMW_1,$$ where ${}_h\KMW_1$ denotes the kernel of the multiplication by $h$ on $\KMW_1$.  Observe that as $G \cong \PGL_2$ is not $\A^1$-connected, $\HA_1(G)$ does not agree with $\piA_1(G)$.  Observe also that $\piA_0(\PGL_2) = \KM_1/2$ and that $\HA_0(\PGL_2)$ is indeed the free strictly $\A^1$-invariant sheaf on the sheaf of sets $\KM_1/2$ (see \cite[Theorem 3.46]{Morel-book}).
\end{example}

\begin{remark}\label{rem pin} \hfill 
\begin{enumerate}
\item The above description of the oriented cellular $\A^1$-chain complex of $G$ depends on the choice of a pinning of $G$, since we used the explicit lift of $s_\alpha$ in $N_G(T)$ defined in Convention \ref{convention lifting} to identify $\HA_0(\Omega_0)$ with $\ZA[\dot{s}_{\alpha}T]$.  If we multiply the standard pinning of $\SL_2$ by $u\in k^\times$, the element $w_\alpha$ becomes $\begin{pmatrix} 0 & -u \\ u^{-1} & 0\end{pmatrix}$ and the identification $\dot{s}_{\alpha}T$ with $T$ has to be modified by multiplication by $u$.  Now, an easy computation using the results of \cite{Morel-book} shows that multiplication by a square on $T$ induces the identity isomorphism on $\ZA[T]$. Thus, it follows that the above description of the oriented cellular $\A^1$-chain complex of $G$ in rank $1$ only depends on the choice of a pinning ``up to square".  This corresponds to an orientation of the normal bundle of $B\subset G$ as opposed to its trivialization.

\item However, if one chooses an orientation of the normal bundle of $B\subset G$ and a lifting of $s_\alpha$ in $N_G(T)$ which is not given by the above formula, the oriented cellular $\A^1$-chain complex will a priori be different (although isomorphic).

\item The automorphisms of $\SL_2$ coming from the action of $\PGL_2$ are given, for every $\gamma\in k^\times$ by
\[
\begin{pmatrix} a & c \\ b & d\end{pmatrix} \mapsto  \begin{pmatrix} a & \gamma^{-1} c \\ \gamma b & d\end{pmatrix}.
\]
Such an automorphism only preserves the oriented cellular structure if $\gamma$ is a square. In general, this automorphism is a cellular (auto)morphism and the induced automorphism of the oriented cellular $\A^1$-chain complex is the multiplication by $\<\gamma\>$ on $\ZA(\G_m) = \KMW_1$ (coming from the orientation).

\item The above computations can be used to describe the cellular $\A^1$-chain complex of a split reductive $k$-group of semisimple rank $1$. Let ${\rm rad}(G)$ denote the radical of $G$, which is a (split) subtorus of $T$ contained in the center of $G$. The oriented cellular $\A^1$-chain complex of $G_{\rm der}$ with the chosen pinning has the form
\[
\KMW_1\otimes\ZA[T] \xrightarrow{\partial} \ZA[\dot{s}_\alpha T].
\]
We know that the intersection ${\rm rad}(G) \cap G_{\rm der}$ is a finite subgroup of the center of $G_{\rm der}$.  If $G_{\rm der}$ is adjoint (that is, isomorphic to $\PGL_2$), then the center of $G_{\rm der}$ is trivial and $G$ is the product group ${\rm rad}(G) \times G_{\rm der}$.  If $G_{\rm der}$ is simply connected (that is, isomorphic to $\SL_2$), then there are two cases. If ${\rm rad}(G) \cap G_{\rm der}$ is trivial, then again $G = {\rm rad}(G) \times G_{\rm der}$ as a group. If ${\rm rad}(G) \cap G_{\rm der} = \mu_2 = Z(\SL_2)$, then $G$ is isomorphic to ${\rm rad}(G)\times_{\mu_2} G_{\rm der}$ as a group. Observe that $\G_m\times_{\mu_2} {\rm rad}(G) \to T$ is an isomorphism of groups, since $\mu_2$ is also the intersection of $\G_m$ (the maximal torus in $G_{\rm der}$ given by the coroot) and ${\rm rad}(G)$. Thus $G_{\rm der} \times_{\G_m} T \to G$ is an isomorphism of $k$-schemes with right and left $T$-action. Now, if one chooses a split torus $T'$ such that $T = \G_m \times T'$ as group, then $G_{\rm der} \times T' \to G$ is an isomorphism of $k$-schemes with right and left $T$-action. In any case, the cellular $\A^1$-chain complex of $G$ is just the tensor product of the cellular $\A^1$-chain complex of $G_{\rm der}$ and the group ring of a split torus.  In fact, one may assume only that $G$ is a connected affine group with reductive quotient split of semi-simple rank $1$ in order to determine its cellular $\A^1$-chain complex.

\item The differential 
\[
\KMW_1\otimes\ZA[T]\xrightarrow{\partial} \ZA[\dot{s} T]
\]
is right and left $T$-equivariant. The right action of $T$ on the factor $\KMW_1$ is trivial, and is the obvious one on the other factor, and the left action of $T$ on $\KMW_1$ is through the character $\alpha: T\to \G_m$, and is the obvious one on $\ZA[\dot{s}T]$.  If $u$ is a unit (in $T$), then the action of $u$ on $\KMW_1$ is the multiplication by $\<u\>:\KMW_1 \to \KMW_1$. 
\end{enumerate}
\end{remark}

\subsection{The general case} \hfill 
\label{subsection complex in rank r}

In this section, we explain how a weak pinning (see Definition \ref{definition weak pinning}) gives an orientation of the cellular $\A^1$-chain complex of a split, semisimple algebraic group.  These results play an important role in orienting the cellular $\A^1$-chain complex of a generalized flag variety, a detailed study of which will be taken up in the sequel (also see  Section \ref{subsection G/B}).  Our proof of the main theorem will use Proposition \ref{prop: normal bundle} and Lemma \ref{lem:comptub}, which rely on the differential calculus on $G$ (Lemma \ref{lem:adjoint}). 

Throughout this subsection, unless otherwise specified, we will assume that $G$ is a split semisimple group of rank $r$ over $k$.  We fix a pinning $(T, B, \{u_{\alpha}\}_{\alpha \in \Delta})$ of $G$ (see \cite[XXIII \S 1 \'Epinglages]{SGAIII}, \cite{Demazure}), which consists of 
\begin{itemize}
\item a maximal $k$-split torus $T$ of $G$;
\item a Borel subgroup $B$ of $G$ containing $T$; and
\item an isomorphism $u_{\alpha}: \G_a\stackrel{\simeq}{\to} U_\alpha$ of the root groups, for every simple root $\alpha \in \Delta$,
\end{itemize}
where $\Delta$ is the basis of of the root system of $G$ determined by the choice of $T$ and $B$.  It will be convenient to fix an ordering $\{\alpha_1,\dots,\alpha_r\}$ of $\Delta$.  It follows from \cite[XXII Th\'eor\`eme 1.1]{SGAIII} that we get for each $\alpha\in\Delta$ an explicit lift $\dot{s}_\alpha$ in $N_G(T)$ of $s_\alpha\in W$ (also see Convention \ref{convention lifting}).  We will also need, for every element $v\in W$, a choice of a reduced expression (as product of simple reflections $s_{\alpha}$, where $\alpha\in \Delta$).  We denote then by $\dot{v}\in N_G(T)$ the product in the same ordering of the lifts $\dot{s}_\alpha$ of the reflections appearing in the chosen reduced expression of $v$.  For an element $v\in W$, we set $$U_v = \underset{\alpha\in\Phi_v}{\prod} U_\alpha,$$ where $\Phi_{v} = \{\alpha\in\Phi^+~|~v(\alpha) \in \Phi^{-}\}$.  Observe that with our notation, $U_{w_0}$ is the big cell $\Omega_0$ of $G$. We also know that the product morphism
\begin{equation}
\label{eq prod U_v}
U_v  U_{w_0 v} \cong U
\end{equation}
is always an isomorphism of $k$-schemes \cite[Lemma 8.3.5]{Springer}.  
We will simply denote the Bruhat cell $B \dot{w}B$ by $BwB$ because of its independence on the choice of the lift $\dot{w} \in N_G(T)$ of $w \in W$.  Recall from \cite[Chapter 8, Section 3]{Springer} that for an element $w\in W$, the Bruhat cell $B{w}B$ is isomorphic as a $k$-scheme through the product in $G$ with
\[
U_{w^{-1}}\times T \times U \xrightarrow{\simeq} B\dot{w}B; \quad (u,t,u')\mapsto u\dot{w}tu'.
\]
Hence, we will write $BwB = U_{w^{-1}} \dot{w} TU$.  The above isomorphism shows that as a $k$-scheme, the Bruhat cell $B\dot{w}B$ is isomorphic to a product of an affine space and $T$; thus, it is cohomologically trivial and the above filtration defines a cellular structure on $G$. In what follows, we need to compute the  differential of the associated cellular $\A^1$-chain complex in low degrees ($\leq 2$ to be precise).  We will hence have to choose orientations of that cellular structure carefully in order to facilitate our cellular homology computations.

We need some recollection on ``differential calculus'' on  $G$ to describe the normal bundles involved in the cellular $\A^1$-chain complex.  

\begin{convention}
Given a vector bundle $\sE$ (that is, a locally free sheaf of $\sO_X$-modules) on a $k$-scheme $X$, we will denote its underlying total space by $\V(\sE)=\Spec(\Sym_k^*(\sE^\vee))$.  We will often abuse the notation and call $\V(\sE)$ a vector bundle on $X$, when the meaning is clear from the context.
\end{convention}

Recall that the tangent bundle $T_G$ (dual to $\Omega^1_{G/k}$) of $G$ is trivial as $G$ is a smooth algebraic group over $k$.  Let $\mathfrak g:= T_{G, e}$ be the tangent space of $G$ at $e$ (that is, the stalk of $T_G$ at $e$; this is the Lie algebra of $G$). Then the morphism 
\begin{equation}
\label{eq theta}
\theta: \V(\mathfrak{g})\times G\to \V(T_G); \quad (v,g)\mapsto (d(?\times g)_e(v),g) 
\end{equation}
and the induced morphism $\mathfrak{g} \otimes_k \sO_G \to T_G$ are isomorphisms of vector bundles over $G$.  Here $?\times g:G\to G$ is the right multiplication by $g$ and $d(?\times g)_e: g\to T_{G, g}$ denotes its differential at $e$.   One easily checks that this isomorphism is $G$-right equivariant, for the trivial $G$ action on $\mathfrak g$ on the left and the natural $G$-action on the right.  However, $\theta$ is not left $G$-equivariant.  For $g\in G$, one may verify (using $(g\times ?)\circ (?\times h) = (?\times h) \circ (g\times ?)$, for $h\in G$) that $\theta^{-1}\circ d(g\times ?) \circ \theta $ is the constant automorphism of vector bundles $\V(\mathfrak{g})\times G \stackrel{\simeq}{\to} \V(\mathfrak{g})\times G$ induced by the adjoint morphism
\[
Ad(g): \mathfrak{g}\xrightarrow{\simeq} \mathfrak{g},
\]
which is the differential at $e$ of the conjugation map $h \mapsto ghg^{-1}$.  In this way one may easily prove the following fact, the proof of which is left to the reader. 

\begin{lemma}
\label{lem:adjoint} 
Let $\mu:G\times G\to G$ be the product morphism. Using our identification $\V(T_G) = \V(\mathfrak{g})\times G$, and thus $\V(T_{G\times G}) = \V(\mathfrak{g} \oplus \mathfrak{g})\times G \times G$, the differential $$d(\mu): \V(\mathfrak{g} \oplus \mathfrak{g})\times G \times G \to  \V(\mathfrak{g})\times G$$ at $(h,g)\in G\times G$ (with obvious meaning) is the $k$-linear morphism 
\[ 
(\Id_{\mathfrak{g}}, Ad(h)) :\mathfrak{g}\oplus \mathfrak{g} \to \mathfrak{g}.
\]
\end{lemma}

Let $w\in W$ and set $v: = w_0w \in W$; we then also have $ w_0v = w$, since $w_0$ is an element of order $2$.  We are going to define a canonical tubular neighborhood of $BwB$ in $G$.  We consider the open subset $\dot{v}(\Omega_0)\subset G$ obtained by left translation of $\Omega_0$ by $\dot{v}$. Using 
\eqref{eq prod U_v}, it follows that
\[
\dot{v}(\Omega_0) = \big(\dot{v} U_v \dot{v}^{-1} \big) \cdot \big( \dot{v} U_{w_0v}\dot{v}^{-1} \big) \cdot \dot{v}\dot{w}_0 TU.
\]
Clearly, $\dot{w}_0\dot{v}$ is a lift of $w$ in $N_G(T)$, so that $\dot{v}\dot{w}_0 TU = \dot{w} TU$.

\begin{lemma} 
\label{lem: tubular neighborhood} 
With the above assumptions and notation, we have:
\begin{enumerate}[label=$(\alph*)$]
\item $\dot{v} U_{w_0v}\dot{v}^{-1} = U_{w^{-1}}$;

\item $\dot{v}U_v\dot{v}^{-1} = \underset{\beta\in\Phi^-|w(\beta)\in\Phi^-}{\prod} U_{\beta}$.
\end{enumerate}
Consequently, 
\begin{equation}
\label{eq v.Omega0}
\dot{v}(\Omega_0) = \left(\underset{\beta\in\Phi^-|w^{-1}(\beta)\in\Phi^-}{\prod} U_{\beta}\right) \cdot \big( U_{w^{-1}} \dot{w} TU \big) = \left(\underset{\beta\in\Phi^-|w^{-1}(\beta)\in\Phi^-}{\prod} U_{\beta}\right) \cdot BwB  
\end{equation}
is an open subset of $G$ only depending on $w \in W$ and containing $BwB =  U_{w^{-1}} \dot{w} TU$.
\end{lemma}

\begin{proof} \hfill 
\begin{enumerate}[label=$(\alph*)$]
\item By \cite[Exercise 8.1.12.(2)]{Springer}, for any root $\alpha$, we have $\dot{v} U_\alpha \dot{v}^{-1} = U_{v(\alpha)}$, where $v(\alpha)$ denotes the standard action of $v \in W$ on the root $\alpha$. Thus, one has 
$$\dot{v} U_{w_0v}\dot{v}^{-1} = \underset{\alpha\in \Phi_{w_0v}}{\prod} U_{v(\alpha)}.$$ 
Now, $\alpha\in \Phi_{w_0v}$ if and only if $\alpha\in \Phi^+$ and $w_0 (v(\alpha))\in \Phi^{-}$.  Since $w_0(\Phi^-) = \Phi^+$, this is equivalent to $\alpha\in\Phi^+$ and $v(\alpha) \in\Phi^+$. But $\alpha = v^{-1} (v(\alpha)) = w_0  w^{-1}(v(\alpha)) \in \Phi^+$. Thus, such $\alpha$'s are precisely those in $\Phi^+$ such that $v(\alpha)\in \Phi^+$ and $w^{-1}(v(\alpha)) \in \Phi^-$.  This means that $v(\alpha) \in \Phi_{w^-}$. But the cardinality of $\Phi_{w_0v}$ is the length of $w_0v=w$, and so also of $w^{-1}$. This proves $(a)$.

\item Since $\dot{v}U_v\dot{v}^{-1}= \underset{\alpha\in\Phi^+|v(\alpha)\in\Phi^-}{\prod} U_{v(\alpha)}$, setting $\beta = v(\alpha)$ one easily obtains 
\[
\{\alpha\in\Phi^+|v(\alpha)\in\Phi^-\} = v^{-1}\{\beta\in\Phi^-|w^{-1}(\beta)\in \Phi^{-}\},
\]
using $v^{-1} =  w^{-1} w_0$ and $(b)$ follows.
\end{enumerate}
The formula \eqref{eq v.Omega0} is a straightforward consequence of $(a)$ and $(b)$ and of the discussion preceding Lemma \ref{lem: tubular neighborhood}.
\end{proof}

\begin{remark} \hfill 
\begin{enumerate}
\item Let $\alpha\in\Phi$. We denote by $\mathfrak{g}_\alpha\subset \mathfrak{g}$ the image of the differential at $1\in G$ of the morphism $U_\alpha\to G$.  This differential induces an isomorphism of $1$-dimensional $k$-vector spaces $\mathfrak{u}_\alpha\cong \mathfrak{g}_\alpha$ and it is well known that $\mathfrak{g}_\alpha\subset \mathfrak{g}$ is the rank one weight $k$-vector subspace of $\mathfrak{g}$ corresponding to the character $\alpha$. This leads to the decomposition $\mathfrak{g} = \underset{\alpha\in\Phi}{\oplus} \mathfrak{g}_\alpha \oplus \mathfrak{h}$, where $\mathfrak{h}$ is the Cartan subalgebra of $\mathfrak{g}$ (which is the Lie algebra or the tangent space of $T$).

\item The normal bundle $\nu_w$ of $BwB$ in $G$ is constant in the sense that it is of the form $\nu_w = \mathfrak{v}_w \otimes_k \sO_G$.  But observe that the canonical epimorphism
\[
T_G|_{BwB} \twoheadrightarrow \nu_w
\]
is not constant! It can be made explicit at any point of $BwB$ by using Lemma \ref{lem:adjoint} and the standard formulas for the adjoint morphism acting on $\mathfrak{g}$.
\end{enumerate}
\end{remark}

\begin{proposition} 
\label{prop: normal bundle} 
We keep the above notation and assumptions. Let $w\in W$ be an element of the Weyl group of $G$.
\begin{enumerate}[label=$(\alph*)$]
\item If $\nu_w$ denotes the normal bundle of $BwB$ in $G$, then the composition 
\[
\underset{\beta\in\Phi^-|w^{-1}(\beta)\in \Phi^-}{\oplus} \mathfrak{u}_\beta \hookrightarrow \mathfrak{g}|_{BwB} \twoheadrightarrow \mathfrak{v}_w
\]
is an isomorphism, where $\mathfrak{g}|_{BwB}$ is the vector space underlying $T_G|_{BwB}$.

\item The right $B$-action on $G$ induces through the above isomorphism the trivial action on the normal bundle of $BwB$ in $G$ (and the obvious right action on the base).

\item The left $B$-action on $G$ induces the following action on the normal bundle of $BwB$ in $G$: the adjoint action of $T$ (through $B \twoheadrightarrow T$) on the $k$-vector space $\underset{\beta\in\Phi^-|w^{-1}(\beta)\in \Phi^-}{\oplus} \mathfrak{u}_\beta$ (and the obvious left action on the base).
\end{enumerate}
\end{proposition}

\begin{proof} 
Consider the open subset 
\[
\underset{\alpha\in\Phi^+|v(\alpha)\in\Phi^-}{\prod} U_{v(\alpha)} \times B\dot{w}B
\]
as a tubular neighborhood of $B\dot{w}B$ and observe that the morphism of smooth $k$-schemes induced by taking the product in $G$: 
\[
\left(\underset{\beta\in\Phi^-|w^{-1}(\beta)\in\Phi^-}{\prod} U_{\beta}\right) \times B\dot{w}B \to \left(\underset{\alpha\in\Phi^+|v(\alpha)\in\Phi^-}{\prod} U_{v(\alpha)}\right) \times B\dot{w}B
\]
is an isomorphism. For every point $x\in B\dot{w}B$, the above tubular neighborhood induces, by taking the differential at the point $(0,x)$, the isomorphism 
\[
\left( \underset{\beta\in\Phi^-|w^{-1}(\beta)\in \Phi^-}{\oplus} \mathfrak{u}_\beta \right)  \oplus (T_G|_{B\dot{w}B})_x \cong \mathfrak{g}|_{B\dot{w}B},
\]
given by the sum of the canonical inclusions (the canonical inclusion of the first summand being given by $\underset{\beta\in\Phi^-|w^{-1}(\beta)\in \Phi^-}{\oplus} \mathfrak{u}_\beta \cong \underset{\beta\in\Phi^-|w^{-1}(\beta)\in \Phi^-}{\oplus} \mathfrak{g}_\beta \subset \mathfrak{g}$). This proves $(a)$ and the parts $(b)$, $(c)$ are easy consequences of $(a)$.  
\end{proof}

\subsubsection{\bf Orientations in codimension $1$}
\label{subsubsection codim1} \hfill 

Let $w\in W$ be an element of length $\ell-1$, where $\ell$ is the length of the longest element $w_0 \in W$. There is a unique $\alpha_i\in \Delta$ with $w= w_0 s_i$.  If $\alpha'_i$ denotes the simple root $w_0(-\alpha_i) \in \Delta$, then we have $w = s'_iw_0$, where $s'_i\in W$ is the reflection corresponding to $\alpha_i'$.  We claim that \begin{equation}
\label{eq s'}                                                                                                                                                                                                                                                    
s'_i= w_0s_iw_0^{-1} = w_0s_iw_0.                                                                                                                                                                            \end{equation}
Set $\sigma_i := w_0s_iw_0^{-1} = w_0s_iw_0$.  Indeed, for $w\in W$, the set $\Phi_w = \{\alpha\in\Phi^+|w(\alpha)\in\Phi^-\}$ has $\ell(w)$ elements, by \cite[Lemma 8.3.2 (ii)]{Springer}. If $\beta \in \Phi^+-\{\alpha'_i\}$, then $w_0(\beta)\in\Phi^--\{-\alpha_i\}$ and $s_iw_0(\beta) \in \Phi^-$ so that $\sigma_i(\beta)\in\Phi^+$. If $\beta = \alpha'_i$, then  $w_0(\beta) = -\alpha_i$ and $\sigma_i(\beta) = -\alpha'_i$, so that $\Phi_{\sigma_i} = \{\alpha'_i\}$.  Hence, $\sigma_i = s_i'$ as claimed.

For $i\in\{1,\dots,r\}$, let us denote by $w_i$ the product $w_0s_i$ and  $Y_i:= B\dot{w}_iB$ the corresponding codimension $1$-cell.  In that case, the open neighborhood of $Y_i$ obtained in Lemma \ref{lem: tubular neighborhood} is $U_{-\alpha'_i}Y_i$.  It follows that we get an isomorphism 
\[
\left(\mathfrak{u}_{-\alpha'_i} \otimes \sO_{Y_i} \right) \oplus T_{Y_i} \cong T_G|_{Y_i}, 
\]
which is determined by the isomorphism $\mathfrak{u}_{-\alpha'_i}\cong \mathfrak{g}_{-\alpha'_i}$ on the first factor (followed by the canonical inclusion) and the canonical inclusion $T_{Y_i,y} \subset \mathfrak{g}$ for each $y\in Y_i$ on the second factor. This isomorphism induces the above isomorphism
\begin{equation}
\label{eq nu_Y_i 1}
\nu_{Y_i} \xrightarrow{\simeq} \mathfrak{u}_{-\alpha'_i} \otimes_k \sO_{Y_i}.
\end{equation}
Observe here that one may also consider the product $Y_iU_{-\alpha_i}$ for an open neighborhood of $Y_i$. We get in that case an isomorphism of vector bundles 
\[
T_{Y_i}\oplus \left(\mathfrak{u}_{-\alpha_i} \otimes \sO_{Y_i} \right) \cong T_G|_{Y_i}, 
\]
which, for every $y \in Y_i$, is the canonical inclusion $T_{Y_i,y} \subset \mathfrak{g}$ on the first factor and induced by the morphism $Ad(y): \mathfrak{u}_{-\alpha_i} \subset \mathfrak{g}$. This isomorphism induces also an isomorphism
\begin{equation}
\label{eq nu_Y_i 2}
\nu_{Y_i} \xrightarrow{\simeq} \mathfrak{u}_{-\alpha_i}\otimes_k \sO_{Y_i},
\end{equation}
which does not agree in general with \eqref{eq nu_Y_i 1}.

\begin{remark}
In fact one may generalize the above observations and prove the analogues of Lemma \ref{lem: tubular neighborhood} and Proposition \ref{prop: normal bundle} for multiplying on the right of the Bruhat cell $B\dot{w}B$. In that case, one would see that the left $B$-action induces the trivial action on the normal bundle of the cell. 
\end{remark}

Comparing the isomorphisms \eqref{eq nu_Y_i 1} and \eqref{eq nu_Y_i 2} above, we get an isomorphism
\begin{equation}
\label{eq nu_Y_i}
\mathfrak{u}_{-\alpha_i}\xrightarrow{\simeq} \mathfrak{u}_{-\alpha'_i} 
\end{equation}
of $k$-vector spaces inducing the isomorphism of line bundles $\mathfrak{u}_{-\alpha'_i} \otimes_k \sO_{Y_i} \cong  \mathfrak{u}_{-\alpha_i}\otimes_k \sO_{Y_i}$ over $Y_i$ which we now describe. Observe first that this isomorphism is canonical and does not depend on any choices, except $\alpha_i$, which also determines $\alpha'_i = w_0(-\alpha_i)$. 

\begin{lemma} \label{lem:comptub} With the above assumptions and notation, let $y_{i}:Y_i\to \G_m$ denote the invertible regular function obtained by composing the canonical projection $Y_i\to T$ and the character $\alpha_i:T\to\G_m$.  Then the isomorphism \eqref{eq nu_Y_i} is the composition of the multiplication $\mathfrak{u}_{-\alpha_i}\stackrel{\simeq}{\to} \mathfrak{u}_{-\alpha_i}$ by $(y_{i})^{-1}$, the isomorphism $Ad(\dot{s}_i):\mathfrak{u}_{-\alpha_i}\stackrel{\simeq}{\to}\mathfrak{u}_{\alpha_i}$ and $Ad(\dot{w_0}): \mathfrak{u}_{\alpha_i}\stackrel{\simeq}{\to}\mathfrak{u}_{-\alpha'_i}$.
\end{lemma}

\begin{proof} 
The automorphism group of a line bundle over $Y_i$ is the group of units $\sO({Y_i})^\times$ of $Y_i$.  We need to compute the unit in $\sO(Y_i)^\times$ corresponding to the automorphism $\text{\eqref{eq nu_Y_i 1}}^{-1} \circ \text{\eqref{eq nu_Y_i 2}}$ of $\nu_{Y_i}$.  Observe that the closed immersion $\dot{w_0} \dot{s_i}T\subset Y_i$ induces an isomorphism on the group of invertible functions.  Now, we explicitly compare the two isomorphisms $T_{Y_i} \oplus \left(\mathfrak{u}_{-\alpha_i} \otimes_k \sO_{Y_i}\right) \cong T_G|_{Y_i}$ and $\left(\mathfrak{u}_{-\alpha'_i} \otimes_k \sO_{Y_i} \right)\oplus T_{Y_i} \cong T_G|_{Y_i}$, corresponding to \eqref{eq nu_Y_i 1} and \eqref{eq nu_Y_i 2} respectively, on restriction to $\dot{w_0}\dot{s_i}T$. On the summand $T_{Y_i}$, both are given by the obvious inclusion, as well as on the summand $\mathfrak{u}_{-\alpha'_i}$. On the other hand, on the summand $\mathfrak{u}_{-\alpha_i}$ it is given by $Ad(y)$, with $y\in Y_i$ the point at which it is computed. Setting $y = \dot{w_0}\dot{s_i}t \in \dot{w_0}\dot{s_i}T$, we get $Ad(y) = Ad(\dot{w}_0)\circ Ad(\dot{s}_i) \circ (\alpha_i(t))^{-1}$ and the lemma follows.
\end{proof}

\begin{remark} \hfill
\label{rem codim 1 comparison}
\begin{enumerate}
\item Observe that the isomorphism $Ad(\dot{s_i}):\mathfrak{u}_{-\alpha_i}\stackrel{\simeq}{\to} \mathfrak{u}_{\alpha_i}$ can be explicitly described. Indeed, the choice of the pinning of $G$ uniquely determines an isomorphism $u_{\alpha_i}:\G_a\cong U_{\alpha_i}$. Now by \cite[ XX, Th\'eor\`eme 2.1]{SGAIII}, we have a canonical duality between $U_{\alpha_i}$ and $U_{-\alpha_i}$, so that in that case we get an isomorphism $U_{\alpha}\stackrel{\simeq}{\to} U_{-\alpha}$ and in \cite[XXIII, \S 1.2]{SGAIII}, it is proved that $Ad(\dot{s_i})$ is exactly the inverse to this isomorphism.

\item Observe also that, although the isomorphism \eqref{eq nu_Y_i} between $\mathfrak{u}_{-\alpha_i}$ and $\mathfrak{u}_{-\alpha'_i}$ is canonical, its description in Lemma \ref{lem:comptub} does depend on the pinning and the choices of reduced expressions made at the beginning! Indeed, the invertible function $x_{\alpha_i}:Y_i\to \G_m$ on $Y_i = U\dot{w}_0\dot{s}_iTU$ obtained by pulling back $\alpha_i: T \to \G_m$ by the natural morphism $Y_i \to \dot{w}_iT \stackrel{\simeq}{\to} T$ depends on the choice of the lifting of $w_i= w_0s_i$, which is precisely cancelled by $Ad(\dot{w_0}) \circ Ad(\dot{s}_i)$.
\end{enumerate}
\end{remark}

\subsubsection{\bf Orientations for a general Bruhat cell}
\label{subsubsection orientations} \hfill

It follows from the previous discussion that defining an orientation of the cellular structure on $G$ in the sense of Definition \ref{definition cellular} is equivalent to giving, for each $w\in W$, an orientation of the $k$-vector space
\begin{equation}
\label{eq orient nu_w}
\mathfrak{v}_w:= \underset{\beta\in\Phi^-|w^{-1}(\beta)\in \Phi^-}{\oplus} \mathfrak{u}_\beta.  
\end{equation}
In codimension one, as we have seen in the previous example, it exactly amounts to give an orientation of each of the $\mathfrak{u}_{-\alpha}$, $\alpha\in \Delta$. Now by \cite[XXII Th\'eor\`eme 1.1]{SGAIII} $\mathfrak{u}_\alpha$ and $\mathfrak{u}_{-\alpha}$ are canonically dual to each other, for $\alpha\in \Phi$.  Thus, an orientation for the Bruhat cellular structure on $G$ in codimension $1$ is the same thing as a pinning of $G$ up to the multiplication of each $u_\alpha$ by squares of units (for each $\alpha$). 

\begin{definition}
\label{definition weak pinning}
Consider the equivalence relation on the set of pinnings of $G$ defined as follows: two pinnings $(T, B, \{u_{\alpha}\}_{\alpha \in \Delta})$ and $(T', B', \{u'_{\alpha}\}_{\alpha \in \Delta'})$ are equivalent if and only if $T=T'$, $B=B'$ and $u_{\alpha}$ is a multiple of $u'_{\alpha}$ by an element of $k^{\times 2}$ for every $\alpha \in \Delta=\Delta'$.  We will call an equivalence class of pinnings a {\em weak pinning} of $G$.
\end{definition}

Given a pinning $(T,B, \{u_\alpha\}_\alpha)$ of $G$, one has the canonical lift $\dot{s}_\alpha$ of $s_\alpha$ in $N_G(T)$ for each simple root $\alpha$ (see Convention \ref{convention lifting}).  Using the formula \eqref{eq:walpha}, one sees that multiplying the $u_\alpha$ by squares in $k^\times$ multiplies the element $\dot{s}_\alpha$ also by a square in $T$ (by which we mean an element of $\underset{i=1}{\overset{r}{\prod}} k^{\times 2}$).  Now, fix for each $v\in W$ a reduced expression $$v = \sigma_1\circ \dots \circ \sigma_{\ell(v)}$$ of $v$ as a product of reflections $\sigma_i$ associated to simple roots in $\Delta$. Then the product (in the given ordering) $\prod_i \dot{\sigma}_i$ is a lift of $v$ in $N_G(T)$, which we denote by $\dot{v}$.

Let $w\in W$, and consider $v: = w_0w$ so that we have $w =  w_0v$. Clearly,
\[
\{\beta\in\Phi^-|w^{-1}(\beta)\in \Phi^-\} = \{\beta\in\Phi^-|v^{-1}(\beta)\in \Phi^+\} = -\Phi_{v^{-1}}.
\]
Thus, for any $w$ we can rewrite the isomorphism \eqref{eq orient nu_w} concerning the normal bundle of $B\dot{w}B$ as
\[
\mathfrak v_{w} \cong \underset{\alpha\in \Phi_{v^{-1}}}{\oplus} \mathfrak{u}_{-\alpha}. 
\]
By \cite[Chapter 8 Section 3]{Springer}, if $v = \sigma_{1} \circ\dots \circ\sigma_{\ell(v)}$ is a reduced expression of $v$ and if $\beta_j$ is the simple root corresponding to $\sigma_j$, then
\[
\Phi_{v^{-1}} = \{\beta_{1}, \sigma_1(\beta_{2}),\dots,\sigma_{1}\circ\dots\circ \sigma_{\ell(v)-1}(\beta_{\ell(v)})\}.
\]
Thus with our assumptions and choices, we get an constant isomorphism of constant vector bundles
\[
\underset{i = 1}{\overset{\ell(v)}{\oplus}}~ \G_a \cong \nu_{w} \cong \underset{\alpha\in \Phi_{v^{-1}}}{\oplus} \mathfrak{u}_{-\alpha} \otimes_k \sO_{B\dot{w}B}
\]
induced by the direct sum of the isomorphisms of vector spaces
\[
k \xrightarrow {u_{-\beta_i}} \mathfrak{u}_{-\beta_i} \xrightarrow{\simeq} \mathfrak{u}_{-\sigma_{1}\circ\dots\circ \sigma_{i-1}(\beta_i)}
\]
given by the pinning followed by the conjugation by $\dot{\sigma}_{1}\circ\dots\circ \dot{\sigma}_{i-1}$.  The associated orientation is easily seen to be only dependent on the weak pinning of $G$ and the reduced expression of $v$. This shows that a weak pinning of $G$ and a choice of a reduced expression of $v$ for each $v\in W$ define an orientation of the Bruhat cellular structure of $G$.
We will call this orientation {\em standard orientation} associated to the weak pinning and the choice of reduced expressions.  Observe that an element of $W$ of length $\leq 2$ admits a unique reduced expression and consequently, there are no choices involved in that case.  We assume henceforth that a choice of a weak pinning of $G$ and for each $v\in W$ a choice of a reduced expression of $v$ is made. 

Let $w\in W$. The Bruhat cell $B\dot{w}B$ does not depend on any of the above choices, but in order to make computations with the oriented cellular $\A^1$-chain complex of $G$ we need to describe the corresponding summand in $\Ctcell_*(G)$ explicitly. The normal bundle is already oriented by our above choices; the corresponding summand in $\Ctcell_{i}(G)$ where $i = \ell(w_0) - \ell(w)$ is the codimension of $B\dot{w}B$ in $G$ is
\[
\KMW_i\otimes \HA_0(B\dot{w}B).
\]
Assume now that we are given a reduced expression of $w$, which is not necessarily the one previously used to orient the normal bundles of the Bruhat cells. We then get a lift $\tilde{w} \in N_G(T)$ of $w$ and have $$B\dot{w}B = B\tilde{w}B = U\tilde{w}TU$$ so that $\HA_0(B\dot{w}B) = \ZA[\tilde{w}T]$.  We will use this freedom later in the proof of the main theorem, to get, at least in degree $\leq 2$ a canonical description of $\Ctcell_*(G)$ in degree $\leq 2$ depending only on a weak pinning of $G$ and a choice of a reduced expression of $w_0$.

\begin{remark} 
The obvious isomorphism of sheaves induced by the multiplication on the left by $\tilde{w}$ 
\[
\ZA[T] \cong \ZA[\tilde{w}T]
\]
only depends on both the chosen reduced expression of $w$ and the weak pinning. Indeed, $\ZA[T]$ is isomorphic to $\ZA[\G_m]\otimes_{\A^1}\dots \otimes_{\A^1} \ZA[\G_m]$ ($r$ factors), and (as in the rank one case above), multiplication by squares (in $\G_m$) induces the identity on each factor of the tensor product. Observe that our choice of the lift $\tilde{w}$ does depend (up to squares) on the weak pinning as the lifts $\dot{s}_i$ for the simple roots do, and of course also on the chosen reduced expression of $w$.
\end{remark}

For each $i\in\{1,\dots,\ell\}$ the above choices and conventions thus define an isomorphism:
\[
\Ctcell_{i}(G) \cong \underset{w\in W; ~\ell(w) = \ell - i}{\oplus} \KMW_{i}\otimes \ZA[\tilde{w}T], 
\]
where the factor $\KMW_{i}\otimes \ZA[\tilde{w}T]$ corresponds to the normal bundle of $BwB:=B\dot{w}B = B\tilde{w}B$. Again, one should notice that the identification of the factor $\KMW_{i}\otimes \ZA[\tilde{w}T]$ not only depends on the orientation of the normal bundle of $BwB$, which uses a reduced expression of $v: = w_0w$, but also on the choice of a reduced expression of $w$ itself, giving the lifting $\tilde{w}$. In dimension $0$ for instance, observe that the writing $\Ctcell_{0}(G) \cong \ZA[\~{w}_0T]$ depends on a choice of a reduced expression for $w_0$ and the weak pinning.  On the other hand, note that the identification $\Ctcell_{\ell}(G) = \KMW_{\ell}\otimes \ZA[B]$ corresponding to the unique cell of maximal codimension $\ell$ in $G$ (that is, the Borel subgroup $B$), one uses another reduced expression for $w_0$. 

\begin{remark}
\label{remark normal horizontal}
In order to explicitly describe the oriented cellular $\A^1$-chain complex $\Ctcell_*(G)$, we need a weak pinning of $G$ and for any element of $W$ two choices of reduced expressions: the ``normal expression'' and the ``horizontal expression''.  For $w\in W$, one gets an orientation of the normal bundle of the cell $BwB$ by using a the normal reduced expression for $v = w_0 w$ and the identification  $\HA_0(BwB) = \ZA[\tilde{w}T]$ using the horizontal reduced expression of $w$. The lift of $w\in W$ in $N_G(T)$ corresponding to the normal reduced expression of $w$ will be denoted by $\dot{w}$ and the one corresponding to the horizontal reduced expression will be denoted by $\tilde{w}$. 
\end{remark}

\begin{remark} \hfill
\label{rem important1}
\begin{enumerate}
\item Observe that the structure of a right $\ZA[T]$-module on $\Ctcell_{i}(G)$ is the obvious one. The right action of $T$ on the factors $\KMW_i$ is trivial, and is the obvious one on the other factors. The structure of left $\ZA[T]$-module is the obvious one on $\ZA[\dot{w}T]$ and is given on $\KMW_i$ as the tensor product of the characters $\alpha: T\to \G_m$ of $T$ corresponding to the roots involved in the reduced expression of the $v$'s used; this follows from Proposition \ref{prop: normal bundle}.

\item The cellular $\A^1$-chain complex $\Ccell_*(G)$ is also a complex of left and right $\ZA[B]$-modules, but observe that in fact $\ZA[B] = \ZA[T]$. Given a weak pinning of $G$ and choice of reduced expressions for elements of $W$, the oriented cellular structure on $G$ is $B$-right invariant by (1) and induces an oriented cellular structure on $G/B$ (which is in fact strictly cellular) as well. The morphism $G\to G/B$ is thus automatically oriented cellular in the obvious way and induces an isomorphism of chain complexes $\Ctcell_*(G)\otimes_{\ZA[T]}\Z \cong \Ctcell_*(G/T)$.

\item It is clear by construction that the differential in the complex $\Ctcell_*(G)$ restricted on the summand $\KMW_{i}\otimes \ZA[\dot{w}T]\subset \Ctcell_{i}(G)$ corresponding to $BwB$ (where $\ell(w)=\ell - i$) has its image contained in the direct sum
\[
\underset{w<v; ~\ell(v)= \ell -i +1}{\oplus} \KMW_{i-1}\otimes \ZA[\dot{w}T], 
\]
where $\leq$ denotes the Bruhat ordering, with $w\leq v$ meaning $BwB \subset \overline{BvB}$ (see, for example, \cite[\S 8.5.4]{Springer}).

\item It is possible to choose a pinning of $G$ so that the lift of an element $w\in W$ defined as above using a reduced expression of $w$ does not depend on the reduced expression. However it seems impossible to find a (weak) pinning such that the orientation of each normal bundle of the Bruhat cell $BwB$, as defined above, only depends on $w$ and not on the chosen reduced expression! This may be verified explicitly for $SL_3$ for the normal bundle of $B$.
\end{enumerate}
\end{remark}


Given a weak pinning of $G$ and a choice of normal and horizontal reduced expressions for elements of $W$, there is a canonical description of the oriented cellular $\A^1$-chain complex of $G$. We will have to use it explicitly in low degrees. In degree $0$ one has:
\[ 
\Ctcell_{0}(G) = \ZA[\tilde{w}_0T].
\]
It is known that codimension $1$ cells $BwB$ correspond exactly to words $w\in W$ of length $\ell-1$, which are exactly of the form $w_i:=w_0s_i = s'_iw_0$ for some $i\in \{1,\dots,r\}$, where $r$ denotes the rank of $G$. Thus, one has
\[ 
\Ctcell_{1}(G) = \underset{i = 0}{\overset{r}{\oplus}} \KMW_{1}\otimes \ZA[\tilde{w}_iT].
\]
Let us analyze now the oriented chain complex of $G$ in degree $2$, which corresponds to the codimension $2$ Bruhat cells. Each word $w\in W$ of length $\ell-2$ can be written as $w_0s_is_j$ for some $i \neq j$ as $\ell(w_0w) = 2$ by \cite[Proposition 1.77]{Abramenko-Brown}.  We will assume here for simplicity that the Dynkin diagram of $G$ is connected.  

\begin{notation}
\label{notation codim 2 indices}
For a pair of indices $(i,j)$ let us write $|i-j|$ for the distance between $\alpha_i$ and $\alpha_j$ in the Dynkin diagram of $G$. If $|i-j|\geq 2$ (that is, there is no edge connecting $\alpha_i$ and $\alpha_j$ in the Dynkin diagram), then $w_0s_is_j = w_0s_js_i$ (as $s_i$ and $s_j$ commute with each other).  If $|i-j| = 1$, then necessarily $w_0s_is_j \neq w_0s_js_i$ (recall that we assume that the Dynkin diagram is connected).  We will write $[i,j]$ for the class of $(i,j)$ modulo the following equivalence relation on $\{1, \ldots, r\}^2$: $(i,j)\equiv (j,i)$ if and only if $|i-j|\geq 2$ or $i = j$. We will also denote by $w_{[i,j]}$ the word $w_0s_is_j$.
\end{notation}

With the above notation, we get
\[
\Ctcell_2(G) = \underset{[i,j]}{\oplus} \KMW_2 \otimes \Z[\tilde{w}_{[i,j]}T].
\]
The right $\ZA[T]$-module structure of $\Ctcell_2(G)$ is the obvious one on $\Z[\tilde{w}_{[i,j]}T]$ and trivial on each of the $\KMW_2$ factors. The left $\ZA[T]$-module structure is, as described in Remark \ref{rem important1}, the obvious one on the factor $\ZA[\dot{w}_{[i,j]}T]$, taking into account the formula for $t\in T$: $t \cdot \dot{w} = \dot{w}  w^{-1}(t)$, and the left action of $T$ on the factor $\KMW_*$ is through the tensor product of the characters, in the given ordering, associated to the simple roots appearing in writing the reduced expression of $v=ww_0$, which here has length $\leq 2$.

\subsection{The free strictly \texorpdfstring{$\A^1$}{A1}-invariant sheaf on a split torus} \hfill  
\label{subsection Z[T]}

We end this section by describing the structure of the sheaf $\ZA[T]$ for a split torus $T$ over $k$, understanding of which is clearly central to all our computations. 

\begin{convention}
\label{convention MW}
For $u$ a unit (in some field extension $F$ of $k$), the associated symbol in $\KMW_1(F)$ is denoted by $(u) \in \KMW_1(F)$. By abuse of notation, we will often drop $F$ from the notation and simply write $(u) \in \KMW_1$ etc.  Recall that $T$ is $k$-split torus of rank $r$ and is pointed by $1$. For $t\in T(F)$, we will denote by $[t] \in \ZA[T]$ the image of the basis element of $\Z[T]$ under the morphism $\Z[T]\to \ZA[T]$ (note again the abuse of notation, where we suppress $F$).  We will denote by $(t)\in \ZA[T]$ the element $[t]-[1]$. Observe that if $T = \G_m$, then $\ZA[T] = \Z \oplus \KMW_1$ and for $t\in T$ the element $(t)$ just defined is in $\KMW_1$ and equals the corresponding symbol, making the two notations compatible.
\end{convention}

Choose an isomorphism
\[
\prod_{i=1}^r \alpha^\vee_i : \prod_{i=1}^r \G_m \xrightarrow{\simeq} T.
\]
In the situation of our main theorem (Theorem \ref{theorem intro precise}, Theorem \ref{thm: main}), the torus $T$ is a maximal $k$-split torus of a simply connected group $G$, so the above isomorphism is chosen to be the one given by the simple coroots.  For $t\in T$, we let $t_i\in \G_m$ be its $i$-th component through the above isomorphism; so in other words,
\[
 t_i = \varpi_i(t),
\]
where the $\{\varpi_i\}_i$ denote the fundamental weights, or the basis of $\Hom(T, \G_m)$ dual to $\{\alpha_i\}_i$.

Notice that given the group homomorphism $\alpha_i^\vee: \G_m \to T$, we have $\alpha_i^\vee(u) \in T$ for a unit $u$, which corresponds to the element $[\alpha_i^\vee(u)]$ of $\ZA[T]$.  The symbol $(\alpha_i^\vee(u))$  denotes the element $[\alpha_i^\vee(u)] - [1]$ of $\ZA[T]$.  The morphism
\[
(\alpha^\vee_i(-)): \G_m \to \ZA[T]; \quad u \mapsto (\alpha_i^\vee(u)) \in \ZA[T]
\]
extends automatically to a morphism of sheaves  $\alpha^\vee_i(-): \ZA(\G_m) = \KMW_1 \to \ZA[T]$ because $\ZA[T]$ is strictly $\A^1$-invariant.

\begin{notation}
\label{notation []_i}
For the sake of brevity, we simply denote the morphism $\KMW_1 \to \ZA[T]$ induced by $\alpha_i^{\vee}$ by 
\[
u\mapsto (u)_i := (\alpha_i^\vee(u)) = [\alpha_i^\vee(u)] - [1], 
\]
where $u$ is a unit (in a field extension $F$ of $k$).  For $t \in T$, we let $[t]_i\in \ZA[T]$ denote the image of $[\varpi_i(t)] = 1 + (\varpi_i(t)) \in \ZA[\G_m] = \Z \oplus \ZA(\G_m)$ through the morphism $\ZA[\G_m] \to \ZA[T]$ induced by the coroot $\alpha^\vee_i$. 

\end{notation}

Since the ring $\ZA[T]$ is commutative, so for $v$ a unit (in the field extension $F$) and $i \neq j$, one has in $\ZA[T](F)$:
\[
(u)_i(v)_j = (v)_j(u)_i.
\]
In case $i = j$ we have further:
\[
(u)_i (v)_i = (\eta(u) (v))_i,
\]
where $\eta: \KMW_2 \to \KMW_1$ is the multiplication by $\eta$ \cite{Morel-book}.  This follows easily from the fact that in the ring $\ZA[\G_m] = \Z \oplus \ZA(\G_m) = \Z \oplus \KMW_1$ the product of $(u)$ and $(v)$ is given by $\eta(u)(v)$, which follows from the very definition of $\eta$.  The isomorphism
\begin{equation}
\label{eqn ZT decomposition}
\underset{i=1}{\overset{r}{\otimes}} \ZA[\G_m] \cong \ZA[T] 
\end{equation}
as a sheaf of rings induced by taking the product in $\ZA[T]$
\begin{equation}
\label{eq:ZA[T]}
\begin{split}
\Z \oplus \left( \underset{1\leq i_1<\dots<i_s\leq r}{\oplus} \KMW_s \right) &\xrightarrow{\simeq} \ZA[T] \\
(u_1)\cdots(u_s) &\mapsto (u_1)_{i_1}\cdots(u_s)_{i_s}
\end{split}
\end{equation}
is an isomorphism of sheaves. If $t = (u_1,\ldots,u_r) \in T = \underset{i=1}{\overset{r}{\prod}}~ \G_m $, then the following equality holds in $\ZA[T]$ through the above isomorphism:
\[
(t) = [u_1,\dots, u_r ] - [ 1 ] = \oplus_{1\leq i_1<\dots<i_s\leq r} (u_1)_{i_1}\cdots (u_r)_{i_r}.
\]
For example, in the case $r=2$, $\G_m\times \G_m = T$ and if $t = (u,v)\in \G_m\times \G_m$, then we have
\[
(t) =  (u)_1 + (v)_2 + (u)_1(v)_2.
\]

Set $\bH_T := \ZA[T]$ and for each $i$, let $\bH_i\cong \ZA[\G_m]$ denote the sub-Hopf algebra of $\bH_T$ given by the image of the morphism $\ZA[\G_m]\to \bH_T$ induced by $\alpha^\vee_i:\G_m\to T$. Thus $\bH_T = \underset{i=1}{\overset{r}{\otimes}} \bH_i$ as a Hopf algebra. Now, define  
\begin{equation}
\label{eq:phi_i}
\phi_i: \KMW_2 \to \KMW_1 \otimes \bH_i; \quad (u,v) \mapsto (u)\otimes (v)_i - \eta (u)(v)\otimes [1],
\end{equation}
where $(u)_i\in \bH_i$ denotes the element $[u]-[1]$ of $\ZA[\G_m] \cong \bH_i$.
Thus for each $i$, we have the obvious exact sequence
\[
(E_i): \quad \quad 0\to \KMW_2 \xrightarrow{\phi_i} \KMW_1 \otimes \bH_i \xrightarrow{\cdot} \bH_i \to \Z \to 0,
\]
where $\cdot$ is (induced by) the product.  This exact sequence, considered as a chain complex $(E_i)_*$ with $\Z$ in degree $-1$, is contractible as each of its terms is a projective object in $Ab_{\A^1}(k)$.  It follows that the tensor product $C_*: = \otimes_{i=1}^r (E_i)_*$ is also contractible and hence, acyclic. So we get an exact sequence
\begin{equation}
\label{eq:C*}
C_3 \to C_2 \to C_1 \to C_0 \to \Z \to 0, 
\end{equation} 
where $C_{-1} = \Z$, $C_0 = \underset{i=1}{\overset{r}{\otimes}} \bH_i = \bH$, 
$$C_1 \cong \left( \underset{i=1}{\overset{r}{\oplus}} \KMW_1\right) \otimes \left( \underset{i=1}{\overset{r}{\otimes}} \bH_i \right) \cong \underset{i=1}{\overset{r}{\oplus}} \KMW_1 \otimes \bH$$
and the differential $\~\partial_1: C_1 \to C_0$ on restriction to the $i$th summand is given by the $\bH$-module homomorphism
\[
\KMW_1 \otimes \bH \to \bH; \quad (u) \otimes [t] \mapsto [\alpha_i^{\vee}(u)t],
\]
for every $i$.  Thus, the previous contractible chain complex produces a presentation of $\~Z_1:= \Ker \~\partial_1$ (by right $\bH$-modules):
\[
 C_3 \to C_2 \to \~Z_1 \to 0.
\]
Thus, it follows that the induced exact sequence
\[
C_3\otimes_{\ZA[T]} \Z \to C_2\otimes_{\ZA[T]} \Z \to \~Z_1\otimes_{\ZA[T]} \Z \to 0
\]
is a presentation of $\~Z_1\otimes_{\ZA[T]} \Z$. Since the complexes $(E_i)$ are exact sequences of $\bH_i$-modules, we have 
\[
C_*\otimes_{\ZA[T]} \Z \cong \underset{i=1}{\overset{r}{\otimes}} (E_i\otimes_{\bH_i}\Z)_*. 
\]
Now, for every $i$, the complex
\[
(E_i)_*\otimes_{\bH_i}\Z: \quad 0\to \KMW_2 \stackrel{-\eta}{\longrightarrow} \KMW_1  \stackrel{0}{\longrightarrow} \Z \to \Z \to 0
\]
is (isomorphic to) the augmented cellular $\A^1$-chain complex of $\P^2$.  Consequently, $C_*\otimes_{\ZA[T]} \Z$ is (isomorphic to) the cellular $\A^1$-chain complex of $(\P^2)^r$.  It follows easily that the morphism 
\[
C_3\otimes_{\ZA[T]} \Z \to C_2\otimes_{\ZA[T]} \Z
\]
is the morphism
\[
\left( \underset{(i,j,m); ~i<j<m}{\oplus} \KMW_3 \right) \oplus \left(\underset{(i,j); ~i\not=j}{\oplus} \KMW_3\right) 
\xrightarrow{\begin{pmatrix}
              0 & 0 \\ 0 & \underset{(i,j)}{\oplus} (\pm \eta)
            \end{pmatrix}
}
\left(\underset{i}{\oplus}~ \KMW_2 \right) \oplus \left( \underset{(i,j); ~i<j}{\oplus} \KMW_2 \right),
\]
where $i,j, m$ run through $\{1,\ldots, r\}$. Going through the previous construction, it is easy to identify the isomorphism just constructed:

\begin{lemma}\label{lem: Z1} The isomorphism
\[
\left(\underset{i}{\oplus}~ \KMW_2 \right) \oplus \left( \underset{(i,j); ~i<j}{\oplus} \KM_2 \right) \xrightarrow{\simeq} \~Z_1\otimes_{\ZA[T]} \Z 
\]
is induced by the following morphisms: for each $i$, its restriction to the $i$-th $\KMW_2$ is $\phi_i$ from \eqref{eq:phi_i} composed with the inclusion $\KMW_1\otimes \bH_i \subset \KMW_1\otimes \bH \subset C_1$. For each $(i,j)$ with $i\neq j$, its restriction to the $\KM_2$ corresponding to $(i,j)$ is 
\[
(u,v) \mapsto (u)\otimes (v)_i - (v)\otimes (u)_j \in (\KMW_1\otimes \bH_i) \oplus (\KMW_1\otimes \bH_j) \subset C_1.
\]
\end{lemma}

\section{The oriented cellular \texorpdfstring{$\A^1$}{A1}-chain complex of a split, semisimple, simply connected algebraic group in low degrees}
\label{section differential}

In this section we prove our main theorem, which determines the sheaf $\Hcell_1(G)$, where $G$ is a split, semisimple, almost simple, simply connected algebraic group over $k$ of rank $r$.  Thus, the Dynkin diagram of $G$ is reduced irreducible with $r$ vertices.  We choose an ordering $\{\alpha_1,\dots,\alpha_r\}$ of the set $\Delta$ of simple roots.  For each $i$, we write $w_i:= w_0s_i$, and for $i\not= j$ we write $w_{[i,j]}:=w_0s_is_j$. Recall also that for each $i$ we let $s'_i$ be the reflection (corresponding to the simple root $w_0(-\alpha_i)$) such that $s'_iw_0 = w_i = w_0s_i$.

We briefly recall our conventions that will be used throughout this section.  We fix the choice of a weak pinning of $G$ and for each $w\in W$ a reduced expression as a product of the $s_{\alpha}$'s, where $\alpha\in \Delta$ runs through the set of simple roots. For $w\in W$ we denote by $\dot{w}$ the lift of $w$ in $N_G(T)$ given by the above formula using the elements $\dot{s}_{\alpha}$ given by the (weak)-pinning and using the chosen reduced expression of $w$. Usually, we will simply write $BwB$ instead of $B\dot{w}B$ and $\ZA[wT]$ instead of $\ZA[\dot{w}T]$.  Recall that $\Ccell_*(G)$ is a complex of right and left $\ZA[T]$-modules (see Remark \ref{rem important1}).

Since $G$ is simply connected, it is also $\A^1$-connected, so we have $\Hcell_0(G) = \Z$.  It follows that the right (and left) action of $\ZA[T]$ on $\Hcell_*(G)$ is trivial, as it extends to an action of $\Hcell_0(G) = \Z$.  The main idea of our main computation is that the exact sequence
\[
\Ccell_2(G) \to Z_1(G) \to \Hcell_1(G)\to 0
\]
induces an exact sequence
\begin{equation}\label{eq: H1cokernel}
\Ccell_2(G)\otimes_{\ZA[T]} \Z \to Z_1(G)\otimes_{\ZA[T]} \Z \to \Hcell_1(G)\to 0,
\end{equation}
by the above observation.

In the Section \ref{subsection differential in degree 1} we describe the sheaf $Z_1(G)\otimes_{\ZA[T]} \Z$ and in the Section \ref{subsection differential in degree 2} we describe the morphism $\Ccell_2(G)\otimes_{\ZA[T]} \Z \to Z_1(G)\otimes_{\ZA[T]}\Z$ induced by the differential. The main theorem will then follow directly from these computations.

\subsection{The differential in degree \texorpdfstring{$1$}{1}} \hfill
\label{subsection differential in degree 1}

We fix a reduced expression of the longest element
\begin{equation}\label{eq:redw0}
w_0 = \sigma_{\ell} \cdots \sigma_1
\end{equation}
in the Weyl group of $G$ as a product of simple reflections.  Recall that (see Remark \ref{remark normal horizontal}) in order to make computations with the oriented cellular $\A^1$-chain complex of $G$, we need to choose a ``normal'' reduced expression $\dot{w}$ and a ``horizontal'' reduced expression $\~w$ for every element $w \in W$.  For every $w \in W$, we have already fixed a choice of the normal reduced expression $\dot{w}$ at the beginning of this section.  We need to fix the horizontal reduced expressions.  We begin by recalling a few known facts about the Weyl group of $G$, which will be used.

\begin{lemma}
\label{lm:red1} 
Let $w\in W$ and let $\sigma_{q}\cdots \sigma_1$ be a reduced expression of $w$. Let $v\in W$ with $v<w$ for the Bruhat order and such that $\ell(v) = \ell(w)-1 = q-1$.  Then there is a unique $\lambda\in\{1,\dots,\ell(w)\}$ such that the product
\[
\sigma_q \cdots \widehat{\sigma}_{\lambda} \cdots \sigma_1 := \sigma_q \cdots \sigma_{\lambda + 1}  \sigma_{\lambda - 1} \cdots \sigma_1
\]
is a reduced expression of $v$.
\end{lemma}

\begin{proof} 
The existence of $\lambda$ follows from the exchange condition \cite[{(E)} p.79 and 3.59, p.137]{Abramenko-Brown}.  If $\mu\in\{1,\dots,q\}$ were such that $\sigma_q\cdots \widehat{\sigma}_{\lambda} \cdots \sigma_1 = \sigma_q \cdots\widehat{\sigma}_{\mu} \cdots \sigma_1$, with say $\lambda>\mu$, then it follows after simplification and multiplication on the right with $\sigma_\mu$ that 
\[
{\sigma}_{\lambda -1 } \cdots {\sigma}_{\mu+1} = \sigma_{\lambda} \cdots \sigma_\mu,
\]
contradicting the minimality of the reduced expression we started with.
\end{proof}

\begin{remark} There is no generalization of Lemma \ref{lm:red1} for the case where $\ell(v) = \ell(w)-2$. For example, in the Weyl group of type $A_2$ (so $G \cong \SL_3$), the longest word is $w_0 = s_1s_2s_1$ (also equal to $s_2s_1s_2$) and there are two different ways to get $s_1$ by removing two simple reflections. For the Weyl group of type $C_2$ (so $G \cong \Sp_4$), the longest word has the form $w_0 = s_2s_1s_2s_1$ (also equal to $s_1s_2s_1s_2$) and there are two different ways to get $s_2s_1$ by removing two reflections.
\end{remark}


\begin{lemma}\label{lm:red2} Let $w\in W$ be of length $q$ and let $s, s'\in W$ be simple reflections such that $s'w = ws$ and $\ell(ws) = q+1$.  Let $\alpha$ and $\alpha'$ the simple roots corresponding to $s$ and $s'$, respectively. Then
\[
 w(\alpha) = \alpha'.
\]
\end{lemma}

\begin{proof} Let $\sigma_{q} \cdots \sigma_1$ be a reduced expression of $w$.  The equation $s'w = ws$, that is, $s' = wsw^{-1}$ shows that $w(\alpha) = \pm\alpha'$. If $w(\alpha) = -\alpha'$, then $\alpha\in \Phi_{w}$.  From \cite[Chapter 8 Section 3]{Springer} we have
\[
\Phi_{w} = \{\beta_{1}, \sigma_1(\beta_{2}),\dots,\sigma_{1}\circ\dots\circ \sigma_{q-1}(\beta_{q})\},
\]
where $\beta_{1},\dots,\beta_q$ are the simple roots corresponding to the $\sigma_1, \ldots, \sigma_q$, respectively. Thus, there exists $j\in \{0,\dots,q-1\}$ with $\alpha = \sigma_{1}\circ\dots\circ \sigma_{j}(\beta_{j+1})$. This implies that in $W$
\[
s = \sigma_{1}\cdots \sigma_{j} \sigma_{j+1}  (\sigma_{1}\cdots \sigma_{j})^{-1};
\]
but then
\[
w s = \sigma_{q} \cdots\sigma_{j+2} \sigma_{j} \cdots \sigma_{1},
\]
which contradicts that the length of $ws$ is $q+1$.  Therefore, we must have $w(\alpha) = \alpha'$.
\end{proof}

\begin{remark} \hfill
\label{rem:red2}
\begin{enumerate}
\item Observe that Lemma \ref{lm:red2} (and its proof) is still valid if one replaces $s$ by a reflection in $W$, that is, an element of the form $vsv^{-1}$ for some $v\in W$ and a simple reflection $s$ such that the root $v(\alpha)$ is positive (with $\alpha$ the simple root corresponding to $s$).

\item Observe that Lemma \ref{lm:red2} does not contradict the equation $w_0(-\alpha_i) = \alpha'_i$ that was already used in Section \ref{section Bruhat decomposition}. Indeed, $w_0s_i = s'_iw_0$; however, $\ell(s'_iw_0) \neq \ell(w_0)+1$.
\end{enumerate}
\end{remark}

We now fix the horizontal reduced expressions for Bruhat cells in codimension $\leq 1$.

\begin{notation}
\label{notation horizontal codim 1}
For the chosen reduced expression $w_0 = \sigma_{\ell}\cdots\sigma_1$ of $w_0$ as in \eqref{eq:redw0}, we set $$\dot{w}_0 = \tilde{w}_0 = \dot{\sigma}_{\ell}\cdots \dot{\sigma}_1 \in N_G(T)$$ using the chosen weak pinning of $G$. We denote by $\beta_{\mu}$ the simple root corresponding to $\sigma_{\mu}$.  For $i\in\{1,\dots,r\}$, let $\lambda_i\in\{1,\dots,\ell\}$ be given by Lemma \ref{lm:red1} applied to $w_i := w_0s_i$.  Define $w_i'':= \sigma_\ell \cdots{\sigma}_{\lambda_i +1}$ and $w'_i:= \sigma_{\lambda_i-1}\cdots\sigma_{1}$ so that the chosen reduced expression \eqref{eq:redw0} of $w_0$ has the form
\[
w_0 = w_i''\sigma_{\lambda_i} w_i'.
\]
We then take $w_i''w_i'$ for the horizontal reduced expression of $w_i = w_0s_i$ and set $$\tilde{w}_i := \dot{w}_i''\dot{w}_i'$$ to be the corresponding lift to $N_G(T)$.
\end{notation}

In particular, for the cellular $\A^1$-chain complex of $G$ with the orientation described in Section \ref{subsection complex in rank r}, we will use the following identifications: $\Ctcell_0(G) = \ZA[\~w_0T]$ and $\Ctcell_1(G) = \underset{i = 1}{\overset{r}{\oplus}} \KMW_1\otimes \ZA[\tilde{w}_iT]$. 

We now determine the differential $\partial_1: \Ctcell_1(G) \to \Ctcell_0(G)$.  We will continue to use the notation used in Section \ref{section Bruhat decomposition} (more precisely, Sections \ref{subsubsection codim1} and \ref{subsubsection orientations}) regarding the Bruhat cellular structure on $G$.

We use the above horizontal reduced expressions to identify the factor $\HA_1(Th(\nu_{Y_i}))$ of $\Ccell_1(G)$ corresponding to $Y_i$ in the following way.  We set $\overline{\Omega}_{0i}:=\Omega_0 \cup Y_i$, where $\Omega_0 = Bw_0B$ and $Y_i:= Bw_iB$.  It is easy to see that $\overline{\Omega}_{0i}$ is an open subset in $G$ and the product in $G$ induces an isomorphism of smooth $k$-schemes 
\[
\Omega_{\ell}\times_B \dots\times_B\Omega_{\lambda_{i+1}}\times_B P_{\lambda_i} \times_B \Omega_{\lambda_{i-1}} \times_B\dots \times_B \Omega_1 \xrightarrow{\simeq} \overline{\Omega}_{0i},
\]
where $P_{\lambda_i}$ is the parabolic subgroup corresponding to the root $\beta_{\lambda_i}$. Indeed, $P_{\lambda_i}$ is the union of its big cell $\Omega_{\lambda_i}$ and its translate $\dot{\sigma}_{\lambda_i}(\Omega_{\lambda_i})$, so that the above iterated fiber product is the union of the two open subschemes $\Omega_{\ell}\times_B \dots\times_B\Omega_{\lambda_{i+1}}\times_B \Omega_{\lambda_i} \times_B \Omega_{\lambda_{i-1}} \times_B\dots \times_B \Omega_1$ and $\Omega_{\ell}\times_B \dots\times_B\Omega_{\lambda_{i+1}}\times_B \dot{\sigma}_{\lambda_i}(\Omega_{\lambda_i}) \times_B \Omega_{\lambda_{i-1}} \times_B\dots \times_B \Omega_1$ and the morphism induced by the product is clearly a bijection onto $\overline{\Omega}_{0i}$ (using the Bruhat decomposition) and its restriction on each of the previous open subschemes is an isomorphism onto an open subscheme of $\overline{\Omega}_{0i}$: its big cell $\Omega_{\ell}\dots\Omega_{\lambda_i} \dots \Omega_1$ and its translate $\dot{\Sigma}.\Omega_{\ell}\dots\Omega_{\lambda_i} \dots \Omega_1$ with $\dot{\Sigma} = (\dot{\sigma}_{\ell} \cdots \dot{\sigma}_{\lambda_{i}+1}).\dot{\sigma}_{\lambda_{i}}.(\dot{\sigma}_{\ell} \cdots \dot{\sigma}_{\lambda_{i}+1})^{-1}$.
For every $\mu\in\{1,\dots,\ell\}$, the big cell in the parabolic subgroup 
\[
P_{\mu} = B\sigma_{\mu}B \amalg B
\] 
corresponding to $\beta_{\mu}$ is given by 
\[
\Omega_\mu = U_{\beta_{\mu}}\sigma_{\mu}B = B\sigma_{\mu}B = B\sigma_{\mu}U_{\beta_{\mu}}.
\]
Using these descriptions, we get an isomorphism 
\[
\overline{\Omega}_{0i} \cong U_{\beta_\ell}\dot{\sigma}_\ell \cdots U_{\beta_{\lambda_i+1}}\dot{\sigma}_{\lambda_i+1} P_{\lambda_i} \dot{\sigma}_{\lambda_i-1} U_{\beta_{\lambda_i-1}} \cdots \dot{\sigma}_1 U_{\beta_1}
\]
as well as the isomorphisms
\[
\Omega_0 = \Omega_{\ell}\times_B \cdots \times_B \Omega_1 \cong U_{\beta_\ell}\dot{\sigma}_\ell \cdots U_{\beta_1}\dot{\sigma}_1 T  U
\]
and 
\[
\begin{split}
 Y_i & = \Omega_{\ell}\times_B \cdots\times_B \Omega_{\lambda_i+1}\times_B B\times_B \Omega_{\lambda_i-1}\times_B\dots\times_B\Omega_1 \\
& \cong U_{\beta_\ell} \dot{\sigma}_\ell \cdots U_{\beta_{\lambda_i+1}}\dot{\sigma}_{\lambda_i+1} B \dot{\sigma}_{\lambda_i-1} U_{\beta_{\lambda_i-1}}\cdots \dot{\sigma}_1 U_{\beta_1},
\end{split}
\]
for every $i$.  Let $\tilde{\gamma}_i: G \dashrightarrow \A^1$ be the rational map on $G$ defined on $\Omega_0$ by the formula 
\begin{equation}
\label{eq:gammai}
\~\gamma_i(g_{\beta_\ell}\dot{\sigma}_\ell \cdots g_{\beta_1}\dot{\sigma}_1 t  u):= (g_{\beta_{\lambda_i}})^{-1},
\end{equation}
for all $g_{\beta_\mu} \in U_{\beta_\mu}$, $t \in T$ and $u \in U$.  Observe that this definition depends upon the choice of the  pinning of $G$.  Using the above description of $Y_i$ we see that the open subscheme 
\begin{equation}\label{eq:tub}
\sV_i:= U_{\beta_\ell} \dot{\sigma}_\ell \cdots U_{\beta_{\lambda_i+1}}\dot{\sigma}_{\lambda_i+1} U_{-\beta_{\lambda_i}}TU \dot{\sigma}_{\lambda_i-1}U_{\beta_{\lambda_i-1}}\cdots \dot{\sigma}_1 U_{\beta_1}
\end{equation}
of $\overline{\Omega}_{0i}$ can be considered as an open tubular neighborhood of $Y_i$. Through the pinning $u_{-\beta_{\lambda_i}}:\Ga \stackrel{\simeq}{\to} U_{-\beta_{\lambda_i}}$ we get a regular function on $\sV_i$, whose zero locus is exactly $Y_i$.  The corresponding rational function on $G$ will be denoted by $\check{\gamma}_i$.

In a dual way, we observe that the open subscheme
\begin{equation}\label{eq:tub'}
U_{\beta_\ell} \dot{\sigma}_\ell \cdots U_{\beta_{\lambda_i+1}}\dot{\sigma}_{\lambda_i+1} UTU_{-\beta_{\lambda_i}} \dot{\sigma}_{\lambda_i-1}U_{\beta_{\lambda_i-1}}\cdots \dot{\sigma}_1U_{\beta_1}
\end{equation}
of $\overline{\Omega}_{0i}$ can also be considered as an open tubular neighborhood of $Y_i$, and gives rise to another regular function on this tubular neighborhood whose zero locus is exactly $Y_i$. The corresponding rational function on $G$ is denoted by ${}^t\tilde{\gamma}_i$. 

\begin{lemma} 
\label{lem:gammai} Using the above notations, we have the following.
\begin{enumerate}[label=$(\alph*)$]
\item One has $\tilde{\gamma}_i = \check{\gamma}_i$.  Moreover, for any $t\in T$ and any $x\in \Omega_0$, one has
\[
\tilde{\gamma}_i (xt) = \tilde{\gamma}_i(x)
\]
and 
\[
\tilde{\gamma}_i (tx)  = \alpha'_i(t)\cdot \tilde{\gamma}_i (x),
\]
where $\alpha'_i = w_0(-\alpha_i)$.

\item For any $t\in T$ and any $x\in \Omega_0$, one has
\[
{}^t \tilde{\gamma}_i (tx) = {}^t\tilde{\gamma}_i(x)
\]
and 
\[
{}^t\tilde{\gamma}_i (xt) = \alpha_i(t)^{-1} \cdot {}^t\tilde{\gamma}_i (x).
\]
\end{enumerate}
\end{lemma}

\begin{proof} 
The first statement in $(a)$ follows from Lemma \ref{lem:comp} and an easy computation. The equality $\tilde{\gamma}_i (xt) = \tilde{\gamma}_i(x)$ is clear by definition of $\~\gamma_i$. For the last statement, note that for all $g_{\beta_\mu} \in U_{\beta_\mu}$, $t,t' \in T$ and $u \in U$, we have
\[
tg_{\beta_\ell}\dot{\sigma}_\ell \cdots g_{\beta_1}\dot{\sigma}_1 t' u = g'_{\beta_\ell}\dot{\sigma}_\ell \cdots g'_{\beta_1}\dot{\sigma}_1 tt' u,
\]
where $g'_{\beta_{i}} = \sigma_{i+1}\circ \cdots \circ \sigma_{\ell}(t) \cdot g_{\beta_{i}} \cdot \left(\sigma_{i+1}\circ \cdots \circ \sigma_{\ell}(t)\right)^{-1}$, for every $i$.  Thus, in particular, we have $g'_{\beta_{\lambda_i}}= (w''_i)^{-1}(\beta_{\lambda_i})(t) \cdot g_{\beta_{\lambda_i}}$ and consequently, for $x = g_{\beta_\ell}\dot{\sigma}_\ell \cdots g_{\beta_1}\dot{\sigma}_1 t' u$, we have $\~\gamma_i(tx) = w''_i(\beta_{\lambda_i})(t) \cdot \~\gamma_i(x)$.  Now, it remains to show that $w''_i(\beta_{\lambda_i}) = \alpha'_i$.  Since
\[
\sigma_{\ell}\cdots \sigma_{1} = w_0 = s'_is'_iw_0 = s'_i \sigma_{\ell}\cdots \widehat{\sigma}_{\lambda_i}\cdots \sigma_{1}, 
\]
we have $s'_i\sigma_{\ell}\cdots \sigma_{\lambda_i+1} = \sigma_{\ell}\cdots \sigma_{\lambda_i}$.  By Lemma \ref{lm:red2} applied to $w=w_i''=\sigma_{\lambda_\ell}\cdots \sigma_{\lambda_i+1}$, it follows that $w''_i(\beta_{\lambda_i}) = \alpha'_i$.  This completes the proof of $(a)$.  The proof of $(b)$ can be obtained in a similar way using the relation $(w'_i)^{-1}(\beta_{\lambda_i}) = \alpha_i$ and is left to the reader.
\end{proof}

The following lemma compares the orientation $\theta_i$ and ${}^t\theta_i$ of the normal bundle $\nu_{Y_i}$ of $Y_i$ in $G$ obtained from $\tilde{\gamma}_i$ and ${}^t\tilde{\gamma}_i$, respectively.

\begin{lemma}\label{lem:comptheta} 
We have the following equality of trivializations as well as orientations of $\nu_{Y_i}$ (on $Y_i$):
\[
{}^t\theta_i = {x}_{-\beta_{\lambda_i}} \cdot \theta_i,
\]
where ${x}_{-\beta_{\lambda_i}}: Y_i \to \G_m$ is the invertible function obtained by composing the projection $Y_i = B\tilde{w}_iTU \to T$ and the character associated with the root $-\beta_{\lambda_i}$.
\end{lemma}

\begin{proof} 
The assertion follows from the comparison of the tubular neighborhoods and the fact that $\sO(Y_i)^\times = \sO(\tilde{w}_iT)^\times$. Then one may explicitly compute this unit on elements of the form $\tilde{w}_i \cdot t$ with $t\in T$ and use the fact that the adjoint automorphism $Ad_{t}:\mathfrak{u}_{-\beta_{\lambda_i}}\cong \mathfrak{u}_{-\beta_{\lambda_i}}$ is the multiplication by $\beta_{\lambda_i}(t)^{-1}$.
\end{proof}

Unless otherwise stated, we will always use for $\nu_{Y_i}$ the right and left $T$-equivariant trivialization $\theta_i$ (which depends upon $\tilde{\gamma_i}$).  The factor $\HA_1(Th(\nu_{Y_i}))$ of $\Ccell_1(G)$ corresponding to $Y_i$ can now be described as the sheaf
\[
\HA_1(Th(\nu_{Y_i})) = \KMW_1 \otimes \ZA[Y_i]  \cong \KMW_1 \otimes \ZA[\tilde{w}_iT].
\]
This description depends on the weak pinning of $G$ and the choice of the reduced expression \eqref{eq:redw0} of $w_0$. Observe that the above tensor product is a tensor product of left and right $\ZA[T]$-modules over $\Z$.  The left and right $\ZA[T]$-module structure is as usual obtained by using the diagonal morphism
\[
\Psi_T: \ZA[T] \to \ZA[T]\otimes \ZA[T]; \quad [t] \mapsto [t] \otimes [t] = [t,t] \in \ZA[T]\otimes \ZA[T] = \ZA[T\times T].
\]

\begin{remark}
\label{rem twist}
Let $M$ be a right and left $\ZA[T]$-module.  Let $\sigma$ be an automorphism of $T$.  We define the left twist
\[
\Z[\sigma]\otimes M 
\]
of $M$ by $\sigma$ to be the right and left $\ZA[T]$-module $\Z[\sigma]\otimes M$ with the obvious right $\ZA[T]$-module structure, and with the left $\ZA[T]$-module structure $\odot$ induced by:
\[
t \odot ([\sigma] \otimes m) = [\sigma] \otimes (\sigma^{-1}(t) \cdot m).
\]
We analogously define the right twist of $M$ by $\sigma$:
\[
M \otimes \Z[\sigma]
\]
with
\[
(m \otimes [\sigma] )\odot t = (m \cdot \sigma(t)) \otimes [\sigma]
\]
giving the right $\ZA[T]$-module structure.  
\end{remark}

In order to describe the differential $\partial_1:\Ctcell_1(G) \to\Ctcell_0(G)$, it will be convenient to use the canonical isomorphism of right and left $\ZA[T]$-modules
\[
\begin{split}
\KMW_1 \otimes \ZA[\tilde{w}_i T] &\xrightarrow{\simeq} \Z[\tilde{w}''_i]\otimes (\KMW_1 \otimes \ZA[T])\otimes \Z[\tilde{w}'_i];\\
(u)\otimes [\tilde{w}_i \cdot t] & \mapsto \tilde{w}''_i\otimes \big( (u) \otimes [\tilde{w}'_i(t)]\big)\otimes \tilde{w}'_i,
\end{split}
\]
with the right and left $\ZA[T]$-module structures given in Remark \ref{rem important1} and Remark \ref{rem twist}.

\begin{lemma} 
\label{lem: partial1}
The restriction of $\partial_1: \Ctcell_1(G) \to \Ctcell_0(G)$ to the factor $\KMW_1 \otimes \ZA[\tilde{w}_iT]\subset \Ctcell_1(G)$ is the morphism (of right and left $\ZA[T]$-modules):
\[
\partial_1: \KMW_1 \otimes \ZA[\tilde{w}_iT] \to \ZA[\tilde{w}_0 T]
\]
induced by 
\[
(u)\otimes [\tilde{w}_i \cdot 1] \mapsto (\tilde{w}_0 \cdot \alpha_i^\vee(u)) \in \ZA(\~w_0T) \subset \ZA[\tilde{w}_0 T].
\]
\end{lemma}

\begin{proof}
From the discussion preceding Lemma \ref{lem:gammai}, it follows that the restriction of $\partial_1$ to the factor $\KMW_1 \otimes \ZA[\tilde{w}_iT]$ is precisely the differential in the complex
\[
\Z[\tilde{w}''_i]\otimes \Ctcell_*(P_{\lambda_i}) \otimes \Z[\tilde{w}'_i]. 
\]
The formula then immediately follows from the rank $1$ case already treated in Section \ref{subsection complex in rank 1}, see \eqref{eq:descriptionpartial}:
\[
\begin{split}
\partial_1: (u)\otimes [\tilde{w}_i \cdot 1] & \mapsto \tilde{w}''_i \cdot \dot{\sigma}_{\lambda_i}\cdot (\beta^\vee_{\lambda_i}(u)) \cdot \tilde{w}'_i\\
& = \tilde{w}''_i \cdot \dot{\sigma}_{\lambda_i}\cdot \tilde{w}'_i \cdot ((\tilde{w}'_i)^{-1}\beta^\vee_{\lambda_i}(u))\\
& = \tilde{w}_0 \cdot (\alpha_i^\vee(u)) = (\tilde{w}_0 \cdot\alpha_i^\vee(u)),
\end{split}
\]
where we have used the fact that $(\tilde{w}'_i)^{-1}(\beta^\vee_{\lambda_i}) = \alpha_i^\vee$, which follows from Lemma \ref{lm:red2}.
\end{proof}

Thus, we obtain a complete description of the first differential $\partial_1: \Ctcell_1(G) \to \Ctcell_0(G)$ in $\Ctcell_*(G)$.  Using this and our computations in Section \ref{subsection Z[T]}, we will now describe the sheaf $Z_1(G) := \Ker(\partial_1)$ and its quotient $Z_1(G)\otimes_{\ZA[T]} \Z$.  Let 
\[
\theta_0:  \ZA[T] \xrightarrow{\simeq}  \ZA[\tilde{w}_0T]
\]
denote the obvious isomorphism $[t]\mapsto [\tilde{w}_0t]$ and
\[
\theta_1: \underset{i=1}{\overset{r}{\oplus}} ~ \KMW_1 \otimes \ZA[T] \xrightarrow{\simeq} \underset{i=1}{\overset{r}{\oplus}} ~ \Z[\tilde{w}''_i] \otimes \KMW_1 \otimes \ZA[\tilde{w}'_iT] = \Ctcell_1(G)
\]
be the obvious isomorphism, induced by 
$(u)\otimes [t] \mapsto \tilde{w}''_i \otimes (u) \otimes [\tilde{w}'_i \cdot t]$ 
on the direct summand indexed by $i$.  The composition $\theta_0^{-1}\circ \partial_1 \circ \theta_1$ is the morphism
\[
\~\partial_1: \underset{i=1}{\overset{r}{\oplus}} ~ \KMW_1 \otimes \ZA[T]  \to  \ZA[T] 
\]
induced by 
$(u)\otimes [t] \mapsto (\tilde{w}_0 \cdot \alpha_i^\vee(u))$ 
on the direct summand indexed by $i$. Observe that the morphism $\~\partial_1$ only depends on the isomorphism given by the ``decomposition'' of $T$ into coroots:
\[
\prod_{i=1}^r \alpha^\vee_i : \prod_{i=1}^r \G_m. \xrightarrow{\simeq} T.
\]
Note further that the differential $\~\partial_1$ precisely agrees with the first differential in the complex $C_*$ (see \eqref{eq:C*} in Section \ref{subsection Z[T]}).  As a direct consequence of Lemma \ref{lem: Z1}, we get the following description of the quotient $Z_1(G)\otimes_{\ZA[T]} \Z$ of $Z_1(G)$.

\begin{lemma} \label{lem:Z1modT} The morphism of sheaves 
\[
\left(\underset{i}{\oplus}~ \KMW_2 \right) \oplus \left( \underset{(i,j); ~i<j}{\oplus} \KM_2 \right) \to Z_1(G) \otimes_{\ZA[T]} \Z 
\]
induced for each $i$ by
\[
(u,v) \mapsto \phi_{i}(u,v): =  (u) \otimes (\tilde{w}_i\cdot \alpha^\vee_i(v)) - (\eta(u)(v))\otimes [\tilde{w}_i \cdot 1_T]  
\]
(in $\KMW_1 \otimes \ZA[\tilde{w}_i.T]$) 
and for each $(i,j)$ with $i \neq j$ by
\[
(u, v) \mapsto \delta_{ij}(u,v):=  (u) \otimes (\tilde{w}_i \cdot \alpha^\vee_j(v)) - (v)\otimes (\tilde{w}_j \cdot \alpha^\vee_i(u)) 
\]
is an isomorphism.
\end{lemma}

\begin{remark}
The above isomorphism in Lemma \ref{lem:Z1modT} depends only on the weak-pinning of $G$ fixed at the beginning.  We will see below that its dependence on the choices of reduced expressions is very mild (see \ref{cor:changeexpression}). In fact, its dependence on the weak pinning itself is also very mild.  Indeed, when one changes the weak pinning, one gets a multiplication by a unit of the form $\<\pi\>$ with $\pi \in k^{\times}$ on each $\KMW_2$ factor, but the identity on the sum of Milnor-K-theory factors, as units in $\GW(k)$ act trivially on $\KM_2$. 
\end{remark}

\subsection{The differential in degree \texorpdfstring{$2$}{2}} \hfill
\label{subsection differential in degree 2} 

In this section, we determine the morphism
\[
\Ctcell_2(G)\otimes_{\ZA[T]} \Z \to Z_1(G)\otimes_{\ZA[T]} \Z, 
\]
whose cokernel is $\Hcell_1(G)$ by \eqref{eq: H1cokernel}. With our conventions and notation fixed earlier in this section and Notation \ref{notation codim 2 indices}, we  have 
\[
\Ctcell_2(G)  = \underset{[i,j]}{\oplus} ~ \KMW_2 \otimes \ZA[\tilde{w}_{[i,j]}T].
\]
We use the reduced expression \eqref{eq:redw0} of $w_0$.  Let $w\in W$ be an element of length $\ell-2$.  There is a pair $(i,j) \in \{1, \ldots, r\}^2$ with $w = w_0s_is_j = w_{[i,j]}$.  The pair $(i,j)$ is unique except in the case $s_i$ and $s_j$ commute, in which case $(j,i)$ is the only other pair satisfying $w=w_{[j,i]}$.

\begin{notation}
\label{notation lambda_i mu_j}
For every $i$, let $\lambda_i\in\{1,\dots,\ell\}$ be the index given by Lemma \ref{lm:red1} so that $\sigma_\ell \cdots \widehat{\sigma}_{\lambda_i} \cdots \sigma_1$ is a reduced expression of $w_i = w_0s_i$.  By Lemma \ref{lm:red1} again, it follows that there is a unique index $\mu_j\in \{1,\dots,\ell\}-\{\lambda_i\}$ such that a reduced expression of $w_{[i,j]}$ can be obtained by removing $\sigma_{\mu_j}$ in the above reduced expression of $w_i$. 

\end{notation}

\begin{lemma} 
\label{lm:lambdaij} 
With the above notation, the following conditions are equivalent:
\begin{enumerate}[label=$(\alph*)$]
\item $\lambda_i > \mu_j$;

\item $\lambda_i > \lambda_j$.
\end{enumerate}
Moreover, if both the conditions hold, then $\mu_j = \lambda_j$.
\end{lemma}

\begin{proof} $(a) \Rightarrow (b)$:  Suppose that $\lambda_i > \mu_j$, so that $w_0s_is_j = \sigma_\ell \cdots \widehat{\sigma}_{\lambda_i} \cdots \widehat{\sigma}_{\mu_j} \cdots \sigma_1$. Since $w_0 s_i = (w_0 s_i s_j) s_j$, it follows after cancelling on the left up to $\sigma_{\mu_j +1}$ that:
\[
\sigma_{\mu_j}\cdots \sigma_1 = {\sigma}_{\mu_j-1} \cdots \sigma_1 s_j, 
\]
which implies that $\sigma_\ell\cdots\sigma_1 = \sigma_\ell\cdots\widehat{\sigma}_{\mu_j}\cdots\sigma_1 s_j$.  Hence, we obtain
$$\sigma_\ell\cdots\widehat{\sigma}_{\lambda_j}\cdots\sigma_1 = w_0 s_j = \sigma_\ell\cdots\sigma_1 s_j = \sigma_\ell\cdots\widehat{\sigma}_{\mu_j}\cdots\sigma_1$$ 
and this implies that $\mu_j = \lambda_j$ by the uniqueness assertion in Lemma \ref{lm:red1}.

$(b) \Rightarrow (a):$  Suppose that $\lambda_i > \lambda_j$. From $w_0s_j = \sigma_\ell \cdots \widehat{\sigma}_{\lambda_j}\cdots\sigma_1$, we get ${\sigma}_{\lambda_j-1}\cdots \sigma_1 = \sigma_{\lambda_j}\cdots \sigma_1 s_j$, after cancelling on the left. Now, using $\lambda_i>\lambda_j$ and multiplying on the left by $\sigma_\ell \cdots \widehat{\sigma}_{\lambda_i}\cdots\sigma_{\lambda_j+1}$, we obtain $\sigma_\ell\cdots \widehat{\sigma}_{\lambda_i}\cdots\widehat{\sigma}_{\lambda_j}\cdots\sigma_1 = \sigma_\ell\cdots \widehat{\sigma}_{\lambda_i}\cdots\sigma_1  s_j$.  Thus,  
\[
\sigma_\ell\cdots\widehat{\sigma}_{\lambda_i}\cdots\sigma_1 = w_0s_i = (w_0s_is_j) s_j= \sigma_\ell\cdots \widehat{\sigma}_{\lambda_i}\cdots\widehat{\sigma}_{\lambda_j}\cdots\sigma_1 s_j = \sigma_\ell\cdots\widehat{\sigma}_{\lambda_j}\cdots\sigma_1,
\]
which implies that $\lambda_j = \mu_j$, again by the uniqueness assertion in Lemma \ref{lm:red1} applied to $w_0s_i = \sigma_\ell \cdots \widehat{\sigma}_{\lambda_i}\cdots \sigma_1$.
\end{proof}

\begin{remark} \hfill
\label{rem:mu=lambda}
\begin{enumerate}
\item In a similar way, Lemma \ref{lm:red1} implies that there is a unique index $\mu_i\in \{1,\dots,\ell\}-\{\lambda_j\}$ so that the expression obtained from the above one for $w_j$ is a reduced expression of $w_{[i,j]}$. One may prove the analogue of the previous lemma; that is, $\mu_i > \lambda_j$ if and only if $\lambda_i > \lambda_j$ and if both the conditions hold, then $\mu_i = \lambda_i$.

\item Observe that if $\lambda_i>\lambda_j$, then $\mu_i = \lambda_i$ and $\mu_j = \lambda_j$. In fact, the following conditions are equivalent:
\begin{enumerate}[label=(\roman*)]
\item $\mu_i = \lambda_i$;

\item $\mu_j = \lambda_j$;

\item $\lambda_i>\lambda_j$ or $\alpha_i$ and $\alpha_j$ are not adjacent in the Dynkin diagram of $G$. 
\end{enumerate}
Indeed, if $\lambda_i<\lambda_j$ then $w_{ji} = \sigma_\ell\cdots \widehat{\sigma}_{\lambda_j} \cdots \widehat{\sigma}_{\lambda_i} \cdots \sigma_1$, which is also $w_{ij}$ by (i) or (ii). Thus, we get $s_is_j = s_js_i$.
\end{enumerate}
\end{remark}

Let $Z_{ij}:=Bw_{[i,j]}B$ denote the Bruhat cell corresponding to $w_{[i,j]}$, and recall that $Y_i$ and $Y_j$ denote the Bruhat cells $Bw_iB$ and $Bw_jB$ corresponding to $w_i$ and $w_j$, respectively. Observe that $Y_i$ and $Y_j$ are exactly the two cells of codimension $1$ in $G$ whose closure contains $Z_{ij}$. Let $\overline{Y}_i$ be the union $Y_i\amalg Z_{ij}$ and $\overline{Y}_j$ be the union $Y_j\amalg Z_{ij}$, for all $i\neq j$.  These are all smooth $k$-subschemes of $G$.  It follows, as we already observed above, that the differential 
\[
\partial_2: \Ccell_2(G)\to \Ccell_1(G)
\]
restricted to the summand of $\Ccell_2(G)$ corresponding to $Z_{ij}$ has its image contained in the direct sum of the two summands of $\Ccell_1(G)$ corresponding to $Y_i$ and $Y_j$.  We will denote by $\partial_2^{(i)}: \HAred_2 (Th(\nu_{ij})) \to \HAred_1(Th(\nu_{i}))$ and $\partial_2^{(j)}: \HAred_2 (Th(\nu_{ij})) \to \HAred_1(Th(\nu_{j}))$ the corresponding components, where for all $i \neq j$, we denote by $\nu_i:=\nu_{Y_i}$ the normal bundle of the immersion $\iota_i:Y_i\hookrightarrow G$ (it is a locally free sheaf of rank $1$ over $\overline{Y}_i$) and $\nu_{ij}$ denotes the normal bundle of the immersion $Z_{ij}\hookrightarrow G$.  We will denote by $\overline{\nu}_i$ the normal bundle of the immersion $\overline{Y}_i\hookrightarrow G$, for every $i$.  The multiplication (in $G$) induces an isomorphism 
\begin{equation}\label{eq: Yi}
\overline{Y}_i \cong \Omega_{\ell}\times_B \dots\times_B P_{\mu_j}\times_B \dots \times_B B \times_B\dots\times_B \Omega_1 
\end{equation}
with $B$ at the $\lambda_i$-th place as well as
\begin{equation}\label{eq: Yj}
\overline{Y}_j \cong \Omega_{\ell}\times_B \dots\times_B P_{\mu_i}\times_B \dots \times_B B \times_B\dots\times_B \Omega_1
\end{equation}
with $B$ at the $\lambda_j$-th place (note that \eqref{eq: Yi} depends upon whether $\lambda_i > \mu_j$ or otherwise and hence, $B$ may appear to the left of $P_{\mu_j}$; similarly for \eqref{eq: Yj}).

Let $\tilde{X}_{ij}: = \Omega_0 \amalg Y_i \amalg Y_j \amalg Z_{ij}$, which is an open subscheme of $G$.  Let $\nu_{ij/i}$ denote the normal bundle of the closed immersion $Z_{ij} \hookrightarrow \overline{Y}_i$ and $\nu_{ij/j}$ denote that of the closed immersion $Z_{ij} \hookrightarrow \overline{Y}_j$. In the cartesian square of closed immersions
\[
\begin{array}{ccc} Z_{ij} & \subset & \-{Y}_i \\ \cap & & \cap \\ \overline{Y}_j & \subset & \tilde{X}_{ij}\end{array}
\]
the closed immersions $\overline{Y}_i \subset  \tilde{X}_{ij}$ and $\overline{Y}_j \subset  \tilde{X}_{ij}$ are transversal.  We thus obtain the identifications
\[
\nu_{ij/i} = (\overline{\nu}_j)|_{Z_{ij}}
\]
and 
\[
 \nu_{ij/j} = (\overline{\nu}_i)|_{Z_{ij}}
\]
as well as a canonical isomorphism
\begin{equation}\label{eq:nuij}
(\overline{\nu}_i)|_{Z_{ij}}\oplus (\overline{\nu}_j)|_{Z_{ij}}   \cong\nu_{ij}.
\end{equation}

We will use the following well-known generalization of $\A^1$-homotopy purity:

\begin{lemma} 
\label{lem:cof} 
There is a canonical $\A^1$-cofiber sequence of the form:
\[
 Th(\nu_i) \to Th(\overline{\nu}_i) \to Th(\nu_{ij}).
\]
\end{lemma}

\begin{proof} One considers the flag of open subsets
\[
\tilde{X}_{ij}-\overline{Y}_{i}\subset \tilde{X}_{ij}-Z_{ij}\subset\tilde{X}_{ij}
\]
in $\tilde{X}_{ij}$ and obtains a cofibration sequence
\[
\big( \tilde{X}_{ij}- Z_{ij}\big) /\big( \tilde{X}_{ij}- \overline{Y}_{i} \big)\subset \tilde{X}_{ij}/\big( \tilde{X}_{ij}- \overline{Y}_{i} \big) \twoheadrightarrow \tilde{X}_{ij}/\big( \tilde{X}_{ij}-Z_{ij} \big)
\]
and applies the $\A^1$-homotopy purity theorem \cite[Theorem 2.23, page 115]{Morel-Voevodsky}.
\end{proof}

It follows from the very construction of the cellular $\A^1$-chain complex described in Section \ref{subsection cellular chain complex} that the connecting homomorphism 
\begin{equation}\label{eq:partial2}
\HAred_2 (Th(\nu_{ij})) \to \HAred_1(Th(\nu_{i}))
\end{equation}
in the $\A^1$-homology sequence of the $\A^1$-cofibration of Lemma \ref{lem:cof} is the restriction of the differential $\partial_2^{(i)}$ to the factor $\HAred_2 (Th(\nu_{ij}))\subset \Ccell_2(G)$, and analogously for $\partial^{(j)}_2$.  The following property uses the fact that $G$ is simply connected in the sense of algebraic groups.

\begin{lemma} 
\label{lem:Pic} 
The group $\Pic(\-{Y}_i)$ is trivial (and so is $\Pic(\-{Y}_j)$).
\end{lemma}
\begin{proof} 
The cofiber sequence of Lemma \ref{lem:cof} gives rise to the exact sequence
\[
\sO(Y_i)^\times \to H^1(Th(\nu_{ij/i});\G_m)\to \Pic(\-{Y}_i) \to \Pic(Y_i). 
\]
Now, $\Pic(Y_i) = 0$, since $Y_i$ is isomorphic to the product of $T$ with an affine space and $H^1(Th(\nu_{ij/i});\G_m) = \Z$.  The morphism $\sO(Y_i)^\times \to H^1(Th(\nu_{ij/i});\G_m) = \Z$ is the valuation at $Z_{ij}$. Recall from the notation at the beginning of Section \ref{section Bruhat decomposition} that $X$ and $Y$ denote the character group and the cocharacter group of $T$. Since $G$ is simply connected, the pairing $X\otimes Y\to \Z$ induces an isomorphism $X\to Y^*$ so that the fundamental weights form a basis of $X$.  Now, \cite[Lemma 4.2]{Brylinski-Deligne} implies that the valuation $\sO(Y_i)^\times \to \Z$ is surjective.
\end{proof}

\begin{remark} 
\label{rem lemma 4.2}
Note that \cite[Lemma 4.2]{Brylinski-Deligne} is not proved in \cite{Brylinski-Deligne}, and is referred to \cite{Demazure} (where it is shown by reduction to the case of $\SL_2$).  In fact, by reduction to the case of $\SL_2$, it is easy to prove this lemma using Lemma \ref{lem:comp}.
\end{remark}

It follows from Lemma \ref{lem:Pic} that one can uniquely determine a 
trivialization of $\nu_i$ by choosing a trivialization of $\overline{\nu}_i$.  Let $\overline{\theta}_i$ be a right and left $T$-equivariant trivialization of $\overline{\nu}_i$ and let $\theta_i$ denote the induced trivialization of $\nu_i$. Then the above $\A^1$-cofiber sequence 
\[
Th(\nu_i) \to Th(\overline{\nu}_i) \to Th(\nu_{ij})
\]
becomes canonically and right and left $T$-equivariantly isomorphic to
\begin{equation}\label{eq:cof}
\Sigma(\G_m)\wedge((Y_i)_+) \to \Sigma(\G_m)\wedge((\overline{Y}_i)_+) \to \Sigma(\G_m)\wedge Th(\nu_{ij/i}),
\end{equation}
where the factor $\Sigma(\G_m)$ on the rightmost term corresponds to  $\theta_i$ (note that $\nu_{ij}$ is the direct sum of $\nu_{ij/i}$ and a rank one trivial bundle). Observe that the axioms of symmetric monoidal triangulated categories implies that the connecting morphism 
\[
\Sigma(\G_m)\wedge Th(\nu_{ij/i}) \to \Sigma(\Sigma(\G_m)\wedge((Y_i)_+))
\]
of the triangle \eqref{eq:cof} is the composite of the smash product with $\Sigma(\G_m)$ of the connecting morphism $\partial_{ij/i}: Th(\nu_{ij/i})\to \Sigma((Y_i)_+)$ of the triangle $(Y_i)_+ \to (\overline{Y}_i)_+ \to Th(\nu_{ij/i})$ and the permutation $\Sigma(\G_m) \wedge \Sigma((Y_i)_+)\cong\Sigma(\Sigma(\G_m)\wedge((Y_i)_+)) $.

Thus, the restriction $\partial_2^{(i)}$ of the differential $\partial_2$ to the factor of $\Ccell_2(G)$ corresponding to $Z_{ij}$ is the morphism of sheaves
\[
\begin{split}
\KMW_1 \otimes \HA_1(Th(\nu_{ij/i})) &\xrightarrow{\simeq} \HA_2(\Sigma(\G_m)\wedge Th(\nu_{ij/i})) \to \KMW_1 \otimes \ZA[Y_i] \\
(u)\otimes x &\mapsto - \HA_1(\partial_{ij/i})(x),
\end{split}
\]
where $u$ stands for the factor $\G_m$ above coming from the restriction of $\overline{\theta}_i$ to $\nu_i$. The sign comes from the permutation morphism on the smash product of two $1$-dimensional simplicial spheres.  Observe that by Lemma \ref{lem:Thomiso}, a change of the trivialization preserving the orientation acts trivially on the $\A^1$-cofibration $\Sigma(\G_m)\wedge((Y_i)_+) \to \Sigma(\G_m)\wedge((\overline{Y}_i)_+)$, up to $\A^1$-homotopy. Thus the previous computations are only dependent on the orientations induced by $\-\theta_i$.

Now, we need to conveniently choose orientations (or trivializations) of both $\overline{\nu}_i$ and $\overline{\nu}_j$ so as to facilitate our computation of $\partial_2$.

\begin{lemma} \label{lem:fiberproduct} In the case $\lambda_i>\lambda_j$, the morphism of $k$-schemes
\begin{equation}
\label{eqn lem:fiberproduct}
\Omega_{\ell}\times_B \dots\times_B\Omega_{\lambda_{i+1}}\times_B P_{\beta_{\lambda_i}} \times_B \dots \times_B P_{\beta_{\lambda_j}} \times_B\dots\times_B \Omega_1 \to \tilde{X}_{ij}
\end{equation}
induced by taking the product in $G$ is an isomorphism of cellular $k$-schemes. 
\end{lemma}

\begin{proof} We proceed in the same way as in Section \ref{subsection differential in degree 1} to describe $\overline{\Omega}_{0i}$ as a fiber product. From a combinatorial point of view, the Bruhat decomposition implies that this morphism induces a bijection onto $\tilde{X}_{ij}$, which is open in $G$ (its complement in $G$ being a union of Bruhat cells stable under the Bruhat ordering). Now the big cell $\Omega_{\ell}\times_B \dots\times_B\Omega_{\lambda_{i+1}}\times_B \Omega_{\beta_{\lambda_i}} \times_B \dots \times_B \Omega_{\beta_{\lambda_j}} \times_B\dots\times_B \Omega_1$ and its four left translates, obtained by replacing the big cell in $P_{\beta_{\lambda_j}}$ and $P_{\beta_{\lambda_j}}$ by their left translate through $\dot{\sigma}_{\lambda_i}$ and or $\dot{\sigma}_{\lambda_j}$ respectively, cover both left and right hand side of the morphism \eqref{eqn lem:fiberproduct}.   The restriction of \eqref{eqn lem:fiberproduct} on each of these open sets is clearly an isomorphism of schemes.
\end{proof}

In view of Lemma \ref{lem:fiberproduct}, one may compute $\Ccell_*(\~X_{ij})$ using the rank one case and the K\"unneth formula in the case $\lambda_i > \lambda_j$.  However, we take a more direct approach.

\begin{lemma} \label{lem:gammaij}
We will use the notation of Lemma \ref{lem:gammai}. If $\lambda_i>\lambda_j$, then $\tilde{\gamma}_i$ is defined on a neighborhood of $\overline{Y}_i$ so that it defines a right and left $T$-equivariant trivialization $\overline{\theta}_i$ of $\overline{\nu}_i$ and ${}^t\tilde{\gamma}_j$ is defined on a neighborhood of $\overline{Y}_j$ so that it defines a right and left $T$-equivariant trivialization ${}^t\overline{\theta}_j$ of $\overline{\nu}_j$.
\end{lemma}

\begin{proof} 
This follows easily from Lemma \ref{lem:fiberproduct}. For instance, one sees that
\[
U_{\beta_\ell}\dot{\sigma}_\ell \cdots U_{\beta_{\lambda_i}+1}\dot{\sigma}_{\lambda_{i}+1}  U_{-\beta_{\lambda_i}}  U_{\beta_{\lambda_i}-1}\dot{\sigma}_{\lambda_{i}-1}\cdots U_{\beta_{\lambda_j}+1}\dot{\sigma}_{\lambda_{j}+1}  P_{\lambda_j} U_{\beta_{\lambda_j}-1}\dot{\sigma}_{\lambda_{j}-1}\cdots \dot{\sigma}_1 U_1
\]
is an open neighborhood of $\overline{Y}_i = U_{\beta_\ell}\dot{\sigma}_\ell \cdot \widehat{U}_{-\beta_{\lambda_i}} \cdots P_{\lambda_j}\cdots \dot{\sigma}_1 U_1$ on which $\tilde{\gamma}_i$ is defined and clearly defines $\overline{Y}_i$ as its zero locus. In the case of ${}^t\tilde{\gamma}_j$, one uses the open neighborhood
\[
U_{\beta_\ell}\dot{\sigma}_\ell \cdots U_{\beta_{\lambda_i}+1}\dot{\sigma}_{\lambda_{i}+1} P_{\beta_{\lambda_i}} U_{\beta_{\lambda_i}-1}\dot{\sigma}_{\lambda_{i}-1} \cdots U_{\beta_{\lambda_j}+1}\dot{\sigma}_{\lambda_{j}+1}  U_{-\lambda_j}  U_{\beta_{\lambda_j}-1}\dot{\sigma}_{\lambda_{j}-1} \cdots \dot{\sigma}_1 U_1
\]
of $\overline{Y}_j$ and proceeds similarly.
\end{proof}

\begin{remark} \hfill
\label{rem: triv}
\begin{enumerate}
\item One may show that under the assumptions of Lemma \ref{lem:gammaij}, $\tilde{\gamma}_j$ does not extend to $\overline{Y}_j$.

\item In the case $\lambda_i < \lambda_j$, there is no analogue of Lemma \ref{lem:fiberproduct}.  In this case, neither ${}^t\tilde{\gamma}_i$ nor $\tilde{\gamma}_i$ is defined on an open neighborhood of $\overline{Y}_i$.   Thus, we have to choose a right and left $T$-equivariant trivialization of $\overline{\nu}_i$ in another way, see Section \ref{subsection i<j}.
\end{enumerate}
\end{remark}

In view of Lemma \ref{lem:gammaij}, we will always use the following orientation of $\nu_{ij}$ in what follows. By \eqref{eq:nuij}, we have $\Lambda^2(\nu_{ij})= \overline{\nu}_i|_{Z_{ij}}\otimes \overline{\nu}_j|_{Z_{ij}}$. So in order to orient $\nu_{ij}$, it suffices to choose an orientation $\theta'_i$ of $\overline{\nu}_i$ and $\theta'_j$ of $\overline{\nu}_j$ and then take the tensor product of these. In the case $\lambda_i>\lambda_j$, we take for $\theta'_i$ the orientation induced by $\tilde{\gamma}_i$ and for $\theta'_j$ the one induced by ${}^t\tilde{\gamma}_j$. 

We thus get an identification of sheaves
\[
\KMW_1\otimes \KMW_1 \otimes \ZA[Z_{ij}] \cong \HAred_2 (Th(\nu_{ij}))
\]
with the first factor $\KMW_1$ corresponding to $\theta'_i$ and the second one corresponding to $\theta'_j$.  The units in $\G_m$ corresponding to symbols in $\KMW_1$ in the computations below will be denoted by $u$ for $\theta'_i$ and $v$ for $\theta'_j$. We also get identifications 
\[
\KMW_1 \otimes \ZA[Z_{ij}] \cong \HAred_1 (Th(\nu_{ij/i}))
\]
using $\theta'_j|_{Z_{ij}}$ as the orientation of $\nu_{ij/i}$ and similarly 
\[
\KMW_1 \otimes \ZA[Z_{ij}] \cong \HAred_1 (Th(\nu_{ij/j}))
\]
using $\theta'_i|_{Z_{ij}}$ as the orientation of $\nu_{ij/j}$.  Each of these isomorphisms is left and right $T$-equivariant, where each factor is equipped with the corresponding left and right $\ZA[T]$-module structure. For instance, the factor $\KMW_1$ corresponding to $\theta'_i$ is endowed with the left and right $\ZA[T]$-module structures coming from the equivariant properties of the function used to define the trivialization (see Lemma \ref{lem:gammai}).  If we use these orientations $\theta'_i$ of $\overline{\nu}_i$ and $\theta'_j$ of $\overline{\nu}_j$ to identify $\HAred_1 (Th(\nu_{i}))$ with $\KMW_1\otimes \ZA[Y_i]$ and $\HAred_1 (Th(\nu_{j}))$ with $\KMW_1\otimes \ZA[Y_j]$, then we obtain the following fact from the above discussion.

\begin{lemma}
\label{lem:diff2} 
The differential
\[
\partial^{(i)}_2: \KMW_1\otimes \KMW_1 \otimes \ZA[Z_{ij}] \cong \HAred_2 (Th(\nu_{ij}))  \to \HAred_1 (Th(\nu_{i})) \cong \KMW_1 \otimes \ZA[Y_{i}]
\]
in $\Ctcell_*(G)$ is given by 
\[
(u)\otimes (v) \otimes x\mapsto - (u)\otimes \HA_1(\partial_{ij/i})\big((v) \otimes x \big) 
\]
and the differential
\[
\partial^{(j)}_2: \KMW_1\otimes \KMW_1 \otimes \ZA[Z_{ij}] \cong \HAred_2 (Th(\nu_{ij})) \to \HAred_1 (Th(\nu_{j})) \cong \KMW_1 \otimes \ZA[Y_{j}]
\]
is given by 
\[
(u)\otimes (v) \otimes x \mapsto (v) \otimes \HA_1(\partial_{ij/j})\big((u) \otimes x \big).
\]
\end{lemma}

\begin{remark} 
The sign in the formula for $\partial^{(i)}_2$ in Lemma \ref{lem:diff2} has already been explained in the paragraph preceding Lemma \ref{lem:gammai}.  Observe that the sign disappears in the formula for $\partial^{(j)}_2$ for the following reason.  We have to first use the permutation $\Sigma(\G_m)\wedge \Sigma(\G_m) \cong \Sigma(\G_m)\wedge \Sigma(\G_m)$, which induces after applying $\HA_2$ the automorphism
\[
\KMW_1\otimes \KMW_1 \cong \KMW_1\otimes \KMW_1; \quad (u)\otimes (v) \mapsto - (v)\otimes (u)
\]
and then apply the formula analogous to that of $\partial_2^{(i)}$.  In the computation of the differential, we use again a permutation of the two simplicial circles as explained above, and this cancels the first sign.
\end{remark}

However, in Section \ref{subsection differential in degree 1}, we made the convention to always orient $\nu_i$ with $\tilde{\gamma}_i$ (see the paragraph after Lemma \ref{lem:comptheta}). This was made to be able to explicitly compute $Z_1(G)$ and $Z_1(G)\otimes_{\ZA[T]} \Z$. We have to take this into account in our explicit computations as the restriction of $\theta'_i$ need not be $\theta_i$ (Remark \ref{rem: triv}).

We also need to determine the degree $1$ differentials
\[
\HA_1(\partial_{ij/i}): \KMW_1 \otimes \ZA[Z_{ij}] \cong \HA_1(Th(\nu_{ij/i})) \to \HA_0(Y_i) = \ZA[Y_i]
\]
and
\[
\HA_1(\partial_{ij/j}): \KMW_1 \otimes \ZA[Z_{ij}] \cong \HA_1(Th(\nu_{ij/j})) \to \HA_0(Y_j) = \ZA[Y_j].
\]
The problem here is the following: the restriction of $\theta'_j$ on $Y_i$ is a local parameter for the closed immersion $Z_{ij}\subset Y_i$ and consequently, defines the orientation of $\nu_{ij/i}$ which is used in the determination of the differential. However, there is another local parameter, $\tilde{\gamma}_{ij/i}$, for the closed immersion $Z_{ij}\subset Y_i$, for which the morphism $\HA_1(\partial_{ij/i})$ is ``easy'' to compute, and it is not always the restriction of $\theta'_j$.  We now describe $\tilde{\gamma}_{ij/i}$.

The union $\overline{Y}_{ij}: = Y_i\amalg Z_{ij}$ is a smooth $k$-scheme isomorphic to
\[
\Omega_\ell\times_B \cdots \times_B P_{\mu_j} \times_B \cdots \times_B B \times_B \cdots \Omega_1
\]
(by taking the product in $G$), where $P_{\mu_j}$ is in the $\mu_j$th place and $B$ is in the $\lambda_i$th place. Thus, 
\[
\overline{Y}_{ij} \cong U_{\beta_\ell} \dot{\sigma}_\ell \cdots P_{\mu_j} \dot{\sigma}_{\mu_j-1}U_{\beta_{\mu_j-1}}\cdots \widehat{\sigma}_{\lambda_i}\widehat{U}_{\lambda_i} \cdots \dot{\sigma}_{1}U_{\beta_{1}}
\]
and  
\[
Z_{ij} \cong U_{\beta_\ell}\dot{\sigma}_\ell \cdots B_{\mu_j} \dot{\sigma}_{\mu_j-1}U_{\beta_{\mu_j-1}}\cdots \widehat{\sigma}_{\lambda_i}\widehat{U}_{\lambda_i} \cdots \dot{\sigma}_{1}U_{\beta_{1}}.
\]
Thus, it is easy to compute the morphism $\HA_1(\partial_{ij/i})$, if one chooses the local parameter defining $B$ in $P_{\mu_j}$ through the open neighborhood $U_{-\beta_{\mu_j}}\times B\subset P_{\mu_j}$ for the orientation as it was done in the determination of the rank one case in Section \ref{subsection complex in rank 1}. We denote by $\tilde{\gamma}_{ij/i}$ this local parameter.  

\begin{lemma}\label{lem:gammaij'} 
If $\lambda_i>\lambda_j$, then $\tilde{\gamma}_i|_{Y_j} = \tilde{\gamma}_{ij/j}$ and ${}^t\tilde{\gamma}_j|_{Y_i} = {}^t\tilde{\gamma}_{ij/i}$.
\end{lemma}

\begin{proof} 
This follows in the same way as the proof of Lemma \ref{lem:gammaij} and we leave the details to the reader.
\end{proof}

\begin{remark} \hfill
\label{rem triv comparison}
\begin{enumerate}
\item One may prove an analogous formula to Lemma \ref{lem:comptheta} comparing the trivializations $\theta_{ij/i}$ and ${}^t\theta_{ij/i}$ of $\nu_{ij/i}$ induced by $\tilde{\gamma}_{ij/i}$ and ${}^t\tilde{\gamma}_{ij/i}$:
\[
{}^t\theta_{ij/i} = \tilde{y}_{\beta_{\lambda_j}} \cdot \theta_{ij/i},
\]
where $\tilde{y}_{\beta_{\lambda_j}}: Z_{ij} \to \G_m$ is the invertible function obtained by composing the projection $Z_{ij} = B\tilde{w}_{ij}TU \to T$ and the character $\beta_{\lambda_j}$.

\item If $\lambda_i<\lambda_j$ and if $\alpha_i$ and $\alpha_j$ are adjacent in the Dynkin diagram, then it can be shown that ${}^t\tilde{\gamma}_i|_{Y_j}$ does not agree with ${}^t\tilde{\gamma}_{ij/j}$ and $\tilde{\gamma}_j|_{Y_i}$ does not agree with $\tilde{\gamma}_{ij/i}$ as well.

\item One may verify that for $G = \SL_3$ with the standard pinning, with $w_0 = s_2s_1s_2$ the chosen reduced expression of the longest word in $W$, we are in the case $\lambda_1 = 3 > 1 = \lambda_2$, and that the restriction of $\tilde{\gamma}_2$ to $Y_1$ is different from $\tilde{\gamma}_{12/1}$.
\end{enumerate}
\end{remark}

Now, it remains to put together the previous computations, but at this point it is convenient to distinguish the case $\lambda_i>\lambda_j$ from the case  $\lambda_i<\lambda_j$.

\subsection{Image of the degree \texorpdfstring{$2$}{2} differential: the case \texorpdfstring{$\lambda_i>\lambda_j$}{lambda i > lambda j}} \hfill
\label{subsection i>j}

In this case, we have $\lambda_j = \mu_j$ by Lemma \ref{lm:lambdaij} and this yields a canonical reduced expression of $w_{ij}$ obtained by removing two simple reflections in the chosen reduced expression \eqref{eq:redw0} of $w_0$:
\[
w_{[i,j]}:= w_0s_is_j =  \sigma_\ell\cdots\widehat{\sigma}_{\lambda_i}\cdots\widehat{\sigma}_{\lambda_j}\cdots\sigma_1.
\]
We write this expression in an obvious way as $w_{[i,j]} = w'''_{[i,j]}\widehat{\sigma}_{\lambda_i}w''_{[i,j]}\widehat{\sigma}_{\lambda_j} w'_{[i,j]}$, thereby defining $w'_{[i,j]}$, $w''_{[i,j]}$ and $w'''_{[i,j]}$. Thus, 
\[
w_i = w_0s_i =  w'''_{[i,j]} \widehat{\sigma}_{\lambda_i} w''_{[i,j]}\sigma_{\lambda_j} w'_{[i,j]}
\]
is the reduced expression for $w_i$ obtained by removing $\sigma_{\lambda_i}$ from that of $w_0$.  Consequently, $w''_i = w'''_{[i,j]}$ and $w''_{[i,j]} \sigma_{\lambda_j}w'_{[i,j]} = w'_i$ (see Notation \ref{notation horizontal codim 1}). Also,
\[
w_j = w_0s_j =  w'''_{[i,j]}\sigma_{\lambda_i}w''_{[i,j]}\widehat{\sigma}_{\lambda_j}w'_{[i,j]}
\]
is the reduced expression for $w_j$ obtained by removing $\lambda_j$ from that of $w_0$ and hence, we have $w''_j = w'''_{[i,j]}\sigma_{\lambda_i}w''_{[i,j]}$ and $w'_j= w'_{[i,j]}$.  We now fix the horizontal reduced expressions for Bruhat cells in codimension $2$.

\begin{notation}
\label{notation horizontal codim 2}
With the above notation, we set 
\[
\~w_{[i,j]}:= \dot{w}'''_{[i,j]} \dot{w}''_{[i,j]} \dot{w}'_{[i,j]} = \dot{\sigma_\ell}\cdots\dot{\sigma}_{\lambda_i+1}\dot{\sigma}_{\lambda_i-1}\cdots\dot{\sigma}_{\lambda_j+1}\dot{\sigma}_{\lambda_j-1}\cdots\dot{\sigma_1} \in N_G(T). 
\]
When there is no confusion, we will simply write $w_{ij}, w'_{ij}$ etc for $w_{[i,j]}, w'_{[i,j]}$ etc for the sake of brevity.
\end{notation}

We will use the identification $\ZA[Z_{ij}] \cong \ZA[\tilde{w}_{ij}T]$.  In that case, the restriction of $\theta'_i$ to $\nu_i$ is $\theta_i$ and it follows by Lemma \ref{lem:diff2} that
\begin{equation}
\label{eqn d_2^i}
\partial_2^{(i)} ((u)\otimes (v) \otimes [\tilde{w}_{ij} \cdot t]) = -(u) \otimes \HA_1(\partial_{ij/i})\big((v) \otimes [\tilde{w}_{ij} \cdot t]\big),
\end{equation}
for all $t \in T$.  Now, we want to compute $\HA_1(\partial_{ij/i})\big((v) \otimes [\tilde{w}_{ij} \cdot t]\big)$.  
By the very definition, $\HA_1(\partial_{ij/i})$ is the connecting morphism
\[
\HA_1(Th(\nu_{ij/i})) \to \HA_0(Y_i) = \ZA[\tilde{w}_iT]
\]
coming from the cofibration sequence $Y_i\subset \overline{Y}_i\to Th(\nu_{ij/i})$. If we use the orientation of $\nu_{ij/i}$ induced by $\tilde{\gamma}_{ij/i}$, we get an identification $$\HA_1(Th(\nu_{ij/i})) \cong \KMW_1\otimes \ZA[Z_{ij}] = \KMW_1\otimes \ZA[\tilde{w}_{ij}T],$$ and $\HA_1(\partial_{ij/i})$ is given by the formula for the rank $1$ case determined in Lemma \ref{lem: partial1}. Our conventions say that $v$ in this case comes from the orientation $\theta'_j$ induced by ${}^t\tilde{\gamma}_j$, which restricts on $Y_i$ exactly to ${}^t\tilde{\gamma}_{ij/i}$ by Lemma \ref{lem:gammaij'}.  Thus, in order to compute 
\[
\HA_1(\partial_{ij/i}) :\KMW_1\otimes \ZA[\tilde{w}_{ij}T] \to \ZA[\tilde{w}_iT]
\]
using our orientations, we first need to change the orientation of $\nu_{ij/i}$ to that induced by $\tilde{\gamma}_{ij/i}$, and then apply Lemma \ref{lem: partial1}, together with an appropriate twist. More precisely, with the previous orientations and conventions, the cellular $\A^1$-chain complex of $\overline{Y}_i$ is given by
\[
\Ctcell(\overline{Y}_i) = \Z[\dot{w}'''_{ij}\dot{w}''_{ij}]\otimes \Ctcell_*(P_{\lambda_j}) \otimes \Z[\dot{w}'_{ij}], 
\]
with $\Ctcell_*(P_{\lambda_j})$ being the chain complex: 
\[
\KMW_1\otimes \ZA[T] \xrightarrow{\partial} \ZA[\dot{\sigma}_{\lambda_j}T]; \quad \partial((v) \otimes [t]) \mapsto [\dot{\sigma}_{\lambda_j} \cdot (\beta^\vee_{\lambda_j}(v) \cdot t ) ],
\]
by Remark \ref{rem pin} (4).  We will use Remark \ref{rem triv comparison} and the following straightforward lemma.

\begin{lemma} 
\label{lem:autthom} 
Let $\zeta\in\sO^{\times}_X$ be an invertible function on $X$. The automorphism of $$\HA_1(Th(\sO_X)) = \KMW_1\otimes \ZA[X]$$ induced by the automorphism of the line bundle $\sO_X$ given by multiplication by $\zeta$ is 
given at every point $x \in X$ by multiplication by $\<\zeta(x)\>$ on the factor $\KMW_1$.
\end{lemma}

Therefore, since $\tilde{w}_i = \dot{w}'''_{ij}\dot{w}''_{ij}\dot{\sigma}_{\lambda_j} \dot{w}'_{ij}$, it can be verified that
\[
\HA_1(\partial_{ij/i}): \KMW_1\otimes \ZA[\tilde{w}_{ij}T] \to \ZA[\tilde{w}_iT]
\]
is given by
\begin{equation}
\label{eqn d_2^i'}
(v) \otimes [\tilde{w}_{ij} \cdot t] \mapsto \tilde{w}_i \cdot (\<\beta_{\lambda_j}(t)\> \cdot (v)_j) [t] 
\end{equation}
as $(\dot{w}'_{ij})^{-1}(\beta^\vee_{\lambda_j}) = \alpha^\vee_j$ by Lemma \ref{lm:red2} (as seen in the proof of Lemma \ref{lem:gammai}). 

Over a field extension $F$ of $k$, the sections of the sheaf $\KMW_1\otimes \ZA[\tilde{w}_{ij}T]$ are generated as a right $\ZA[T]$-module by elements of the form $(v)\otimes [\tilde{w}_{ij} \cdot 1]$ and the sections of the sheaf $\KMW_1\otimes \KMW_1\otimes \ZA[\tilde{w}_{ij}T]$ are generated as a right $\ZA[T]$-module by elements of the form $(u)\otimes (v)\otimes [\tilde{w}_{ij} \cdot 1]$, where $u, v \in F^{\times}$ and $1$ denotes the neutral element of $T(F)$. So it suffices to determine $\partial_2$ on these symbols.  Substituting \eqref{eqn d_2^i'} into \eqref{eqn d_2^i}, we obtain
\begin{equation}
\label{eqn d_2^i formula}
\partial^{(i)}_2((u)\otimes (v)\otimes [\tilde{w}_{ij} \cdot 1 ]) = - (u) \otimes \big( (v)_j \cdot [\tilde{w}_i \cdot 1] \big) 
\end{equation}
in $\KMW_1\otimes \ZA[\tilde{w}_iT]$.  
We now compute $\partial^{(j)}_2((u)\otimes (v)\otimes [\tilde{w}_{ij} \cdot 1 ])$. In that case the restriction of $\theta'_j$ to $\nu_j$ is $\beta_{\lambda_j}^{-1}\cdot\theta_j$ and the restriction of $\tilde{\gamma}_i$ to $Y_{j}$ is $\tilde{\gamma}_{ij/j}$ by Lemma \ref{lem:gammaij'}. Now, the chain complex of $\overline{Y}_j$ in this case is:
\[
\Ctcell(\overline{Y}_i) = \Z[\dot{w}'''_{ij}]\otimes \Ctcell_*(P_{\lambda_i}) \otimes \Z[\dot{w}''_{ij}\dot{w}'_{ij}] 
\]
with $\Ctcell_*(P_{\lambda_i})$ being the chain complex (by Remark \ref{rem pin} (4)): 
\[
\KMW_1\otimes \ZA[T] \xrightarrow{\partial} \ZA[\tilde{w}_jT], \quad \partial((u) \otimes [t]) = [\dot{\sigma}_{\lambda_i} \cdot (\beta^\vee_{\lambda_i}(u) \cdot t )]
\]
Now, since $w_0s_is_j = \sigma_\ell\cdots \widehat{\sigma}_{\lambda_i}\cdots\widehat{\sigma}_{\lambda_j}\cdots \dot\sigma_1 s_i  s_j$, we have $\dot{w}''_{ij}\dot{w}'_{ij} = \dot\sigma_{\lambda_i}\cdots \dot\sigma_1\dot{s}_i\dot{s}_j$.  We rewrite this as
\[
\begin{split}
\dot\sigma_{\lambda_i} \tilde{w}''_{ij}\tilde{w}'_{ij} & = \dot{\sigma}_{\lambda_i-1}\cdots\dot\sigma_1 \dot{s}_i \dot{s}_j\\
& = \dot{\sigma}_{\lambda_i-1}\cdots\widehat{\sigma}_{\lambda_j}\cdots\dot\sigma_1 \dot{s}_j \dot{s}_i \dot{s}_j,
\end{split}
\]
which follows from Lemma \ref{lm:red2} in case $s$ is the reflection $s_js_is_j$.  As observed in Remark \ref{rem:red2}, we have $s_j(\alpha_i)\in \Phi^+$.  Consequently, it follows that 
\[
(\dot{w}'_{ij})^{-1}\circ(\dot{w}''_{ij})^{-1}(\beta^\vee_{\lambda_i}) = s_j(\alpha_i^\vee).
\]
Thus, $\HA_1(\partial_{ij/j}): \KMW_1\otimes \ZA[\tilde{w}_{ij}T] \to \ZA[\tilde{w}_jT]$ is given by
\begin{equation}
\label{eqn d_2^j}
(v) \otimes [\tilde{w}_{ij} \cdot t] \mapsto \tilde{w}_j \cdot \big(\<\beta_{\lambda_i}(t)\> \cdot s_j((v)_i)\big) [t]. 
\end{equation}
However, note that the formula \eqref{eqn d_2^j} is with respect to the orientation $\theta_j$.  With respect to the orientation $\theta_j$, the formula for $\partial_2^{(j)}$ is given by Lemma \ref{lem:diff2} as follows:
\[
\partial_2^{(j)}((u)\otimes (v)\otimes [\tilde{w}_{ij} \cdot 1]) = (v) \otimes (\~w_j \cdot s_j\cdot (u)_i).
\]
According to our choice of orientations (made in the paragraph above Lemma \ref{lem:diff2}), in order to get the correct formula for the differential $\partial_2^{(j)}$, we need to compute the image of $(v) \otimes (\~w_j \cdot s_j\cdot (u)_i)$ under the automorphism of $\KMW_1\otimes \ZA[\tilde{w}_jT]$ given by the change of orientation from $\theta_j$ to ${}^t\theta_j$.  By Lemma \ref{lem:comptheta}, ${}^t\theta_j$ differs from $\theta_j$ by the character associated with the root $-\beta_{\lambda_j}$.  We will use the following observation in the determination of $\partial_2^{(j)}$.

\begin{lemma} 
\label{lem:action} 
For any $v \in F^{\times}$, $w \in W$ with a lift $\~w \in N_G(T)$ and $t\in T$, if $T$ acts on $\KMW_1$ on the left through the character $\beta:T\to \G_m$, then one has
\[
(t) \odot \big((v) \otimes [\tilde{w} \cdot 1]\big) = \big(\eta(\beta(t))(v)\big) \otimes [\tilde{w} \cdot 1] + \<\beta(t)\>(v) \otimes [\tilde{w} \cdot w^{-1}(t)]
\]
in $\KMW_1\otimes \ZA[\~{w}T]$, where $\odot$ denotes the $\ZA[T]$-action on $\KMW_1\otimes \ZA[\dot{w}T]$.
\end{lemma}

\begin{proof} One first reduces to the case $T = \G_m$ and $\beta :\G_m\to \G_m$ a group homomorphism. If $\Psi: \ZA[\G_m]\to \ZA[\G_m]\otimes \ZA[\G_m]$ is the 
diagonal, we have
\[
\Psi(t) = (t)_1 + (t)_2 + (t)_1(t)_2,
\]
where we identify $\ZA[\G_m]\otimes \ZA[\G_m]$ with $\Z \oplus \KMW_1 \oplus \KMW_1 \oplus (\KMW_1\otimes \KMW_1)$ and use the conventions of Section \ref{subsection Z[T]}. Clearly, $(t)_1(t)_2 = (-1)(t)$. Moreover, we have $(t)_i (t')_i = (\eta(t)(t'))_i$, for $i=1,2$. Thus, $(t)$ acts through $(v)\mapsto \eta(\beta(t))(v)$ on $\KMW_1$ and using the formula for $\Psi(t)$ and the relation $\<v\>  = \eta(v) + 1$, we get
\[
\begin{split}
(t) \odot \big( (v) \otimes [\tilde{w} \cdot 1 ]\big) & = \big(\eta(\beta(t))(v)\big)\otimes [\tilde{w} \cdot 1] + (v)\otimes (t)[\tilde{w} \cdot 1] + (\eta(\beta(t))(v))\otimes (t)[\tilde{w} \cdot 1] \\
& = \big(\eta(\beta(t))(v)\big)\otimes [\tilde{w} \cdot 1 ] + \<\beta(t)\> (v)\otimes (t)[\tilde{w} \cdot 1] \\
& = \big(\eta(\beta(t))(v)\big) \otimes [\tilde{w} \cdot 1] + \<\beta(t)\>(v) \otimes [\tilde{w} \cdot w^{-1}(t)]. 
\end{split}
\]
\end{proof}

Given any field extension $F$ of $k$ and $u \in F^{\times}$, we define
\begin{equation}
\label{eqn rho_ij}
\rho_{ij} := \beta_{\lambda_j}(s_j(\alpha_i^{\vee}(u)))^{-1} \in F^\times,
\end{equation}
for every $i,j \in \{1, \ldots, r\}$ with $i\neq j$.  Now, putting everything together, it follows from Lemmas \ref{lem:comptheta}, \ref{lem:autthom} and \ref{lem:action} that we have
\begin{equation} 
\label{eqn d_2^j formula}
\partial^{(j)}_2( (u)\otimes (v) \otimes [\tilde{w}_{ij} \cdot 1])  = \<\rho_{ij}\> (v) \otimes \tilde{w}_{j}\cdot s_j(u)_i + \eta (\rho_{ij})(v)\otimes  [\tilde{w}_{j} \cdot 1] 
\end{equation}
in $\KMW_1\otimes \ZA[\tilde{w}_jT]$.  To summarize, we have obtained the following from the above discussion and formulas \eqref{eqn d_2^i formula} and \eqref{eqn d_2^j formula}.

\begin{proposition} 
\label{prop: partial2} 
Let $F$ be a finitely generated field extension of $k$ and let $u,v \in F^{\times}$.  Let $i,j \in \{1,\ldots, r\}$ be such that $i \neq j$ and let the notations be as above.  Then there is a unit $\rho_{ij}$ in $F$ depending only upon $u$ (and given by \eqref{eqn rho_ij}) such that:
\[
\partial_2( (u)\otimes (v) \otimes [\tilde{w}_{ij} \cdot 1 ])  =  \<\rho_{ij}\> (v) \otimes \tilde{w}_{j}\cdot s_j(u)_i + \eta (\rho_{ij})(v)\otimes  [\tilde{w}_{j} \cdot 1] - (u) \otimes \tilde{w}_{i} \cdot (v)_j.
\]
\end{proposition}

We now determine the image of 
$$\partial_2: \Ctcell_2(G) = \underset{[i,j]}{\oplus} ~ \KMW_2 \otimes \ZA[\tilde{w}_{[i,j]}T] \to Z_1(G).$$
We will use the notation consistent with Lemma \ref{lem:Z1modT}.  For the sake of brevity, we will simply write $w_{ij}$ for $w_{[i,j]}$.  Let $F$ be a finitely generated field extension of $k$ and let $u,v \in F^{\times}$.  For each $i$, we denote by 
\[
\phi_i: \KMW_2 \otimes \ZA[\tilde{w}_{ij}T] \to \underset{i=1}{\overset{r}{\oplus}} ~ \KMW_1 \otimes \ZA[\~w_iT]
\]
the morphism defined by
\begin{equation}
\label{eqn formula phi_i}
\begin{split}
(u,v) \otimes [\tilde{w}_{ij} \cdot 1] \mapsto \phi_{i}(u,v): = & ~ (u) \otimes \tilde{w}_i\cdot (\alpha^\vee_i(v))  - \eta(u)(v)\otimes [\tilde{w}_i \cdot 1]  \\
= & ~ (u) \otimes \tilde{w}_i\cdot (v)_i - \eta(u)(v)\otimes [\tilde{w}_i \cdot 1] 
\end{split}
\end{equation}
and for each $(i,j)$ with $i \neq j$, we denote by 
\[
\delta_{ij}: \KMW_2 \otimes \ZA[\tilde{w}_{ij}T] \to \underset{i=1}{\overset{r}{\oplus}} ~ \KMW_1 \otimes \ZA[\~w_iT] 
\]
the morphism defined by
\begin{equation}
\label{eqn formula delta_ij}
\begin{split}
(u, v) \mapsto \delta_{ij}(u,v):= & ~ (u) \otimes \tilde{w}_i \cdot (\alpha^\vee_j(v)) - (v)\otimes \tilde{w}_j \cdot (\alpha^\vee_i(u)) \\  
= & ~ (u) \otimes \tilde{w}_i \cdot (v)_j - (v)\otimes \tilde{w}_j \cdot (u)_i.
\end{split}
\end{equation}
Note that the image of $\phi_i$ lies in the summand $\KMW_1 \otimes \ZA[\tilde{w}_iT]$ of $\Ctcell_1(G)$, whereas the image of $\delta_{ij}$ lies in $\left(\KMW_1 \otimes \ZA[\tilde{w}_iT]\right) \oplus \left(\KMW_1 \otimes \ZA[\tilde{w}_jT]\right)$.

For $i \neq j$, the roots $\alpha_i$ and $\alpha_j$ are related in the group of cocharacters of $T$ by 
\[
s_j(\alpha_i^\vee) = \alpha_i^\vee + n_{ji} \alpha^\vee_j, 
\]
where $n_{ji} = -\<\alpha_j,\alpha_i^\vee\>$ is the corresponding coefficient of the Cartan matrix of $G$ \cite[Chapter VI, \S 1, page 144]{Bourbaki}, which is known to belong to the set $\{0, 1, 2, 3\}$. 
One has $n_{ji}=0$ if and only if $\alpha_i$ and $\alpha_j$ are not adjacent in the Dynkin diagram of $G$.  It follows that in $\ZA[T]$, with our conventions, we have
\begin{equation}
\label{eqn sj(ui)}
s_j(u)_i = (\alpha^\vee_i(u) \alpha^\vee_j(u^{n_{ji}})) =(u)_i+ (u^{n_{ji}})_j + (u)_i(u^{n_{ji}})_j.
\end{equation}

\begin{theorem}
\label{thm: partial_2}
Let $G$ be a split, semisimple, almost simple, simply connected algebraic group of rank $r$ over $k$ and let $F$ be a finitely generated field extension of $k$.  Let $i,j \in \{1,\ldots, r\}$ be such that $i \neq j$.  For $u \in F^{\times}$, let $\rho_{ij}$ be defined by \eqref{eqn rho_ij}.
\begin{enumerate}[label=$(\alph*)$]
\item Suppose that $\alpha_i$ and $\alpha_j$ are not adjacent in the Dynkin diagram of $G$.  Then $\<\rho_{ij}\> = 1$ and for any $u, v\in F^{\times}$, we have
\[
\partial_2((u)\otimes (v) \otimes [\tilde{w}_{[i,j]} \cdot 1]) = \delta_{ji}(v,u).
\]
\item Suppose that $\alpha_i$ and $\alpha_j$ are adjacent in the Dynkin diagram of $G$ with $\lambda_i>\lambda_j$ and $n_{ji}$ even; that is, $n_{ji}=2$. Then $\<\rho_{ij}\> = 1$ and for any $u, v \in F^{\times}$, we have
\[
\partial_2((u)\otimes (v) \otimes [\tilde{w}_{ij} \cdot 1]) = \delta_{ji}(v,u) +  \phi_{j}(v,u^2)(1+(-1)_i).
\]

\item Suppose that $\alpha_i$ and $\alpha_j$ are adjacent in the Dynkin diagram of $G$ with $\lambda_i>\lambda_j$ and $n_{ji}$ odd. Then $\<\rho_{ij}\> = \<u\>$ and we have
\[
\partial_2((u)\otimes (v) \otimes [\tilde{w}_{ij} \cdot 1]) = \delta_{ji}(v,u) + \phi_j(\<u\>(v)(u^n))(1 + (u)_i),
\]
where $n = n_{ij}$.
\end{enumerate}
\end{theorem}

\begin{proof}
We first prove $(a)$.  If $\alpha_i$ and $\alpha_j$ are not adjacent in the Dynkin diagram of $G$, then $s_is_j = s_js_i$ and we have $s_j(\alpha_i^\vee(u)) = (\alpha_i^\vee(u))$.  Since $s_is_j = s_js_i$, we have $w_{[i,j]} = w_{[j,i]}$ and we can always assume $\lambda_i>\lambda_j$, after permuting the indices if required.  Since $\partial_1 \circ \partial_2((u)\otimes (v) \otimes [\tilde{w}_{ij} \cdot 1]) = 0$,
applying Proposition \ref{prop: partial2} and Lemma \ref{lem: partial1}, we obtain
\[
\begin{split}
0 &= \partial_1 \circ \partial_2((u)\otimes (v) \otimes [\tilde{w}_{ij} \cdot 1]) \\
&= \partial_1 \left(\<\rho_{ij}\> (v) \otimes \tilde{w}_{j}\cdot s_j(u)_i + \eta (\rho_{ij})(v)\otimes  [ \tilde{w}_{j} \cdot 1 ] - (u) \otimes \tilde{w}_{i} \cdot (v)_j \right) \\
&= \tilde{w}_0 \cdot \<\rho_{ij}\>(v)_j \cdot (u)_i
 + \tilde{w}_0 \cdot (\eta(\rho_{ij})(v))_j
 - \tilde{w}_0 \cdot (u)_i \cdot (v)_j
\end{split}
\]
in $\ZA[\tilde{w}_0T]$.  By the structure of the commutative algebra $\ZA[T]$ (see \eqref{eqn ZT decomposition}), it follows that $\eta(\rho_{ij})(v) = 0$ in $\KMW_1(F)$ and that $\<\rho_{ij}\>(v)_j \cdot (u)_i - (u)_i \cdot (v)_j =0$.  Since $\eta(\rho_{ij})(v) = 0$ for every $v \in F^{\times}$, it follows from a standard argument using restriction to the purely transcendental field extension $F(t)$ and applying the residue homomorphism at the place $t$ that $\eta(\rho_{ij}) = 0$.  Consequently, $\<\rho_{ij}\>=1$ in $\GW(F)$.  
Thus, we have
\[
\partial_2((u)\otimes (v) \otimes [\tilde{w}_{ij} \cdot 1]) =   (v) \otimes \tilde{w}_{j}\cdot (u)_i  - (u) \otimes \tilde{w}_{i} \cdot (v)_j = - \delta_{ij}(u, v) = \delta_{ji}(v,u),
\]
proving $(a)$.  We now prove $(b)$.  Since $n_{ji}=2$, we have 
\[
s_j(u)_i = (u)_i+ (u^{2})_j + (u)_i(u^{2})_j.
\]
As in the proof of part $(a)$, the relation $\partial_1 \circ \partial_2((u) \otimes (v) \otimes [\tilde{w}_{ij} \cdot 1]) = 0$ after applying applying Proposition \ref{prop: partial2} and Lemma \ref{lem: partial1} yields
\[
0 = \tilde{w}_0 \cdot \<\rho_{ij}\>(v)_j \cdot s_j(u)_i
 + \tilde{w}_0 \cdot (\eta(\rho_{ij})(v))_j
 - \tilde{w}_0 \cdot (u)_i \cdot (v)_j \in \ZA[\~w_0T].
\]
Thus, after cancelling $\tilde{w}_0$, we have the following equality in $\ZA[T]$:
\[
\begin{split}
0 & = (\<\rho_{ij}\>(v))_j \cdot s_j(u)_i + (\eta[\rho_{ij}](v))_j - (u)_i (v)_j\\
& =  (\<\rho_{ij}\>(v))_j(u)_i+ (\<\rho_{ij}\>(v))_j(u^{2})_j + (\<\rho_{ij}\>(v))_j(u)_i(u^{2})_j + (\eta(\rho_{ij})(v))_j - (u)_i (v)_j.
\end{split}
\]
As in the proof of $(a)$, it follows from the structure of $\ZA[T]$ that $\eta(\rho_{ij}) = 0$, whence $\<\rho_{ij}\>=1$ in $\GW(F)$.  Now, by Proposition \ref{prop: partial2}, we have
\[
\begin{split}
\partial_2((u)\otimes (v) \otimes [\tilde{w}_{ij} \cdot 1])
= & ~ (v) \otimes \tilde{w}_{j}\cdot s_j(u)_i  - (u) \otimes \tilde{w}_{i} \cdot (v)_j \\
= & ~ (v) \otimes \tilde{w}_{j}\cdot (u)_i + (v) \otimes \tilde{w}_{j}\cdot (u^{2})_j  + (v) \otimes \tilde{w}_{j}\cdot (u)_i(u^{2})_j\\ 
 & ~ - (u) \otimes \tilde{w}_{i} \cdot (v)_j \\
= & ~ \delta_{ji}(v,u) + (v) \otimes \tilde{w}_{j}\cdot (u^{2})_j + (v) \otimes \tilde{w}_{j}\cdot (u)_i(u^{2})_j.
\end{split}
\]
From the relations $(u)(u) = (-1)(u)$, $(u^2) = h(u)$ and $\eta h =0$ (where $h = 2_{\epsilon} = 1 + \<-1\>$) in Milnor-Witt $K$-theory, it follows that 
$$(u)_i(u^{2})_j = (u)_i(u)_jh = (-1)_i(u)_jh = (-1)_i(u^{2})_j$$
and that $\phi_j(v,u^2) = (v) \otimes \tilde{w}_{j}\cdot (u^{2})_j$.  Therefore, we obtain
\[
\partial_2((u)\otimes (v) \otimes [\tilde{w}_{ij} \cdot 1]) = \delta_{ji}(v,u) +  \phi_{j}(v,u^2)(1+(-1)_i),
\]
as desired.

In order to prove $(c)$, we proceed exactly as above to expand and simplify the relation $\partial_1 \circ \partial_2((u)\otimes (v) \otimes [\tilde{w}_{ij} \cdot 1]) = 0$ using Proposition \ref{prop: partial2}, Lemma \ref{lem: partial1} and \eqref{eqn sj(ui)} to obtain
\[
\begin{split}
0 & = (\<\rho_{ij}\>(v))_j \cdot s_j(u)_i + (\eta(\rho_{ij})(v))_j - (u)_i (v)_j\\
& =  (\<\rho_{ij}\>(v))_j(u)_i+ (\<\rho_{ij}\>(v))_j(u^{n})_j + (\<\rho_{ij}\>(v))_j(u)_i(u^{n})_j + (\eta(\rho_{ij})(v))_j - (u)_i (v)_j\\
& =  (\<\rho_{ij}\>(v))_j(u)_i+ (\eta\<\rho_{ij}\>(v)(u^{n}))_j + (u)_i(\eta\<\rho_{ij}\>(v)(u^{n}))_j + (\eta(\rho_{ij})(v))_j - (u)_i (v)_j
\end{split} 
\]
in $\ZA[T]$.  By the structure of $\ZA[T]$ \eqref{eqn ZT decomposition}, we obtain
\[
\eta\<\rho_{ij}\>(v)(u^{n}) + \eta(\rho_{ij})(v) = 0
\]
and
\[
(\<\rho_{ij}\>(v))_j(u)_i + (u)_i(\eta\<\rho_{ij}\>(v)(u^{n}))_j - (u)_i (v)_j = 0,
\]
for any $u,v \in F^{\times}$.  Since $n$ is odd, we have $\eta(u^{n}) = n_{\epsilon}\cdot\eta(u) = \eta(u)$.  Using the relation $(\rho_{ij}u) = \<\rho_{ij}\>(u)+(\rho_{ij})$ (see \cite[Lemma 3.5]{Morel-book}), we obtain
\[
0 = \eta\<\rho_{ij}\>(v)(u^{n}) + \eta(\rho_{ij})(v) = \eta (\rho_{ij}u) \cdot (v),
\]
for every $v \in F^{\times}$.  Again, using the restriction to $F(t)$ followed by the residue at the place $t$, it follows that 
\[
\eta (\rho_{ij}u) = \<\rho_{ij}u\> - 1 = 0 \in \GW(F). 
\]
Hence, $\<\rho_{ij}\> = \<u\>$ and also $\eta (\rho_{ij}) = \eta (u)$.  Since 
\[
0 = \eta (\rho_{ij}u) = \eta (\rho_{ij}) + \eta \<\rho_{ij}\>(u) = \eta (u) + \eta \<u\>(u), 
\]
we also have $-\eta \<u\>(u) = \eta (u)$.  Therefore, 
\[
\begin{split}
\partial_2((u)\otimes (v) \otimes [\tilde{w}_{ij} \cdot 1]) 
= & ~ \<\rho_{ij}\>(v) \otimes \tilde{w}_{j}\cdot s_j(u)_i + \eta (\rho_{ij})(v)\otimes  [ \tilde{w}_{j} \cdot 1 ] - (u) \otimes \tilde{w}_{i} \cdot (v)_j \\
= & ~ \<u\>(v) \otimes \tilde{w}_{j}\cdot (u)_i + \<u\>(v) \otimes \tilde{w}_{j}\cdot (u^{n})_j + \<u\>(v) \otimes \tilde{w}_{j}\cdot (u)_i(u^{n})_j \\
& ~ + \eta (u)(v)\otimes  [ \tilde{w}_{j} \cdot 1 ] - (u) \otimes \tilde{w}_{i} \cdot (v)_j \\
= & ~ \delta_{ji}(v,u) + \eta(u)(v) \otimes \tilde{w}_{j}\cdot (u)_i + \<u\>(v) \otimes \tilde{w}_{j}\cdot (u^{n})_j \\
& ~ + \<u\>(v) \otimes \tilde{w}_{j}\cdot (u)_i(u^{n})_j + \eta (u)(v)\otimes  [ \tilde{w}_{j} \cdot 1 ] \\
= & ~ \delta_{ji}(v,u) + \left( \<u\>(v) \otimes \tilde{w}_{j}\cdot (u^{n})_j + \eta (u)(v)\otimes  [ \tilde{w}_{j} \cdot 1 ] \right)\left( 1+ (u)_i \right)\\
= & ~ \delta_{ji}(v,u) + \phi_j(\<u\>(v)(u^n))(1 + (u)_i).
\end{split}
\]
\end{proof}

In particular, considering these computations modulo $T$ (that is, after tensoring with $\Z$ over $\ZA[T]$) and using Lemma \ref{lem:Z1modT} we can conclude:

\begin{theorem}
\label{thm:diff2} 
Let $G$ and $F$ be as in Theorem \ref{thm: partial_2}.  Let $i$ and $j$ be distinct indices in $\{1,\dots,r\}$ and assume $\lambda_i>\lambda_j$ for our chosen reduced expression of $w_0$.  For any $u, v \in F^{\times}$, we denote by $\overline{\partial}_{ij} \big( (u)(v) \big)$ the class of
\[
\partial_2((u)\otimes (v)\otimes [\tilde{w}_{ij}\cdot 1])\in Z_1(G)
\]
in $Z_1(G)\otimes_{\ZA[T]} \Z$. Then for any $u, v \in F^{\times}$, we have (with obvious notations):
\begin{enumerate}[label=$(\alph*)$]
\item Suppose that $\alpha_i$ and $\alpha_j$ are not adjacent in the Dynkin diagram of $G$.  Then
\[
\overline{\partial}_{ij} \big( (u)(v) \big) = \overline{\delta}_{ji}(v,u)\in (\KM_2)_{ji} \subset Z_1(G)\otimes_{\ZA[T]} \Z.
\]
\item Suppose that $\alpha_i$ and $\alpha_j$ are adjacent in the Dynkin diagram of $G$ with $n_{ji}$ even; that is, $n_{ji}=2$. Then
\[
\overline{\partial}_{ij} \big( (u)(v) \big) = \overline{\delta}_{ji}(v,u) + h\cdot\overline{\phi}_j(v,u) \in(\KM_2)_{ji} \oplus (\KMW_2)_j \subset Z_1(G)\otimes_{\ZA[T]} \Z.
\]
\item Suppose that $\alpha_i$ and $\alpha_j$ are adjacent in the Dynkin diagram of $G$ with $n_{ji}$ odd.  Then
\[
\overline{\partial}_{ij} \big( (u)(v) \big) = \overline{\delta}_{ji}(v,u) + (n_{ji})_\epsilon \cdot \overline{\phi}_j(v,u) \in (\KM_2)_{ji} \oplus (\KMW_2)_j \subset Z_1(G)\otimes_{\ZA[T]} \Z.
\]
\end{enumerate}
\end{theorem}

\subsection{Image of the degree \texorpdfstring{$2$}{2} differential: the case \texorpdfstring{$\lambda_i<\lambda_j$}{lambda i < lambda j}} \hfill
\label{subsection i<j}


Let $[i,j]$ be a pair of indices such that $\lambda_i<\lambda_j$ in the sense of Notation \ref{notation horizontal codim 1}.  Recall that $\lambda_i$ and $\lambda_j$ depend upon the fixed reduced expression \eqref{eq:redw0} of $w_0$.  We will show that there always exists a reduced expression 
\[
w_0 =  \sigma'_\ell\cdots\sigma'_1   
\]
as a product of simple reflections, possibly different from \ref{eq:redw0}, such that $\lambda'_i>\lambda'_j$, with obvious notation: $\lambda'_i$ denotes the unique index such that    
\[
\sigma'_\ell\cdots\widehat{\sigma'}_{\lambda'_i}\dots\sigma'_1
\]
is a reduced expression of $w_i = w_0s_i$, for every $i$.  

\begin{notation}
\label{notation breve}
We write $\breve{w}_0 = \dot{\sigma'}_\ell\cdots\dot{\sigma'}_1 \in N_G(T)$, and $\breve{w}_i$ for the lift of $w_i$ obtained by removing $\dot{\sigma'}_{\lambda'_i}$ from it.  These choices define another orientation $\breve{\theta}_i$ of $\nu_i$ and another identification
\[
\ZA(Y_i)) \cong \Z[\breve{w}_iT].
\]
\end{notation}

We clearly have an isomorphism of sheaves 
\begin{equation}
\label{eqn Xi_i}
\Upsilon_i: \KMW_1\otimes \Z[\tilde{w}_iT] \xrightarrow{\simeq} \KMW_1\otimes \Z[\breve{w}_iT]. 
\end{equation}
We denote for every $i\in \{ 0,1, \ldots, r\}$
$$\varsigma_i: = (\breve{w}_i)^{-1}\tilde{w}_i,$$
which is an element of $T$.  The element $\varsigma_0$ is related to the elements $\varsigma_i$ for $1 \leq i \leq r$ by the following relation. 

\begin{lemma} 
\label{lem: varsigma}
For any unit $u$ (in a finitely generated field extension of $k$), we have 
\[
[\varsigma_0](u)_i = [\varsigma_i](\<\tau_i\>(u))_i \in \ZA[T].
\]
\end{lemma}

\begin{proof} 
The above equality is equivalent to showing that
\[
([\varsigma_0] - [\varsigma_i][\tau_i]_i) \cdot (u)_i = 0.
\]
This is obtained by applying the formula of the first differential $\partial_1^{(i)}$ from Lemma \ref{lem: partial1}, the comparison isomorphism \eqref{eqn Xi_i} and its analogue in degree $0$.
\end{proof}

\begin{proposition} Let $\tilde{\theta}_i$ denote the orientation of the normal bundle $\nu_i$ of $Y_i$ in $G$ determined by the choice of the reduced expressions in Notation \ref{notation horizontal codim 1}.
\begin{enumerate}[label=$(\alph*)$]
\item Let $\tau_i\in k^\times$ denote the unit corresponding to the comparison of the orientations $\~\theta_i$ and $\breve{\theta}_i$ of $\nu_i$; in other words, the automorphism 
\[
\breve{\theta}^{-1}_i \circ \tilde{\theta}_i : \sO^1_{Y_i}\to \sO^1_{Y_i}.
\]
The composition
\[
k \xrightarrow{{u}_{-\beta_{\lambda_i}}} \mathfrak{u}_{-\beta_{\lambda_i}} \xrightarrow{Ad_{\tilde{w}''_i}} \mathfrak{u}_{-\alpha'_{i}} \xrightarrow{Ad_{(\breve{w}''_i)}^{-1}}  \mathfrak{u}_{-\beta_{\lambda'_i}} \xrightarrow{{u}^{-1}_{-\beta_{\lambda'_i}}} k
\]
through the chosen weak pinning of $G$ corresponds to multiplication by an element of $k^{\times}$, which agrees with $\tau_i$ up to square classes.

\item The isomorphism $\Upsilon_i$ of \eqref{eqn Xi_i} satisfies the following formula, for every units $u$ and $v$ in a finitely generated field extension of $k$ (with obvious notations):
\[
\Upsilon_i(\tilde{\phi}_i(u,v)) = \breve{\phi}_i(\<\tau_i\>(u)(v))\cdot[\varsigma_i] \in \KMW_1\otimes \ZA[\breve{w}_iT].
\]
\end{enumerate}
\end{proposition}

\begin{proof}
Part $(a)$ follows directly from Lemma \ref{lem:autthom}, Lemma \ref{lem: varsigma} and the obvious variants of Proposition \ref{prop: normal bundle} and Lemma \ref{lem:comptub}.  Part $(b)$ follows immediately from the fact that the isomorphism $\KMW_1\otimes \Z[\tilde{w}_iT] \stackrel{\simeq}{\to} \KMW_1\otimes \Z[\breve{w}_iT]$ is induced by
\[
(u)\otimes [\tilde{w}_i \cdot t] \mapsto \<\tau_i\>(u) \otimes [\breve{w}_i \cdot (\varsigma_i t)].
\]
\end{proof}

\begin{corollary} 
For any $j\not=i$, we have 
\[
\varpi_j(\varsigma_0) = \varpi_j(\varsigma_i)
\]
and for $j=i$, we have
\[
\varpi_i(\varsigma_0) \equiv \varpi_i(\varsigma_i)  ~\mathrm{ mod }~k^{\times 2}.
\]
In other words, there exists $\tau_i' \in k^{\times}$ such that 
\[
\varsigma_0 = \varsigma_i \cdot \alpha^\vee_i(\tau'_i) \in T
\]
with $\<\tau_i\> = \<\tau'_i\>$.
\end{corollary}

\begin{proof} 
One expends the equation $([\varsigma_0] - [\varsigma_i][\tau_i]_i) \cdot (u)_i = 0$ in $\ZA[T]$ using the decomposition \eqref{eq:ZA[T]}. For $t = (t_1, \ldots, t_r)\in T$, with $t_i$ in the image of the coroot $\alpha^\vee_i$, the component of $[t](u)_i$ in the summand $\KMW_1$ corresponding to the index $j$ in (\ref{eq:ZA[T]}) is, $0$ for each $j\not=i$, and is $(u) + \eta(t_i)(u) = \<t_i\>(u)$ for $j=i$.  The component of $[t](u)_i$ in the summand $\KMW_2$ corresponding to the index $(i,j)$ in (\ref{eq:ZA[T]}) for $i\not=j$ is $(t_j)(u)$ (up to permutation, depending on whether $i<j$ or $j<i$). Thus, for a given $i$, the equation $([\varsigma_0] - [\varsigma_i][\tau_i]_i)\cdot (u)_i = 0$ for every $u \in F^{\times}$ implies that for each $j\not=i$, we have
\[
\varpi_j(\varsigma_0)(u) = \varpi_j(\varsigma_i)(u) \in \KMW_2,
\]
for every $u \in F^{\times}$.  By restricting to a purely transcendental extension of $F$ and taking a suitable residue, it follows that $(\varpi_j(\varsigma_0)) = (\varpi_j(\varsigma_i))$ in $\KMW_1$ and hence, $\varpi_j(\varsigma_0) = \varpi_j(\varsigma_i)$ in $\G_m$.  On the other hand, for $j=i$, we have
\[
\<\varpi_j(\varsigma_0)\>(u) = \<\varpi_j(\varsigma_i)\>\<\tau_i\>(u) \in \KMW_1,
\]
for every $u \in F^{\times}$ and consequently, in the same way as above, we have $\<\varpi_j(\varsigma_0)\>(u) = \<\varpi_j(\varsigma_i)\tau_i\>$ in $\GW(k)$. This proves the claim.
\end{proof}

\begin{remark} 
\label{rem tau_i}
As one sees above, $\tau_i$ is only well-defined up to squares (since it depends upon the weak pinning). So in fact we may always choose $\tau_i \in k^{\times}$ so that we have $\varsigma_0 = \varsigma_i \cdot \alpha^\vee_i(\tau_i)$ as elements of $T$. Thus, we have 
\[
[\varsigma_0] = [\varsigma_i] \cdot [\tau_i]_i
\]
in $\ZA[T]$.  We will fix such a choice of $\tau_i$, for every $i \in \{1, \ldots, r\}$, in what follows.
\end{remark}

Using these observations, we may rewrite the above change of reduced expression isomorphism \eqref{eqn Xi_i} as well as the morphisms $\phi_i$ and $\delta_{ij}$ (see \eqref{eqn formula phi_i} and \eqref{eqn formula delta_ij}) as follows.

\begin{proposition}
\label{prop chage of expression}
With the above notation, we have the following, for all units $u, v$ in a finitely generated field extension of $k$.
\begin{enumerate}[label=$(\alph*)$]
\item For every $i \in \{1, \ldots, r\}$, we have
\[
\Upsilon_i ((u)\otimes [\tilde{w}_i \cdot 1]) =  (u) \otimes [\breve{w}_i \cdot \varsigma_0] -  \breve{\phi}_i(u,\tau_i) \cdot [\varsigma_i].
\]
\item For every $i \in \{1, \ldots, r\}$, we have
\[
\Upsilon_i(\tilde{\phi}_i(u,v)) = \breve{\phi}_i (u,v) \cdot [\varsigma_0]  - \breve{\phi}_i(u,\tau_i)\cdot[\varsigma_i](v)_i + \breve{\phi}_i(\eta(u)(v)(\tau_i))\cdot [\varsigma_i]
\]
in $\KMW_1\otimes \ZA[\breve{w}_iT]$.
\item For $i\not= j$, we have:
\[
\Upsilon_i(\tilde{\delta}_{ij}(u,v)) = \breve{\delta}_{ij}(u,v)\cdot[\varsigma_0] - \breve{\phi}_i(u,\tau_i)\cdot[\varsigma_i](v)_j + \breve{\phi}_j(v,\tau_i)\cdot[\varsigma_j](u)_i.
\]
\end{enumerate}
\end{proposition}

\begin{proof} We have $\Upsilon_i ((u)\otimes [\tilde{w}_i\cdot 1]) = \<\tau_i\>(u) \otimes [\breve{w}_i \cdot \varsigma_i] $. Since $\<\tau_i\> = 1 + \eta(\tau_i)$, we get
\[
\begin{split}
\Upsilon_i ((u)\otimes [\tilde{w}_i \cdot 1]) = & ~ \<\tau_i\>(u) \otimes [\breve{w}_i \cdot \varsigma_i]\\
= & ~ (u) \otimes [\breve{w}_i \cdot \varsigma_i] + (\eta(u)(\tau_i))\otimes [\breve{w}_i \cdot \varsigma_i] \\
= & ~ (u) \otimes [\breve{w}_i \cdot \varsigma_i] + (u) \otimes [\breve{w}_i \cdot \varsigma_i](\tau_i)_i - (u) \otimes [\breve{w}_i \cdot \varsigma_i](\tau_i)_i \\
& ~ + (\eta(u)(\tau_i))\otimes [\breve{w}_i \cdot \varsigma_i]\\
= & ~ (u) \otimes [\breve{w}_i \cdot \varsigma_i] (1 + (\tau_i)_i) - \breve{\phi}_i(u,\tau_i) \cdot  [\varsigma_i]\\
= & ~ (u) \otimes [\breve{w}_i \cdot \varsigma_i] [\tau_i]_i - \breve{\phi}_i( (u,\tau_i) \cdot  [\varsigma_i]\\
= & ~ (u) \otimes [\breve{w}_i  \cdot \varsigma_0]- \breve{\phi}_i(u,\tau_i) \cdot  [\varsigma_i],
\end{split}
\]
by Remark \ref{rem tau_i}.  This proves $(a)$.  Parts $(b)$ and $(c)$ follow similarly by using the relation $\<\tau_i\> \cdot [\varsigma_i] = [\varsigma_0]$, for each $i$ and the formulas for $\breve{\phi}_i$ and $\breve{\delta}_{ij}$ analogous to \eqref{eqn formula phi_i} and \eqref{eqn formula delta_ij}.
\end{proof}

\begin{corollary}
\label{cor:changeexpression} 
The automorphism
\[
\overline{\Upsilon}:\left(\underset{i}{\oplus}~ \KMW_2 \right) \oplus \left( \underset{(i,j); ~i<j}{\oplus} \KM_2 \right) \to \left(\underset{i}{\oplus}~ \KMW_2 \right) \oplus \left( \underset{(i,j); ~i<j}{\oplus} \KM_2 \right)
\]
induced by the change of reduced expression isomorphism on $\Ccell_1(G)$: 
\[
\underset{i}{\oplus}~ \Upsilon_i: \underset{i}{\oplus}~ \KMW_1\otimes \ZA[\tilde{w}_iT] \xrightarrow{\simeq} \underset{i}{\oplus}~ \KMW_1\otimes \ZA[\breve{w}_iT]
\]
through the identification $\left(\underset{i}{\oplus}~ \KMW_2 \right) \oplus \left( \underset{(i,j); ~i<j}{\oplus} \KM_2 \right) \cong Z_1(G)\otimes_{\ZA[T]} \Z$ given by Lemma \ref{lem:Z1modT} satisfies the following, for any finitely generated field extension $F$ of $k$ and $u, v \in F^{\times}$.
\begin{enumerate}[label=$(\alph*)$]
\item For every $i$, we have
\[
\overline{\Upsilon}(\tilde{\phi}_i(u,v)) = \breve{\phi}_i(\<\tau_i\>(u)(v)) \in \KMW_2.
\]
\item For every $i \neq j$, we have
\[
\overline{\Upsilon}(\tilde{\delta}_{ij}(u,v) = \breve{\delta}_{ij}(u,v) \in \KM_2. 
\]
\end{enumerate}
\end{corollary}
\begin{proof}
This is a direct consequence of the formulas in Proposition \ref{prop chage of expression}, using the observation that after going modulo $T$ (that is, applying $-\otimes_{\ZA[T]}\Z$), one has $$\breve{\phi}_i(u,\tau_i)[\varsigma_i](v)_j = 0$$ and $$\breve{\phi}_j(v,\tau_j)[\varsigma_j](u)_i = 0,$$ since $(u)_i = [u]_i - [1]$ and $(v)_j = [v]_j - [1]$ are both in the augmentation ideal; that is, the kernel of the morphism $\ZA[T] \to \Z$. 
\end{proof}

We are now set to reduce to the case $\lambda_i>\lambda_j$ already treated above using the following lemma and the computations on change of reduced expression for $w_0$ carried out above.

\begin{lemma} 
Let $[i,j]$ be a pair of indices such that $\lambda_i<\lambda_j$ in the sense of Notation \ref{notation horizontal codim 1} with respect to the reduced expression $w_0 = \sigma_\ell\cdots\sigma_1$ given by \eqref{eq:redw0}.  There exists another reduced expression $\sigma'_\ell\cdots\sigma'_1$ of $w_0$ for which, with obvious notations, one has $\lambda'_i>\lambda'_j$.
\end{lemma}

\begin{proof} 
Let $W_{\{i,j\}}\subset W$ be the subgroup generated by $s_i$ and $s_j$.  Note that $W_{\{i,j\}}$ is also a Coxeter group, being the Weyl group of the subgroup $G_{ij}$ of $G$ of semi-simple rank $2$ generated by $T$, $S_{\alpha_i}$ and $S_{\alpha_i}$.  Let $v_0\in W_{\{i,j\}}$ be the longest element of $W_{\{i,j\}}$.  Then $u_0 = w_0v_0$ has the property that for any $v\in W_{\{i,j\}}$, $\ell(u_0v) = \ell(u_0) + \ell(v)$.  This follows easily from the properties of the longest word in Coxeter groups (see \cite[Proposition 2.20]{Abramenko-Brown}).

Now, multiplication by $u_0$ on the left induces a map $W_{\{i,j\}}\to W$ preserving length and takes $v_0$ to $w_0$. Given a reduced expression of $u_0$ and a reduced expression of $v_0$ in $W_{\{i,j\}}$ thus gives a reduced expression of $w_0$ in $W$.  In this way, we reduce the statement of the lemma to the rank $2$ case, as there are only two possible reduced expressions for $v_0$ in the rank $2$ case.  One easily verifies that one of these has the required property.
\end{proof}


We thus obtain the following analogue of Theorem \ref{thm:diff2} in the case $\lambda_i<\lambda_j$, using Corollary \ref{cor:changeexpression}.

\begin{theorem}
\label{thm:diff2'} 
Let $G$ and $F$ be as in Theorem \ref{thm: partial_2}.  Let $i$ and $j$ be distinct indices in $\{1,\dots,r\}$ and assume $\lambda_i<\lambda_j$ for our chosen reduced expression \eqref{eq:redw0} of $w_0$.  For any $u, v \in F^{\times}$, we denote by $\overline{\partial}_{ij} \big( (u)(v) \big)$ the class of
\[
\partial_2((u)\otimes (v)\otimes [\tilde{w}_{ij}\cdot 1])\in Z_1(G)
\]
in $Z_1(G)\otimes_{\ZA[T]} \Z$. Then for any $u, v \in F^{\times}$, we have (with obvious notations):
\begin{enumerate}[label=$(\alph*)$]
\item If $\alpha_i$ and $\alpha_j$ are not adjacent in the Dynkin diagram of $G$, then
\[
\overline{\partial}_{ij} \big( (u)(v) \big) = \overline{\delta}_{ji}(v,u)\in(\KM_2)_{ji} \subset Z_1(G)\otimes_{\ZA[T]} \Z.
\]
\item If $\alpha_i$ and $\alpha_j$ are adjacent in the Dynkin diagram of $G$ with $n_{ji}=2$, then
\[
\overline{\partial}_{ij} \big( (u)(v) \big) = \overline{\delta}_{ji}(v,u) + h \cdot \overline{\phi}_j(v,u) \in (\KM_2)_{ji} \oplus (\KMW_2)_j \subset Z_1(G)\otimes_{\ZA[T]} \Z.
\]
\item If $\alpha_i$ and $\alpha_j$ are adjacent in the Dynkin diagram of $G$ with $n_{ji}$ odd, then there exists $\tau_j\in k^\times$ such that:
\[
\overline{\partial}_{ij} \big( (u)(v) \big) = \overline{\delta}_{ji}(v,u) + (n_{ji})_\epsilon \<\tau_j\> \cdot \overline{\phi}_j(v,u) \in (\KM_2)_{ji} \oplus (\KMW_2)_j \subset Z_1(G)\otimes_{\ZA[T]} \Z.
\]
\end{enumerate}
\end{theorem}

\begin{proof} 
If $\alpha_i$ and $\alpha_j$ are not adjacent in the Dynkin diagram of $G$, then $s_is_j = s_js_i$.  But then we have $w_{ij} = w_{ji}$ and we can always assume $\lambda_i>\lambda_j$ to obtain $(a)$ from Theorem \ref{thm:diff2}.  For $(b)$, we conclude from Theorem \ref{thm:diff2} and Corollary \ref{cor:changeexpression} that there exists $\tau_j\in k^\times$ such that
\[
\overline{\partial}_{ij} \big( (u)(v) \big) = \overline{\delta}_{ji}(v,u) + h \cdot \<\tau_j\> \cdot \overline{\phi}_j(v,u) \in (\KM_2)_{ji} \oplus (\KMW_2)_j \subset Z_1(G)\otimes_{\ZA[T]} \Z.
\]
But by \cite[Lemma 3.9 (ii)]{Morel-book}, we have $h \cdot \<\tau_j\> = h$ and $(b)$ is proved. Part $(c)$ is proved in exactly the same way and is left to the reader. 
\end{proof}

\section{The main theorem, applications and miscellaneous remarks}
\label{section applications}

\subsection{Proof of the main theorem} \hfill
\label{subsection proof} 

Let $S$ be a split, almost simple, simply connected semi-simple $k$-group of rank $1$. Such a group is (non-canonically) isomorphic to $\SL_2$  and thus, $\piA_1(S)$ is (non-canonically) isomorphic to $\KMW_2$.  The action of the adjoint group $\PGL_2$ on $\SL_2$ induces a nontrivial action on $\piA_1(\SL_2) = \KMW_2$.  However, the action of the adjoint group $\PGL_2$ on the quotient 
\[
\KM_2 = \KMW_2/\eta
\]
is trivial, and in fact, the induced epimorphism $\piA_1(S) \twoheadrightarrow \KM_2$ is canonical.

We are now ready to prove the main Theorem.  We will freely make use of basic results on root systems from \cite[Chapter VI]{Bourbaki}.

\begin{theorem} 
\label{thm: main} 
Let $G$ be a split, semisimple, almost simple, simply connected $k$-group, with a fixed maximal torus $T\subset G$ and a fixed Borel subgroup $B\subset G$ containing $T$. We denote by $\Phi$ the corresponding root system, by $\Phi^\vee$ its dual and by $\Delta\subset \Phi$ the set of simple roots corresponding to $B$.  Let $\alpha\in \Delta$ be a long root (so that $\alpha^\vee$ is a small coroot) and let $S_\alpha\subset G$ be the split, semisimple, almost simple, simply connected $k$-subgroup of rank $1$ generated by $U_\alpha$ and $U_{-\alpha}$.
\begin{enumerate}[label=$(\alph*)$]
\item If $G$ is not of symplectic type, then the inclusion $S_\alpha\subset G$ induces a canonical isomorphism
\[
\KM_2 \xrightarrow{\simeq} \piA_1(G)
\]
induced by the morphism $\piA_1(S_\alpha) \to \piA_1(G)$ through the canonical epimorphism $\piA_1(S_\alpha) \to \KM_2$.

\item If $G$ is of symplectic type, then the inclusion $S_\alpha\subset G$ induces a non-canonical isomorphism
\[
\KMW_2 \xrightarrow{\simeq} \piA_1(G).
\]
More precisely, an isomorphism $\KMW_2 \cong \piA_1(G)$ corresponds to the choice of a weak pinning of $S_\alpha$, or equivalently, an isomorphism $\SL_2\cong S_\alpha$ up to the adjoint action of $\PGL_2$. 
\end{enumerate}
\end{theorem}
\begin{proof} 
We choose a labelling $(\alpha_1,\dots,\alpha_r)$ of $\Delta$ as in \cite[\S 4]{Bourbaki}.  Let $M$ be a strictly $\A^1$-invariant sheaf over $Sm_k$.   We prove Theorem \ref{thm: main} by explicitly computing the kernel of the morphism
\begin{equation}
\label{eq:diff}
\Hom_{Ab_{\A^1}(k)} (Z_1(G)\otimes_{\ZA[T]} \Z,M) \to \Hom_{Ab_{\A^1}(k)} (\Ccell_2(G)\otimes_{\ZA[T]} \Z,M)
\end{equation}
induced by the morphism $\Ccell_2(G)\otimes_{\ZA[T]} \Z \to Z_1(G)\otimes_{\ZA[T]} \Z$ of the exact sequence \eqref{eq: H1cokernel}.  We will denote this kernel by $\sM$.   Clearly, we have
\[
\sM =  \Hom_{Ab_{\A^1}(k)} (\Coker(\partial_2 :\Ccell_2(G)\to Z_1(G)),M) = \Hom_{Ab_{\A^1}(k)} (\HA_1(G),M); 
\]
from \eqref{eq: H1cokernel}.  We will show that 
\[
\sM = \begin{cases} {}_\eta M_{-2}(k) = \Hom_{Ab_{\A^1}(k)}(\KM_2, M), &\text{ if $G$ is not of symplectic type;} \\ ~M_{-2}(k) = \Hom_{Ab_{\A^1}(k)}(\KMW_2, M), &\text{ if $G$ is of symplectic type.} \end{cases} 
\]
for every $M \in Ab_{\A^1}(k)$; the theorem then follows from the Yoneda lemma and the fact that $\HA_1(G) = \piA_1(G)$ as already mentioned in the introduction.

We now choose a weak pinning of $G$, and a reduced expression of the longest word $w_0$ of $W$.  We can then orient the cellular $\A^1$-chain complex of $G$ as explained in Section \ref{section differential} and describe the morphism \eqref{eq:diff} using Theorem \ref{thm:diff2} and Theorem \ref{thm:diff2'}.  We first deduce from Lemma \ref{lem:Z1modT} that
\[
\Hom_{Ab_{\A^1}(k)} (Z_1(G)\otimes_{\ZA[T]} \Z,M) = M_{-2}(k)^r \times \left({}_\eta M_{-2}(k)^{\frac{r(r-1)}{2}}\right),
\]
since $\Hom_{Ab_{\A^1}(k)}(\KMW_2,M) = M_{-2}(k)$ for every $M \in Ab_{\A^1}(k)$ by \cite[Lemma 2.34, Theorem 3.37]{Morel-book}.  Let $N$ denote the number of codimension $2$ Bruhat cells in $G$; then with Notation \ref{notation codim 2 indices}, we have 
\[
\Hom_{Ab_{\A^1}(k)} (\Ccell_2(G)\otimes_{\ZA[T]} \Z, M) = M_{-2}(k)^{N}.
\]

Theorems \ref{thm:diff2} and \ref{thm:diff2'} along with the fact that $\delta_{ji}(v,u) = - \delta_{ij}(u,v)$ imply that there exist units $\tau_i\in k^\times$ such that the morphism \eqref{eq:diff} above has the form:
\[
M_{-2}(k)^r\times \left({}_\eta M_{-2}(k)\right)^{\frac{r(r-1)}{2}} \to M_{-2}(k)^{N}; \quad \left((x_i)_i;(y_{ij})_{i<j}\right) \mapsto \left(z_{[i,j]}\right)_{ij}
\]
such that for given $(i,j)$, $i\not= j$ the following holds:
\begin{enumerate}
\item if $\alpha_i$ and $\alpha_j$ are not adjacent in the Dynkin diagram of $G$, then
\[
z_{[i,j]} = - y_{ij};
\]
\item if $\alpha_i$ and $\alpha_j$ are adjacent and $n_{ji}=-\<\alpha_j, \alpha_i^\vee\> = 2$, then 
\[
z_{[i,j]} = - y_{ij} + h \cdot x_j;
\]
\item if $\alpha_i$ and $\alpha_j$ are adjacent and $n = n_{ji} = - \<\alpha_j, \alpha_i^\vee\> \in \{1,3\}$, then
\[
z_{[i,j]} = y_{ij} + \<\tau_j\> \cdot n_\epsilon \cdot x_j.
\]
\end{enumerate}
We now describe the kernel of \eqref{eq:diff}; let $\left((x_i)_i;(y_{ij})_{i<j}\right) \in M_{-2}(k)^r\times \left({}_\eta M_{-2}(k)\right)^{\frac{r(r-1)}{2}}$ be such that $\left(z_{[i,j]}\right)_{ij}=0$.  If $\alpha_i$ and $\alpha_j$ are not adjacent, then $z_{[i,j]} = 0$ is equivalent to $y_{ij} = 0$ by (1) above.

Let $\alpha_i$ and $\alpha_j$ be adjacent roots with $i<j$ and $n_{ji}=1$.  Then $z_{[i,j]} = 0 = z_{[j,i]}$ is equivalent to $y_{ij} = \<-\tau_j\>x_j =  \<-\tau_i\>x_i$ by (3).  Equivalently, we have $x_i = \<-\tau_i\>y_{ij}$ and $x_j = \<-\tau_j\>y_{ij}$.  However, since $y_{ij}\in {}_\eta M_{-2}(k)$, on which $\GW(k)$ acts trivially, we see that
\[
x_i = y_{ij} = x_j.
\]
This implies that $x_i$ is $\eta$-torsion for every $i$.  Furthermore, if $G$ is such that $n_{ij}=1$ for every adjacent pair of roots $\alpha_i$ and $\alpha_j$ with $i<j$, then the morphism of abelian groups
\[
\sM \to \ _{\eta}M_{-2}(k); \quad \left( (x_i)_i;(y_{(ij)})_{i < j}) \right) \mapsto x_1
\]
is an isomorphism (note that we may take $x_i$ for any fixed $i$ in the above isomorphism).  Note that this happens precisely when all the simple roots have the same length; that is, the cases when $G$ is of the type $A$, $D$ or $E$ (in other words, $G$ is simply laced).

Now, assume that $G$ is not simply laced.  Then there exists a unique index $i$ such that $n_{i+1,i} \in \{2,3\}$.  Assume that $n_{i+1,i}=3$; in this case, $i=1$ and $G$ is of the type $G_2$ with $\alpha_2$ being the long root.  In this case, we have $n_{12}=1$ and $y_{12} = -\<\tau_{2}\> \cdot 3_\epsilon \cdot x_{2} =  \<-\tau_1\> \cdot x_1$.  Since the $\GW(k)$-action on $y_{12}$ is trivial, we have
\[
y_{12} = - 3_\epsilon \cdot x_{2} = x_1.
\]
Thus, the morphism of abelian groups
\[
\sM \to \ _{\eta}M_{-2}(k); \quad \left( (x_i)_i;(y_{(ij)})_{i < j}) \right) \mapsto x_2
\]
is an isomorphism.  

Finally, we consider the case where $n_{i+1,i}=2$; that is, $\alpha_{i+1}$ is a long root.  We then have $y_{i,i+1} = -\<\tau_{i+1}\> \cdot h \cdot x_{i+1} =  \<-\tau_i\> \cdot x_i$, by (2).  Since the $\GW(k)$-action on the $y_{ij}$'s is trivial, we obtain
\[
y_{i,i+1} = - h \cdot x_{i+1} = x_i.
\]
Thus, $x_{i+1}$ determines $\left( (x_i)_i;(y_{(ij)})_{i < j})\right)$.  Suppose that $G$ is not of symplectic type.  Then $\alpha_{i+1}$ is connected to another long root and it follows that $x_{i+1} \in {}_\eta M_{-2}(k)$ by the above computations.  Hence, the morphism of abelian groups
\[
\sM \to {}_\eta M_{-2}(k); \quad \left( (x_i)_i;(y_{(ij)})_{i < j}) \right) \mapsto x_{i+1}
\]
is an isomorphism.

It only remains to consider the case where $G$ is of symplectic type.  In this case, $i = r$ and it follows proceeding as above that 
\[
y_{r-1,r} = - h \cdot x_r = x_{r-1}
\]
and the morphism of abelian groups
\[
\sM \to M_{-2}(k); \quad \left( (x_i)_i;(y_{(ij)})_{i < j}) \right) \mapsto x_r
\]
is an isomorphism.  This completes the proof of the theorem.
\end{proof}

\begin{remark} 
\label{rem main precise}
The proof of Theorem \ref{thm: main} above gives a much more precise statement of part $(a)$.  For instance, if $\beta$ is another long root, then the inclusion $S_\beta\subset G$ induces an isomorphism
\[
 \KM_2 \cong \pi_1^{\A^1}(G),
\] 
which is $(-1)^{d(\alpha,\beta)}$ times the isomorphism given in the statement of Theorem \ref{thm: main}, where $d(\alpha,\beta)\in \N$ is the distance between the roots $\alpha$ and $\beta$ in the Dynkin diagram of $G$.  The proof also enables one to describe the morphism 
$\piA_1(S_\alpha) \to \piA_1(G)$ for any root $\alpha$; we leave the details to the reader.
\end{remark}

\subsection{Central extensions of split, semisimple, simply connected groups}
\label{subsection central extensions}
\hfill

In this section, we determine the set of isomorphism classes of central extensions of a split, semisimple, simply connected algebraic group over a field $k$ by strictly $\A^1$-invariant sheaves as a consequence of our main theorem.

Recall that we denote by $Shv_{\rm Nis}(Sm_k) = Shv(Sm_k)$ the category of sheaves of sets on $(Sm_k)_{\rm Nis}$ and similarly by $Ab(k)$ the (abelian) category of sheaves of abelian groups on $(Sm_k)_{\rm Nis}$.  Denote by $\text{$G$-$Ab$}(k)$ the (abelian) category of $G$-sheaves of abelian groups on $(Sm_k)_{\rm Nis}$, that is, sheaves of abelian groups endowed with an action of $G$ by group automorphisms.  We start with a recollection of the following well-known fact, which holds in any topos although we state it only for the big Nisnevich site $(Sm_k)_{\rm Nis}$ of smooth schemes over $k$. 

\begin{lemma} 
\label{lem CExt1}
Let $G$ be a sheaf of groups on $(Sm_k)_{\rm Nis}$.  Let $M\in Ab(k)$ be a sheaf of abelian groups on $(Sm_k)_{\rm Nis}$. Then for any $n\geq 0$, there is a natural isomorphism of abelian groups
\[
H^n_{\rm Nis}(BG;M) \cong \Ext^n_{\text{$G$-$Ab$}(k)} (\Z,M),
\]
where the sheaf $\Z$ is endowed with the trivial action of $G$.
\end{lemma}

\begin{proof} 
Let $D(\text{$G$-$Ab$}(k))$ denote the derived category of the abelian category $\text{$G$-$Ab$}(k)$. The inclusion functor $Ab(k)\to \text{$G$-$Ab$}(k)$ taking a sheaf of abelian groups to itself endowed with the trivial $G$-action is exact and the induced functor $\Theta: D(Ab(k))\to D(\text{$G$-$Ab$}(k))$ admits as left adjoint the functor $LQuot_G: D(\text{$G$-$Ab$}(k)) \to D(Ab(k))$.  Note that $LQuot_G$ is the left derived functor of the quotient functor $\text{$G$-$Ab$}(k) \to Ab(k)$ given by $N\mapsto N/G$.  If $Y = G\times X$ for a sheaf of sets $X$, then clearly
\begin{equation}
\label{eqn LQuot}
 LQuot_G(\Z[Y]) \cong \Z[X].
\end{equation}
Thus, in order to compute $LQuot_G$, one may replace for instance any chain complex of $G$-abelian sheaf by a quasi-isomorphic chain complex, each term of which is a direct sum of sheaves of the form $\Z[G\times X]$ as above, and then applying termwise the above formula \eqref{eqn LQuot}. 

Let $EG$ be the simplicial sheaf of sets constructed in the classical way by setting $EG_n = G^{n+1}$ with $G$ acting diagonally termwise.  Note that $EG$ is simplicially contractible.  The quotient simplicial sheaf of sets $EG/G$ is a model for the classifying space $BG$ of $G$.  Now, let $C_*(BG;\Z)$ be the chain complex (of sheaves of abelian groups) obtained by applying the free sheaf of abelian group functor termwise, and taking the associated chain complex. By definition, we have
\[
H^n(BG;M) = \Hom_{D(Ab(k))}(C_*(BG;\Z),M[n]).
\]
Now we observe that $C_*(BG;\Z) \cong C_*(EG;\Z)/G$.  Since the $n$-th term $E_nG$ of $EG$ is isomorphic as a smooth $k$-scheme with a $G$-action to $G\times G^n$ (with $G^n$ endowed with the trivial $G$-action) for every $n$, it follows that 
\[
C_*(BG;\Z) \cong LQuot_G (C_*(EG;\Z)). 
\]
By the above adjunction, one gets a natural isomorphism
\[
\Hom_{D(Ab(Sm_k))}(LQuot_G (C_*(EG;\Z)),M[n]) \cong \Hom_{D(\text{$G$-$Ab$}(k))}(C_*(EG;\Z),\Theta(M)[n]).
\]
But the group on the right hand side is precisely isomorphic to $\Ext^n_{\text{$G$-$Ab$}(Sm_k)} (\Z,M)$ because $C_*(EG;\Z) \to \Z$ is a quasi-isomorphism of chain complexes of $G$-sheaves of abelian groups.  Thus, we get a natural isomorphism
\[
H^n(BG;M)\cong \Ext^n_{\text{$G$-$Ab$}(k)} (\Z,M),
\]
as claimed.
\end{proof}

Now, let $\CExt(G,M)$ denote the set of isomorphism classes of central extensions of $G$ by $M$ as Nisnevich sheaves of groups on $Sm_k$, which are exact sequences of sheaves of the form
\[
1 \to M \to E \to G \to 1, 
\]
where $M$ is a subsheaf of the center of $E$.  This set forms a group under the operation of Baer sum. The following lemma is also classical:

\begin{lemma} 
\label{lem CExt2}
Let $G$ be a sheaf of groups on $(Sm_k)_{\rm Nis}$.  Let $M\in Ab(k)$ be a sheaf of abelian groups on $(Sm_k)_{\rm Nis}$. Then there is a natural isomorphism of abelian groups
\[
\Ext^2_{\text{$G$-$Ab$}(k)} (\Z,M)\cong \CExt(G,M). 
\]
\end{lemma}

The proof clearly follows by adapting the usual proof in classical group cohomology, see for instance \cite[Theorem 4.1, page 112]{MacLane}. We leave the details to the reader.

For $G$ and $M$ as above, we thus obtain isomorphisms
\begin{equation}
\label{eqn CExt BG}
H^2(BG;M) \cong \Ext^2_{\text{$G$-$Ab$}(k)} (\Z,M)\cong \CExt(G,M)  
\end{equation}
by Lemma \ref{lem CExt1} and Lemma \ref{lem CExt2}.  Consider the \emph{skeletal filtration} $F_{\bullet}G$ on the simplicial classifying space $BG$ of $G$:
\[
\ast = F_0BG \subset F_1BG = \Sigma G \subset \cdots F_{n-1}BG \subset F_nBG \subset \cdots \subset BG.
\]
It is well-known that we get cofiber sequences of the form
\[
F_{n-1}BG \to F_nBG \to \Sigma^n G^{\wedge n},
\]
for every $n$, where we set $\Sigma^i \sX := S^i \wedge \sX$ for any pointed space $\sX$ and any $i \in \N$.  For obvious reasons of connectivity, the restriction morphism induces an isomorphism
\[
H^2_{\rm Nis}(BG;M) = \Hom_{\sH(k)}(BG, K(M,2)) \cong \Hom_{\sH(k)}(F_3BG, K(M,2)).
\]
Moreover as $F_1BG = \Sigma G$ and $H^2(\Sigma G;M) \cong H^1_{\rm Nis} (G;M)$, the filtration $\Sigma G \subset F_2BG \subset F_3BG$ allows one to explicitly understand the canonical morphism
\[
\CExt(G,M) \cong H^2_{\rm Nis}(BG;M) \to H^1_{\rm Nis} (G;M)
\]
taking a central extension of $G$ by $M$ to its underlying $M$-torsor.

The long exact sequence of simplicial homotopy classes of maps into $K(M,2)$ corresponding to the cofiber sequence $F_{2}BG \to F_3BG \to \Sigma^3 G^{\wedge 3}$ has the form
\[
\begin{split}
\Hom_{\sH_s(k)}(\Sigma^3 G^{\wedge 3}, K(M,2)) &\to \Hom_{\sH_s(k)}(F_3BG, K(M,2)) \\ & \to \Hom_{\sH_s(k)}(F_2BG, K(M,2)) \to \Hom_{\sH_s(k)}(\Sigma^2 G^{\wedge 3}, K(M,2)).
\end{split}
\]
But $\Hom_{\sH_s(k)}(\Sigma^3 G^{\wedge 3}, K(M,2))=0$ as the simplicial triple loop space on $K(M,2)$ is contractible. Thus we obtain the exact sequence
\begin{equation}
\label{eqn CExt Mtilde}
0 \to \CExt(G,M) \to \Hom_{\sH_s(k)}(F_2BG, K(M,2)) \to \Hom_{\sH_s(k)}(\Sigma^2 G^{\wedge 3}, K(M,2)) = \tilde{M}(G^{\wedge 3}),
\end{equation}
where for a pointed sheaf of sets $X$, we denote by $\tilde{M}(X)$ the reduced part of $M(X)$ (in other words, the summand of $M(X)$ satisfying $M(X) = M(\ast)\oplus \tilde{M}(X)$, with $\ast$ being the basepoint of $X$).

Now, the long exact sequence corresponding to the cofiber sequence $$\Sigma G = F_{1}BG \to F_2BG \to \Sigma^2 G^{\wedge 2}$$ takes the form
\[
\begin{split}
\tilde{M}(G^{\wedge 2}) \to \Hom_{\sH_s(k)}(F_2BG, K(M,2)) &\to \Hom_{\sH_s(k)}(F_1BG, K(M,2)) = H^1_{\rm Nis}(G, M)\\
& \to \Hom_{\sH_s(k)}(\Sigma G^{\wedge 2}, K(M,2)) = H^1_{\rm Nis}(G \wedge G, M),
\end{split}
\]
where the last morphism is the one induced by the morphism $\pr_1^* + \pr_2^* - \mu^* $ (as this follows from the explicit formula defining $BG$).  

It follows from the above discussion that if $\tilde{M}(G^{\wedge 3}) = \tilde{M}(G^{\wedge 2}) =  0$, then we have an exact sequence
\[
0 \to \CExt(G,M)  \to H^1_{\rm Nis}(G, M) \to  H^1_{\rm Nis}(G \wedge G, M).
\]
If moreover $H^1_{\rm Nis}(G \wedge G, M) = 0$, then the morphism $\CExt(G,M)  \cong H^1_{\rm Nis}(G, M) $ is in fact an isomorphism.

We now assume that $M$ is strictly $\A^1$-invariant.  From the $\A^1$-invariance of $M$, we obtain
\[
M(X) \cong \Hom_{Shv_{\rm Nis}(Sm_k)}(\piA_0(X),M),
\]
for any space $X$.  If $X$ is a pointed space, then $\tilde{M}(X)$ is the group of morphisms of pointed sheaves (of sets) $\Hom_{Shv_{*}(Sm_k)}(\piA_0(X),M)$. If $G$ is a split, semisimple, simply connected algebraic $k$-group, then $G$ is $\A^1$-connected.  Consequently, $G^{\wedge i}$ is $\A^1$-$(i-1)$-connected for any $i \in \N$. But then it follows that 
\[
H^1_{\rm Nis}(G \wedge G, M) = \Hom_{Ab_{\A^1}(k)}(\piA_1(G \wedge G), M) = 0, 
\]
as $G\wedge G$ is $\A^1$-$1$-connected.  We conclude that in that case
\[
 \CExt(G,M)  \to H^1_{\rm Nis}(G, M) 
\]
is an isomorphism.

We summarize the above discussion, which along with our main theorem reproves and generalizes \cite[Theorem 4.7]{Brylinski-Deligne} as follows.

\begin{theorem}
\label{thm CExt}
Let $G$ be a split, semisimple, simply connected algebraic group over $k$ and let $M$ be a strictly $\A^1$-invariant sheaf. Then the canonical morphism 
\[
\CExt(G,M) \xrightarrow{\simeq} H^1_{\rm Nis}(G, M)
\]
is an isomorphism.
\end{theorem}

Furthermore, the group $H^1_{\rm Nis}(G, M) = H^1_{\rm Zar}(G, M)$ can be explicitly computed as in Corollary \ref{corintro:1} in the introduction.

We end this subsection with a comparison of the above approach with Grothendieck's description \cite[VII, no. 1]{SGAVII} of cental extensions of $G$ by $M$ in terms of $M$-torsors over $G$ with an additional property.

\begin{remark} 
Let $G$ be a sheaf of groups and $M$ be a sheaf of abelian groups on $(Sm_k)_{\rm Nis}$.  
\begin{enumerate}
\item Recall from \cite[Section 1]{Brylinski-Deligne} and \cite[VII, no. 1]{SGAVII} that a pair $(E, \phi)$ is said to be a \emph{multiplicative} $M$-torsor on $G$ if $E$ is an $M$-torsor (in the Nisnevich or Zariski) topology, and $\phi$ is an isomorphism 
\[
\phi : \pr_1^* E + \pr_2^*E \xrightarrow{\simeq} \mu^* E 
\]
of $M$-torsors on $G \times G$, where $\pr_i: G \times G \to G$ denotes the $i$th projection, $+$ denotes the sum operation on $M$-torsors and $\mu: G\times G \to G$ denotes the multiplication of $G$, satisfying a compatibility condition involving associativity of $G$; see \cite[Diagram (1.2.2)]{Brylinski-Deligne}.  This compatibility condition can be thought of as a \emph{descent formula} and states that a certain diagram involving $\phi$ and its pullbacks through the various morphisms from $G^3$ to $G^2$ (involving projections and multiplication) is commutative.  Observe that a central extension of $G$ by $M$ defines a multiplicative $M$-torsor on $G$ in an obvious way.  The abelian group of isomorphism classes of multiplicative $M$-torsors on $G$ can be then shown to be isomorphic to $\CExt(G,M)$. 

\item Note that the morphism $\CExt(G,M) \to H^1_{\rm Nis}(G;M)$ taking a central extension (or a multiplicative $M$-torsor) to its underlying $M$-torsor factors through the abelian group 
\[
H^2_{\rm Nis}(F_2BG;M) = \Hom_{\sH_s(k)}(F_2BG, K(M,2)). 
\]
It can be shown that this group can be identified with the abelian group of (conveniently defined) isomorphism classes of pairs $(E, \phi)$ as above which do not necessarily satisfy the descent condition, that is, commutativity of \cite[Diagram (1.2.2)]{Brylinski-Deligne}.  Let us call elements of $\Hom_{\sH_s(k)}(F_2BG, K(M,2))$ \emph{pre-multiplicative $M$-torsors}.  This point of view explains the long exact sequence 
\[
\tilde{M}(G^{\wedge 2}) \to H^2_{\rm Nis}(F_2BG;M) \to H^1_{\rm Nis}(G, M) \to H^1_{\rm Nis}(G \wedge G, M)
\]
appearing in the discussion preceding the statement of Theorem \ref{thm CExt} as follows: the kernel of the morphism $H^1_{\rm Nis}(G, M) \to H^1_{\rm Nis}(G \wedge G, M)$ consists of $M$-torsors for which there exists at least one isomorphism $\phi$ in the definition of multiplicative $M$-torsors on $G$.  On the other hand, the image of $\tilde{M}(G^{\wedge 2})$ in $\Hom_{\sH_s(k)}(F_2BG, K(M,2))$ parametrizes the various $\phi$'s defining a multiplicative structure on a given $M$-torsor. 

\item In the same way, the long exact sequence \eqref{eqn CExt Mtilde}
\[
0 \to \CExt(G,M) = H^2_{\rm Nis}(BG;M) \to H^2_{\rm Nis}(F_2BG;M) \to \tilde{M}(G^{\wedge 3})
\]
explains that a pre-multiplicative $M$-torsor (that is, an element of $H^2_{\rm Nis}(F_2BG;M)$) given by a pair $(E, \phi)$ is a multiplicative torsor if its image in $\tilde{M}(G^{\wedge 3})$ is trivial, which can be shown to be equivalent to the commutatitivity of \cite[Diagram (1.2.2)]{Brylinski-Deligne}.
\end{enumerate}
\end{remark}

Although Theorem \ref{thm CExt} is stated in the context of split, semisimple, simply connected algebraic groups, most of what we discussed in this section holds much more generally and can be formally stated as follows:

\begin{proposition} 
Let $G$ be an $\A^1$-connected sheaf of groups over $k$ and let $M$ be a strictly $\A^1$-invariant sheaf. Then the canonical morphisms 
\[
\CExt(G,M) \to H^2_{\rm Nis}(BG;M)\to H^1_{\rm Nis}(G, M)
\]
are isomorphisms.
\end{proposition}
\begin{proof}
Since $G$ is $\A^1$-connected, we have $H^1_{\rm Nis}(G, M) \cong \Hom_{Ab_{\A^1}(k)}(\piA_1(G), M)$.  It follows from the results of \cite{Morel-book} that $BG$ is $\A^1$-$1$-connected and that $\piA_2(BG) \cong \piA_1(G)$. Moreover, the evaluation on $\piA_2$ induces an isomorphism
\[
H^2_{\rm Nis}(BG;M) = \Hom_{\sH(k)}(BG, K(M,2)) \cong \Hom_{Ab_{\A^1}(k)}(\piA_2(BG),M).
\]
These isomorphisms altogether imply the desired isomorphisms, finishing the proof.
\end{proof}

\subsection{Cellular \texorpdfstring{$\A^1$}{A1}-homology of the flag variety \texorpdfstring{$G/B$}{G/B} in low degrees} \hfill 
\label{subsection G/B}

Let $G$ be a split, semisimple, almost simple, simply connected algebraic group over $k$ of rank $r$.  Fix a maximal $k$-split torus $T$ of $G$ and a Borel subgroup $B$ containing $T$.  In this subsection, we describe the cellular $\A^1$-chain complex for the full flag variety $G/B$, which is much simpler compared to the one for $G$ described above, because of the fact that $G/B$ is cellular in the sense of Definition \ref{definition strict cellular}.  Also note that the canonical quotient morphism $G/T \to G/B$ is an $\A^1$-weak equivalence.

The Bruhat decomposition of $G$ induces a decomposition $G/B = \underset{w \in W}{\cup} B\dot{w}B/B$, which induces a stratification by open subsets
\[
\emptyset = \-\Omega_{-1} \subsetneq \-\Omega_0 \subsetneq \-\Omega_1 \subsetneq \cdots \subsetneq \-\Omega_{\ell} = G, 
\]
where $\ell$ is the length of the longest element $w_0$ in $W$ and for every $i$, one has
\[
\-\Omega_i - \-\Omega_{i-1}= \underset{\ell(w) = \ell-i}{\coprod} B \dot{w}B/B \simeq \underset{\ell(w) = \ell-i}{\coprod} \A^{\ell-i}.  
\]
The cellular $\A^1$-chain complex of $G/B$ is oriented by the choice of a weak pinning of $G$ (Definition \ref{definition weak pinning}) and the choice of a normal reduced expression (Remark \ref{remark normal horizontal}) as a product of simple reflections for every element of the Weyl group $W$ of $G$.  The cellular $\A^1$-chain complex of $G/B$ associated with the Bruhat decomposition has the form (see Remark \ref{rem twist}) 
\[
\Ctcell_i(G/B) = \underset{\ell(w)= \ell - i}{\oplus} \ZA(\G_m^{\wedge i}) \otimes \Z[w] = \underset{\ell(w)= \ell - i}{\oplus} \KMW_i \otimes \Z[w].
\]
Theorems \ref{thm:diff2} and \ref{thm:diff2'} enable us to explicitly describe differentials $\-\partial_i$ in $\Ctcell_*(G/B)$ in degrees $i\leq 2$.  We first observe that for every $i \geq 1$ we have
\[
\Hom_{Ab_{\A^1}(k)}(\KMW_{i+1}, \KMW_{i}) = W(k) \text{ and } \Hom_{Ab_{\A^1}(k)}(\KMW_{1}, \Z) = 0.
\]
This shows that $\-\partial_1 = 0$.  The words $w \in W$ with length $\ell - 1$ are all of the form $w_0s_i$, for $1 \leq i \leq n$ and those with length $\ell - 2$ are all of the form $w_0s_is_j$, for $1 \leq i,j \leq n$.  We further have the relation $w_0s_is_j=w_0s_js_i$ if and only if the roots $\alpha_i$ and $\alpha_j$ are not adjacent in the Dynkin diagram of $G$.  We describe the differential $\-\partial_2$ in the case of $G$ of rank $2$ as a straightforward consequence of Theorems \ref{thm:diff2} and \ref{thm:diff2'}.  

\begin{theorem}
\label{thm:G/B}
Let $G$ be a split, semisimple, almost simple, simply connected algebraic group over $k$ of rank $r$ with a fixed maximal $k$-split torus $T$ and a Borel subgroup $B$ containing $T$.  We have isomorphisms of complexes
\[
\Ccell_*(G/B) \cong \Ccell_*(G/T) \cong \Ccell_*(G)\otimes_{\ZA[T]} \Z.
\]
and the following holds.
\begin{enumerate}[label=$(\alph*)$]
\item For $i,j$ with $i \neq j$, the component of $\-\partial_2$ from $\KMW_2 \otimes \Z[w_0s_is_j]$ to $\KMW_1\otimes \Z[w_0s_i]$ is always $0$.

\item If moreover $\alpha_i$ and $\alpha_j$ are adjacent in the Dynkin diagram of $G$, then the component of $\-\partial_2$ from $\KMW_2\otimes \Z[w_0s_is_j]$ to $\KMW_1\otimes \Z[w_0s_j]$ is given by multiplication by $n_{\epsilon} \cdot \eta \in W(k)$, where $n = n_{ji} = \alpha_j^{\vee}(\alpha_i)$. 
\end{enumerate}
\end{theorem}

\begin{corollary}
With the notation of Theorem \ref{thm:G/B}, we have $\Hcell_0(G/B) = \Hcell_0(G/T) = \Z$ and 
\[
\Hcell_1(G/B) = \Hcell_1(G/T) =
\begin{cases}
\quad  \quad T , &\text{ if $G$ is not of symplectic type;} \\ 
\left(\underset{i=1}{\overset{r-1}{\prod}}~\G_m \right) \times \KMW_1, &\text{ if $G$ is of symplectic type.}
\end{cases}
\]
\end{corollary}

\begin{remark} \label{rmk alternate proof 1}
The cellular $\A^1$-chain complex of the full flag variety $G/B$ of $G$ as in Theorem \ref{thm:G/B} is considerably simpler to handle compared to $\Ccell_*(G)$.  Using an appropriate generalization of the methods of Section \ref{section Bruhat decomposition} and Section \ref{section differential}, one may compute the entire cellular $\A^1$-homology of $G/B$.  The details will appear in a sequel to this paper.
\end{remark}

\begin{remark} \label{rmk alternate proof 2}
Let $G$ be a split, semisimple, simply connected algebraic group over $k$.  There exists an $\A^1$-fibration sequence 
\begin{equation}
\label{eqn fibration}
G \to G/T \xrightarrow{\phi} BT. 
\end{equation}
Recall that $BT$ can be given a cellular structure by choosing a basis of the character group $Y = \Hom(T, \G_m)$ of $T$, which gives an isomorphism $T \cong \underset{i=1}{\overset{r}{\prod}}~\G_m$ and thereby gives $(B\G_m)^r = (\P^{\infty})^r$ as a model of $BT$.  Using the results of Section \ref{subsection homology of Pn} and the K\"unneth formula, one can compute the complete cellular $\A^1$-homology of $BT$.  The morphism $\phi:G/T \to BT$ induces an epimorphism $\Hcell_1(\phi): \Hcell_1(G/T) \to \Hcell_1(BT) = T$, since $G$ is $\A^1$-connected.  It can be easily verified that $\Hcell_1(\phi)$ is an isomorphism if and only if $G$ is not of symplectic type.
\end{remark}

\begin{remark} \label{rmk alternate proof 3}
Let $G$ be a split, semisimple, simply connected algebraic group over $k$.  There exists an exact sequence of the form
\begin{equation}
\label{eqn key exact sequence}
\Hcell_2(G/T)  \xrightarrow{\Hcell_2(\phi)} \Hcell_2(BT) \to \Hcell_1(G) \to \Hcell_1(G/T)\xrightarrow{\Hcell_1(\phi)} T \to 0.  
\end{equation}
A part of the exact sequence \eqref{eqn key exact sequence} can be concluded in an elementary way: the long exact sequence of $\A^1$-homotopy sheaves associated with the $\A^1$-fibration $G \to G/T \stackrel{\phi}{\to} BT$ gives the exact sequence
\[
\piA_1(G) \to \piA_1(G/T)\to T \to 0,
\]
which on abelianization gives the exact sequence of last three terms in \eqref{eqn key exact sequence}.  However, this information is not sufficient for the calculation of $\piA_1(G)\simeq \HA_1(G)$.  The exact sequence \eqref{eqn key exact sequence} can be obtained as the five-term exact sequence associated with the Serre spectral sequence or the motivic Rothenberg-Steenrod spectral sequence corresponding to the $\A^1$-fibration \eqref{eqn fibration}, computing the cellular $\A^1$-homology of $BT$.
\end{remark}

\begin{remark} \label{rmk alternate proof 4}
Let $G$ be a split, semisimple, almost simple, simply connected algebraic group over $k$.  The morphism $\Hcell_2(BT) \to \Hcell_1(G)$ factors through $\Hcell_2(BT)_W$, the group of coinvariants for the action of the Weyl group $W$ of $G$ on $\Hcell_2(BT)$.  It can be shown that there are isomorphisms 
\[                                                                                                                                                        
\Hcell_1(G) \simeq \Hcell_2(BT)_W \simeq \KM_2,
\]
when $G$ is not of symplectic type.  However, when $G$ is of symplectic type, one can show that
\[
\Hcell_2(BT)_W \simeq \KM_2/2 \oplus {}_\eta \KMW_2. 
\]
These computations are achieved by determining the cokernel of the morphism $\Hcell_2(\phi)$ in \eqref{eqn key exact sequence} using, for example, comparison with Suslin homology and results of Demazure \cite{Demazure} on the multiplicative structure of the Chow ring of $G/B$.  A detailed analysis of coinvariants of the Weyl group action on $\Hcell_*(BT)$ and its relationship with $\Hcell_*(BG)$ will be taken up in the sequel to this paper.
\end{remark}

\subsection{Formalization of the method and relation to a theorem of E. Cartan}  
\label{sec:method+Cartan}\hfill 

We begin this section by formalizing the method of our proof of the main theorem (Theorem \ref{thm: main}), so that the method can be applied under suitable realization functors to obtain results pertaining to other homology theories (such as Suslin homology).

We consider a triangulated $\otimes$-category $\sD$ endowed with a $t$-structure compatible with the $\otimes$-structure. As we use the homological point of view, we will denote by $\sD_{\geq 0}$ (respectively, $\sD_{\leq 0}$) the full sub-category of non-negative object (respectively, of non-positive objects) for the $t$-structure. If $\Sigma:\sD\xrightarrow{\simeq}\sD$ is the suspension functor for the triangulated structure, we denote, for any $n\in \Z$,  by $\sD_{\geq n}$ the full subcategory of objects $D$ such that $\Sigma^{\circ -n}(D)$ is non-negative.  We write $H_i^t(D)$ for the homology in degree $i$ of an object $D \in \sD$.  Then for any non-negative object $D\in \sD_{\geq 0}$, there is a exact triangle of the form
\[
 D_{\geq 1} \to D \to H^t_0(D) \to D_{\geq 1}[1]
\]
with $D_{\geq 1}\in \sD_{\geq 1}$ and $H^t_0(D)\in \sA$, the heart of the $t$-structure, which is an abelian category. We refer the reader to \cite{BBD} for the precise definition and basic properties of t-structures on triangulated categories.

\begin{definition}
\label{defn reasonable functor}
We call a functor
\[
 C_*: \sH(k) \to \sD
\]
a {\em reasonable homological chain functor} for smooth $k$-varieties if it satisfies the following properties:
\begin{enumerate}
\item $C_*$ is a symmetric $\otimes$-functor;
\item $C_*(\A^1)$ is the zero object of $\sD$; and
\item if $U$, $V$ are open subschemes of a smooth $k$-scheme $X$ covering $X$, the following obvious Mayer-Vietoris diagram in $\sD$:
\[
C_*(U\cap V) \to C_*(U) \oplus C_*(V) \to C_*(X) \to \Sigma(C_*(U\cap V))
\]
is an exact triangle.
\end{enumerate}
\end{definition}

If $C_*: \sH(k) \to \sD$ is a reasonable homological chain functor, then it is easy to see that $C_*(\P^1) \cong \Sigma C_*(\G_m)$ and $C_*(\A^2-\{0\})\cong \Sigma C_*(\G_m^{\wedge 2})$ are both objects of $\sD_{\geq 1}$.  For $M$ and $N$ in $\sA$, we set $M\otimes_0 N := H^t_0(M\otimes N)$.  We set
\[
\Z^t(1):= H^t_1(\P^1) = \tilde{H}^t_0(\G_m)
\]
and
\[
\Z^t(2) := H^t_1(\A^2-\{0\}) = \tilde{H}^t_0((\G_m)^{\wedge 2}) \cong \Z^t(1)\otimes_0 \Z^t(1).
\]
Furthermore we denote by $\eta^t: \Z^t(2) \to \Z^t(1)$ the $H^t_1$ applied to the Hopf map $\A^2-\{0\}\to \P^1$.

Observe that any unit $u$ of $k$ defines an element $(u) \in \Hom_{\sD}(\Z,\Z^t(1))$, and also defines an automorphism $\<u\>$ of $\Z^t(1)$.  We then define $$\Z^t(n) := \Z^t(1)^{\otimes_0 n},$$ for $n\geq 1$.  Then $\eta^t\otimes \Z^t(1)$ induces a morphism $\Z^t(3) \to \Z^t(2)$. For any pair of units $(u,v)$ in $k$ we have a corresponding symbol $(u)(v)\in \Hom_{\sD}(\Z,\Z^t(2))$ and for any units $\alpha$, $u$ and $v$ in $k$, the induced action of $\<\alpha\>$ on the symbol $(u)(v)$ is given by (see Remark \ref{remark intro K2})
\[
\<\alpha\> \cdot (u)(v) = (u)(v) + \eta  (\alpha) (u)(v).
\]

\begin{theorem}
\label{thm: formalization}
Let $C_*: \sH(k) \to \sD$ be a reasonable homological chain functor for smooth $k$-varieties as above. Then for any split simply connected semi-simple group $G$, $C_*(G)\in \sD_{\geq 1}$.  Moreover, if $G$ is almost simple, then we have the following:
\begin{enumerate}[label=$(\alph*)$]
\item If $G$ is not of symplectic type, then the inclusion $S_\alpha\subset G$ induces a canonical isomorphism
\[
\Z^t(2)/\eta^t \xrightarrow{\simeq} H_1^t(C_*(G))
\]
induced by the morphism $H_1^t(C_*(S_\alpha)) \to H_1^t(C_*(G))$ through the canonical epimorphism $H_1^t(C_*(S_\alpha))\cong \Z^t(2) \to \Z^t(2)/\eta^t$.

\item If $G$ is of symplectic type, then the inclusion $S_\alpha\subset G$ induces a non-canonical isomorphism
\[
\Z^t(2) \xrightarrow{\simeq} H_1^t(C_*(G)).
\]
\end{enumerate}
\end{theorem}
\begin{proof} 
The proof of Theorem \ref{thm: main} can clearly be adapted word-by-word to the more general case of a triangulated $\otimes$-category admitting reasonable homological chain functor $C_*$ for smooth $k$-varieties.  Observe, for instance, that the use of the $\A^1$-homotopy purity theorem used above can always be obtained by an explicit open tubular neighborhood which is isomorphic to a trivial vector bundle (of rank one or two) over the closed subscheme (a Bruhat cell) under consideration (see Lemma \ref{lem:fiberproduct}). The relations on symbols that we use are proven in $\sH(k)$.  Thus, these relations continue to hold on applying the functor $C_*: \sH(k) \to \sD$, so that we get exactly the same relations in the symbols computed in $\sD$. We leave the details to the reader.
\end{proof}

Theorem \ref{thm: formalization} can be used to uniformly explain and reprove some classical results from topology and known results about motives under suitable realization functors, as can be seen below.

\begin{example} \hspace{1cm}
\label{ex formalization}
\begin{enumerate}
\item The main example of a reasonable homological chain functor for smooth $k$-varieties is the $\A^1$-chain functor $C_*^{\A^1}: \sH(k)\to D_{\A^1}(k)$. In that case Theorem \ref{thm: formalization} is nothing but our main theorem (Theorem \ref{thm: main}).

\item If $\sD = DM(k)$ and $C_*: \sH(k)\to DM(k)$ is the functor constructed by Voevodsky associating with every space its motive, we see that $\Z^t(n) = \KM_n$ and $H^t_1 = \HS_1$, the first Suslin homology sheaf. Thus, in this case, Theorem \ref{thm: formalization} precisely becomes \cite[Theorem A]{Gille-Suslin-homology}, first obtained by S. Gille (Corollary \ref{corintro:2}).

\item If $\sD = D(\sA b)$ is the derived category of abelian groups, $k = \R$ and $C_*: \sH(k) \to D(\sA b)$ is the functor $X\mapsto C^{\rm sing}_*(X(\R);\Z)$, induced by the singular chain complex of the topological space of real points with integral coefficients, then $\Z^t(1) = H_1^{\rm sing}(\P^1(\R);\Z) \cong \Z$ and $\eta$ is the multiplication by $2: \Z \to \Z$. Theorem \ref{thm: formalization} in this case proves the following: for any split, semisimple, simply connected algebraic group $G$ over $\R$, the topological space $G(\R)$ is path-connected; if moreover $G$ is almost simple and not of symplectic type, there is a canonical isomorphism
\[
\pi_1(G) = \Z/2\Z;  
\]
and if $G$ is of symplectic type, there is a non canonical isomorphism
\[
\pi_1(G) = \Z.
\]

\item If $\sD = D(\sA b)$ is the derived category of abelian groups, $k = \C$ and $C_*: \sH(k) \to D(\sA b)$ is the functor $X\mapsto C^{\rm sing}_*(X(\C);\Z)$, induced by the singular chain complex of the topological space of complex points with integral coefficients, then $\Z^t(1) = H_1^{\rm sing}(\P^1(\C);\Z) =0$ and $\eta = 0$. Theorem \ref{thm: formalization} in this case proves that for any split, semisimple, simply connected algebraic group $G$ over $\C$, the topological space $G(\C)$ is path-connected and simply connected.
\end{enumerate}
\end{example}

We now explain how our methods used in the proofs of the main results of the paper (especially, cellular structures)   can further be used to give a ``motivic'' proof of the following theorem of E. Cartan \cite[page 496]{Dieudonne} (see \cite{Cartan} and \cite[Theorems B and B']{Bott} for the classical proofs using different methods).

\begin{theorem} 
\label{thm Cartan pi3}
For any split, semisimple, almost simple, simply connected algebraic group $G$ over $\C$, the topological space $G(\C)$ is $2$-connected and there is a canonical isomorphism:
\[
\pi_3(G(\C)) \cong \Z.
\]
\end{theorem}

\begin{proof} 
We will use the stratification of $G(\C)$ induced by the Bruhat decomposition \eqref{eqn bruhat dec} of $G$.  We observe that the inclusion $\Omega_2(\C) \subset G(\C)$ is $4$-connected; this follows from topological purity by observing that the Bruhat cells in the complement of $\Omega_2(\C)$ have complex codimension $\geq 3$ and consequently, real codimension $\geq 6$.  Hence, $H^{\rm sing}_i( \Omega_2(\C);\Z) \cong H^{\rm sing}_i( G(\C);\Z)$ for $i\in\{0,\dots 4\}$.  By Example \ref{ex formalization}(4), $G(\C)$ is path-connected and simply connected.  By the Hurewicz theorem, it suffices to prove that $H^{\rm sing}_2( \Omega_2(\C);\Z)  = 0$ and $H^{\rm sing}_3( \Omega_2(\C);\Z) \cong \Z$.

The principal $B(\C)$-bundle $G(\C) \to G/B(\C)$ along with the fact that the natural map $T(\C)\to B(\C)$ is a homotopy equivalence gives the fibration
\[
T(\C) \Rightarrow G(\C) \to G/B(\C). 
\]
Consider the restriction of this fibration
\[
T(\C) \Rightarrow \Omega_2(\C) \to \overline{\Omega}_2(\C),
\]
where $\overline{\Omega}_2(\C) : = \Omega_2(\C)/B(\C)$.  We use the homological Serre spectral sequence with integral coefficients for the above fibration.

Now, $\overline{\Omega}_2(\C)$ is a cellular space with $2$-cells parameterized by the simple roots or the words of length $\ell-1$ and $4$-cells by the words of length $\ell-2$.  Let $W^{(i)}\subset W$ denote the subset of the Weyl group $W$ of $G$ consisting of the words of length $\ell - i$, for each $i$ (recall that $\ell$ is the length of the longest element of $W$).  For simplicity, we write $H_n$ for singular homology with coefficients in $\Z$.  Then $\overline{\Omega}_2(\C)$ is simply connected and its integral singular homology is given by 
\[
H_0(\overline{\Omega}_2(\C)) = \Z; \quad H_2(\overline{\Omega}_2(\C)) = \underset{w\in W^{(1)}}{\oplus} \Z; \quad H_4(\overline{\Omega}_2(\C)) = \underset{w\in W^{(2)}}{\oplus} \Z;
\]
and $H_n(\overline{\Omega}_2(\C)) = 0$, if $n \notin \{0,2,4\}$.
%
The integral singular homology of $T(\C) = \prod_{i =1}^r \G_m(\C) \cong \prod_{i =1}^r S^1$ is the exterior algebra $\Lambda^* Y$, where $Y$ is the cocharacter group $\Hom( \G_m, T)$, whose elements are in degree $1$. The $E^2_{p,q}$-term of the Serre spectral sequence with integral coefficients is thus:
\[
E^2_{p,q} = H_p(\overline{\Omega}_2(\C);H_q(T(\C))) = H_p(\overline{\Omega}_2(\C))\otimes_\Z H_q(T(\C)) = H_p(\overline{\Omega}_2(\C))\otimes_\Z \Lambda^q Y
\]
and the Serre spectral sequence has the form
\begin{equation}
\label{eq serre spectral seq}
E^2_{p,q} = H_p(\overline{\Omega}_2(\C))\otimes_\Z \Lambda^q Y \Rightarrow H_{p+q}(\Omega_2(\C)).
\end{equation}
It follows that $E^2_{0,q} = \Lambda^q Y$, $E^2_{2,q} = Y \otimes_\Z \Lambda^q Y$ and $E^2_{4,q} = Y^{(2)} \otimes_\Z \Lambda^q Y$ where we set $Y^{(2)}: = \underset{W^{(2)}}{\oplus} \Z$.  We wish to understand the differential $d^2: E^2_{*,*} \to E^2_{*-2,*+1}$ in \eqref{eq serre spectral seq}.

Consider the cofibrations
\[
\Omega_0(\C) \subset \Omega_1(\C) \to Th(\nu_1)(\C)
\]
and 
\[
\Omega_1(\C) \subset \Omega_2(\C) \to Th(\nu_2)(\C)
\]
obtained by taking the complex points of the corresponding $\A^1$-cofibration sequences $\Omega_0 \subset \Omega_1 \to Th(\nu_1)$ and $\Omega_1 \subset \Omega_2 \to Th(\nu_2)$.  The filtration $\Omega_0(\C) \subset \Omega_1(C) \subset \Omega_2(\C)$ is exactly the inverse image of the skeletal filtration of $\overline{\Omega}_2(\C)$ under the ``fibration'' $\Omega_2(\C) \to \overline{\Omega}_2(\C)$ up to reindexing (as there are only even-dimensional cells).  Now $\Omega_0(\C)\cong T(\C)$ and consequently, its homology is $\Lambda_* Y$.  Then $Th(\nu_1) \cong \underset{i=1}{\overset{r}{\vee}} (\A^1/(\A^1-\{0\}))\wedge (T_+)$ so that $Th(\nu_1)(\C) \cong \underset{i=1}{\overset{r}{\vee}} S^2\wedge (T(\C)_+)$ and its reduced singular homology is given by 
\[
H_*^{\rm sing}(Th(\nu_1)(\C)) [-2] = \underset{i=1}{\overset{r}{\oplus}} \Lambda^* Y = \Lambda^* Y \otimes_\Z Y.
\]
This is precisely the column $p=2$ of the Serre spectral sequence \eqref{eq serre spectral seq}.  In the same way, we obtain $Th(\nu_2)(\C) \cong \underset{W^{(2)}}{\vee} S^4\wedge (T(\C)_+)$ so that its reduced singular homology is 
\[
H_*^{\rm sing}(Th(\nu_1)(\C)) [-4] = \underset{W^{(2)}}{\oplus} \Lambda_*(Y) = \Lambda_*(Y)\otimes_\Z Y^{(2)},
\]
which is precisely the column $p=4$ of the Serre spectral sequence \eqref{eq serre spectral seq}.

The differential $d^2: E^2_{*,*} \to E^2_{*-2,*+1}$ is then analyzed as follows. In the cofibration $\Omega_0 \subset \Omega_1 \to Th(\nu_1)$, each Bruhat cell $Bw_iB=Y_i$ of codimension $1$ admits (using our conventions from Section \ref{subsection differential in degree 1}) an explicit tubular neighborhood in $\Omega_1$ isomorphic to $\A^1\times Y_i$, so that we get a homotopy cocartesian square of spaces of the form:
\[
\begin{xymatrix}{
\underset{i=1}{\overset{r}{\coprod}} (\A^1-\{0\})\times Y_i \ar[r] \ar[d] & \underset{i=1}{\overset{r}{\coprod}} \A^1\times Y_i \ar[d] \\
\Omega_0 \ar[r] & \Omega_1 .} 
\end{xymatrix}
\]
Taking complex points and using the homotopy equivalence between $Y_i(\C)$ and $T(\C)$, the restriction of the connecting homomorphism $$H_{p+2}(Th(\nu_1)(\C)) \to H_{p+1}(\Omega_0(\C))$$ to the $i$th direct summand in the decomposition $H_{p+2}(Th(\nu_1)(\C)) \cong \underset{i=1}{\overset{r}{\oplus}} H_{p+2}(S^2\wedge (T(\C)_+))$ is given by the composition
\[
\Lambda^p Y = H_p(Y_i(\C)) \subset H_{p+1}((\C^1-\{0\})\times Y_i(\C)) \to H_{p+1}(\Omega_0(\C)) = H_{p+1}(T(\C)) = \Lambda^{p+1} Y.
\]
Moreover, since the above morphism and and the homotopy equivalence $Y_i(\C) \to T(\C)$ are both left $T(\C)$-invariant, evaluating on the homology turns it into a morphism of $H_*(T(\C)) = \Lambda^* Y$-modules. 

The differential $d^2: E^2_{2,0} \to E^2_{0,1}$ is the identity map of $Y$; this follows from the fact that for each $i$, the morphism $\A^1-\{0\}\to (\A^1-\{0\})\times (Y_i) \to \Omega_0 \cong T$ is given by the $i$th coroot $\alpha_i^\vee$.  More generally, it follows that $d^2: E^2_{2,*} \to E^2_{0,*+1}$ is given by the product (in $\Lambda^* Y$) map
\[
\Lambda^* Y \otimes_\Z Y \to \Lambda^* Y.
\]
Since this is a surjective morphism, $H_p(\Omega_2(\C)) = 0$ for $p\leq 2$. 
Analyzing the spectral sequence further, we see that $\pi_3(G(\C)) = H_3(G(\C)) = H_3(\Omega_2(\C))$ is the homology of the complex
\begin{equation}
\label{eqn dual diagram BD}
Y^{(2)} \to Y\otimes_\Z Y \twoheadrightarrow \Lambda_2(Y)
\end{equation}
at $Y\otimes_\Z Y$, where the left morphism is given by the differential
\[
d^2: Y^{(2)} = E^2_{4,0} \to E^2_{2,1} = Y\otimes_\Z Y 
\]
of the Serre spectral sequence \eqref{eq serre spectral seq}.  We claim that the restriction of the differential $d^2$ to the factor $\Z$ corresponding to $w_{ij} = w_0s_is_j$ in $Y^{(2)} = \underset{W^{(2)}}{\oplus} \Z$ is given by 
\begin{equation}
\label{eq: explicit diff serre}
\Z \to Y\otimes_\Z Y; \quad 1\mapsto s_j(\alpha_i^\vee)\otimes \alpha^\vee_j - \alpha_j^\vee \otimes \alpha^\vee_i \in \underset{W^{(1)}}{\oplus} Y = Y \otimes_\Z (\underset{W^{(1)}}{\oplus} \Z ) = Y\otimes_\Z Y.
\end{equation}
This follows from the proofs of Theorem \ref{thm:diff2} and  Theorem \ref{thm:diff2'}.  Recall that in the proof of Theorem \ref{thm:diff2}, we used for each $w_{ij}$ an explicit open neighborhood of $Z_{ij}$ in $\tilde{X}_{ij}$ given by the total space of a trivial rank $2$-vector bundle embedded as an open subset in $\tilde{X}_{ij}$ and consequently, an explicit model $(\A^2-\{0\})\times Z_{ij} \to \Omega_1$ for the connecting morphism appearing in the differential of the cellular $\A^1$-chain complex. Proceeding as above and using our computations, we may identify for each $w_{ij}$ the morphism given by the composite
\[
\begin{split}
\Z & = H_3((\C^2-\{0\})\times Z_{ij}(\C))  \\ 
& \cong H_3(S^3\times T(\C)) \to H_3(\Omega_1(\C)) \to H_3(\underset{i=1}{\overset{r}{\vee}} S^2 \wedge (T(\C)_+)) \cong H_3(\Omega_1/\Omega_0(\C)) = Y\otimes_\Z Y 
\end{split}
\]
with the morphism \eqref{eq: explicit diff serre} above. Observe that because of the use of complex points and integral singular homology, the issue of change of orientations showing up in the proofs of Theorem \ref{thm:diff2} and  Theorem \ref{thm:diff2'} disappears, as any complex vector bundle is canonically oriented.

Now to conclude, we just observe that the dual of the diagram \eqref{eqn dual diagram BD} appears in the spectral sequence of the filtration computing $H^*(G;\KM_2)$ in degrees $\leq 2$, that is, computing $H^*(\Omega_2;\KM_2)$ with respect to the filtration $\Omega_0 \subset \Omega_1 \subset \Omega_2$.  Since $H^2(G;\KM_2) = H^2(\Omega_2;\KM_2) = \Z$ canonically, we conclude that the complex dual to \eqref{eqn dual diagram BD} has homology in the middle equal to $\Z$ (as it is a diagram of finite type free $\Z$-modules).  This gives us the isomorphism
\[
\pi_3(G(\C)) = H_3(G(\C)) = H_3(\Omega_2(\C)) =\Z.  
\]
\end{proof}

\begin{remark}
The \eqref{eqn dual diagram BD} appearing in the proof of Theorem \ref{thm Cartan pi3} is precisely dual to \cite[(4.4.3), page 38]{Brylinski-Deligne}.  In fact the two situations are closely related for the following reason. There is a canonical natural transformation of functors on smooth $k$-schemes
\[
H^p(X;\Z(q)) \to H^{p}(X(\C);\Z),
\]
where $H^p(X;\Z(q))$ are the integral motivic cohomology groups of Voevodsky \cite{Voevodsky-book}.  This natural transformation induces a morphism of the cohomological spectral sequences in integral motivic cohomology for $\Omega_2$ (or $G$) induced by the above flag to the cohomological spectral sequence in integral cohomology for $\Omega_2(\C)$ (or $G(\C)$) induced by the complex points of the same flag. The latter is as we observe above the Serre spectral sequence (after conveniently reindexing). Now, for a smooth $k$-scheme $X$, one has
\[
 H^p(X;\Z(q)) \cong H^{p-q}(X;\KM_q),
\]
for $p\geq 2q-2$ \cite[Corollary 2.4]{Voevodsky-Milnor}. Hence, the cohomological spectral sequence in integral motivic cohomology for $\Omega_2$ coincides with the same for $H^{p}(X;\KM_q)$ in the area we are concerned with (after a suitable reindexing). The above comparison morphism from integral motivic cohomology to integral singular cohomology of the complex points thus identifies our sequence \eqref{eqn dual diagram BD} to the dual of \cite[(4.4.3), page 38]{Brylinski-Deligne}.  It follows that in fact we have proved the following motivic variant of Cartan's theorem on $\pi_3$:

\begin{corollary} 
For any split, semisimple, almost simple, simply connected algebraic group $G$ over $\C$, $G(\C)$ is $2$-connected and the canonical morphism
\[
 \Z = H^1(G;\KM_2) \cong H^3(G;\Z(2)) \to H^{3}(G(\C);\Z)
\]
is an isomorphism.
\end{corollary}
\end{remark}

\begin{remark} 
We can also easily conclude that the morphism $Y^{(2)} = \oplus_{W^{(2)}} \Z \to Y\otimes Y$ in \eqref{eqn dual diagram BD} is a monomorphism.  It is proven in \cite[Proof of Proposition 4.6]{Brylinski-Deligne} that the morphism dual to the above morphism is surjective. In our proof of Theorem \ref{thm Cartan pi3} above, we showed that the homology in the middle of \eqref{eqn dual diagram BD} is free of rank one. As each group in \eqref{eqn dual diagram BD} is free of finite type, the injectivity of $Y^{(2)} \to Y\otimes Y$ follows by observing that the rank of $Y^{(2)}$ is exactly one less than the rank of the kernel of $Y\otimes_\Z Y \twoheadrightarrow \Lambda_2(Y)$, which is the rank of the second symmetric power $S^2(X)$ on $X = Y^*$. This is easy to verify and it implies that $H^2(G;\KM_2) = 0$. From Corollary \ref{corollary cellular cohomology vanishing} we know that $H^i(G;\KM_2) = 0$ for $i> 2$. This gives another proof of \cite[Proposition 4.6]{Brylinski-Deligne}, which we restate here for the convenience of the reader.

\begin{corollary}
\label{cor cohomology}
For any split, semisimple, almost simple, simply connected algebraic group $G$ over $k$, $H^0(G;\KM_2) = K^M_2(k)$, $H^1(G;\KM_2) = \Z$ and $H^i(G;\KM_2) = 0$ for $i\geq 2$.
\end{corollary}
\end{remark}

\end{document}